\documentclass[12pt]{article}
\usepackage{filecontents}
\usepackage[a4paper, margin = 1in]{geometry}

\usepackage[verbose]{backref}

\RequirePackage[hyphens]{url}
\usepackage{mathtools}

\usepackage{amsmath}
\usepackage{amssymb}
\usepackage{latexsym}

\newcommand{\isopil}{\stackrel{\raisebox{0.1ex}[0ex][0ex]{\(\sim\)}}%
			{\raisebox{-0.15ex}[0.28ex]{\(\rightarrow\)}}}
\newcommand{\isleftadjointto}{\dashv}

\newcommand{\TSB}[2]{\operatorname{Bar}_{#1}({#2})}

\usepackage{eucal}

\DeclareMathAlphabet{\mathbbe}{U}{bbold}{m}{n}
\newcommand{\simplexcategory}{\mathbbe{\Delta}}
\newcommand{\OMEGA}{\mathbbe{\Omega}}
\newcommand{\fatnerve}{\mathbf{N}}

\renewcommand{\simplexcategory}{\boldsymbol{\Delta}}
\renewcommand{\OMEGA}{\boldsymbol{\Omega}}

\newcommand{\op}{^{\text{{\rm{op}}}}}

\newcommand{\dstrees}{H}

\RequirePackage[raiselinks=false,colorlinks=true,citecolor=blue,urlcolor=blue,linkcolor=blue,bookmarksopen=true,pagebackref]{hyperref}
\newcommand{\arxiv}[1]{\href{http://arxiv.org/pdf/#1}{arXiv:#1}}

\usepackage{tikz}
\usetikzlibrary{matrix,arrows}
\usetikzlibrary{decorations.pathmorphing}

\tikzset{
  dots/.style args={#1per #2}{
    line cap=round,
    dash pattern=on 0 off #2/#1
  }
}

\tikzset{
  rot/.style={shift={(-4.5pt,0pt)}, rotate=-45}
}

\tikzset{
  act /.tip = >|
}

\tikzset{
  wavy/.style={
	/tikz/decorate, /tikz/decoration={
	  snake, segment length=3.5,
	  amplitude=.7,
	  post=lineto,
	  post length=2pt
	}
  }
}

\tikzset{
    onto/.style={/tikz/commutative diagrams/twoheadrightarrow}
}
\tikzset{
    into/.style={/tikz/commutative diagrams/rightarrowtail}
}

\tikzset{
  onedot/.pic={
	\fill (0,0) circle[radius=0.065];
  }
}

\tikzset{
  /tikz/commutative diagrams/on top/.style={inner sep=1pt, description}
}

\usepackage{tikz-cd}
\newcommand{\drpullback}{\arrow[phantom]{dr}[very near start,description]{\lrcorner}}
\newcommand{\dlpullback}{\arrow[phantom]{dl}[very near start,description]{\llcorner}}
\newcommand{\urpullback}{\arrow[phantom]{ur}[very near start,description]{\urcorner}}
\newcommand{\ulpullback}{\arrow[phantom]{ul}[very near start,description]{\ulcorner}}
\newcommand{\drpullbackS}{\arrow[phantom]{ddr}[very near start,description]{\lrcorner}}

\newcommand{\dlactinert}{\arrow[phantom]{dl}[very near start,description]{\urcorner}}
\newcommand{\ulactinert}{\arrow[phantom]{ul}[very near start,description]{\lrcorner}}

\newcommand{\dpullback}{\arrow[phantom]{dd}[near start,description, "$\mathclap{\lrcorner}$" rot]{}}

\newcommand{\PPP}{\mathsf{P}}
\newcommand{\QQQ}{\mathsf{Q}}
\newcommand{\RRR}{\mathsf{R}}
\newcommand{\MMM}{\mathsf{M}}
\newcommand{\SSS}{\mathsf{S}}
\newcommand{\TTT}{\mathsf{T}}

\newcommand{\into}{\rightarrowtail}

\newcommand{\bigleftbrace}[1]{\left\{\raisebox{0pt}[#1pt]{}\right.}
\newcommand{\bigrightbrace}[1]{\left.\raisebox{0pt}[#1pt]{}\right\}}

\newcommand{\inlineDotlessTree}{%
\raisebox{-4pt}{ \begin{tikzpicture}
  \draw (0.0, 0.0) -- (0.0, 0.5);
\end{tikzpicture} }}

\newcommand{\inlineonetree}{%
\raisebox{-4pt}{ \begin{tikzpicture}
  \draw (0.0, 0.0) -- (0.0, 0.5);
  \fill (0.0, 0.25) circle[radius=0.06];
\end{tikzpicture} }}

\newcommand{\inlinecorolla}{%
\raisebox{-4pt}{ \begin{tikzpicture}
\coordinate (dot) at (0.0, 0.0);
\draw (0.0, -0.24) -- (dot);
\fill (dot) circle[radius=0.05];
\draw (dot) -- +(-0.18, 0.19);
\draw (dot) -- +(-0.07, 0.26);
\draw (dot) -- +(0.07, 0.26);
\draw (dot) -- +(0.18, 0.19);
\end{tikzpicture} }}
	
\newcommand{\inlineblobtriv}{%
\raisebox{-4pt}{ \begin{tikzpicture}
  \draw (0.0, 0.0) -- (0.0, 0.5);
  \draw (0.0, 0.25) circle[radius=0.1];
\end{tikzpicture} }}

\newcommand{\ctreeone}{
\begin{tikzpicture}
  \fill (0.0, 0.0) circle[radius=0.05];
\end{tikzpicture}
}

\newcommand{\ctreetwo}{
\begin{tikzpicture}
  \fill (0.0, 0.0) circle[radius=0.05];
  \draw (0.0, 0.0) -- (0.0, 0.245);
  \fill (0.0, 0.245) circle[radius=0.05];
\end{tikzpicture}
}

\newcommand{\ctreethreeV}{
\begin{tikzpicture}
  \fill (0.0, 0.0) circle[radius=0.05];
  \draw (0.0, 0.0) -- (-0.098, 0.245);
  \fill (-0.098, 0.245) circle[radius=0.05];
  \draw (0.0, 0.0) -- (0.098, 0.245);
  \fill (0.098, 0.245) circle[radius=0.05];
\end{tikzpicture}
}

\newcommand{\ctreethreeL}{
\begin{tikzpicture}
  \fill (0.0, 0.0) circle[radius=0.05];
  \draw (0.0, 0.0) -- (0.0, 0.245);
  \fill (0.0, 0.245) circle[radius=0.05];
  \draw (0.0, 0.245) -- (0.0, 0.49);
  \fill (0.0, 0.49) circle[radius=0.05];
\end{tikzpicture}
}

\usepackage{ stmaryrd }
\newcommand{\actto}{\rightarrow\Mapsfromchar}

\DeclareRobustCommand\comma{%
  \mathchoice%
    {\kern1.2pt\raise0pt\hbox{$\displaystyle\downarrow$}\kern1.2pt}
    {\kern1pt\raise0pt\hbox{$\textstyle\downarrow$}\kern1pt}
    {\kern0.4pt\raise0pt\hbox{$\scriptstyle\downarrow$}\kern0.4pt}
    {\kern0.2pt\raise0pt\hbox{$\scriptscriptstyle\downarrow$}\kern0.2pt}
}%

\DeclareRobustCommand\uppercirc{%
  \mathchoice%
    {\kern0pt\raise0.55ex\hbox{$\displaystyle\circ$}\kern0.8pt}
    {\kern0pt\raise0.58ex\hbox{$\textstyle\circ$}\kern0.8pt}
    {\kern0pt\raise0.45ex\hbox{$\scriptstyle\circ$}\kern0.4pt}
    {\kern0pt\raise0.4ex\hbox{$\scriptscriptstyle\circ$}\kern0.2pt}
}%

\newcommand{\bd}{\uppercirc}

\DeclareRobustCommand\upperodot{%
  \mathchoice%
    {\kern0pt\raise0.72ex\hbox{\scriptsize $\displaystyle\odot$}\kern0.8pt}
    {\kern0pt\raise0.72ex\hbox{\scriptsize $\textstyle\odot$}\kern0.8pt}
    {\kern0pt\raise0.45ex\hbox{\tiny $\scriptscriptstyle\odot$}\kern0.4pt}
    {\kern0pt\raise0.4ex\hbox{\normalsize $\scriptscriptstyle\odot$}\kern0.2pt}
}%

\newcommand{\bdr}{\upperodot}

\DeclareRobustCommand\upperstar{%
  \mathchoice%
    {\kern0pt\raise0.55ex\hbox{$\displaystyle *$}\kern0.8pt}
    {\kern0pt\raise0.58ex\hbox{$\textstyle *$}\kern0.8pt}
    {\kern0pt\raise0.45ex\hbox{$\scriptstyle *$}\kern0.4pt}
    {\kern0pt\raise0.4ex\hbox{$\scriptscriptstyle *$}\kern0.2pt}
}%
\DeclareRobustCommand\lowerstar{%
  \mathchoice%
    {\kern0pt\raise-0.65ex\hbox{$\displaystyle *$}\kern0.8pt}
    {\kern0pt\raise-0.68ex\hbox{$\textstyle *$}\kern0.8pt}
    {\kern0pt\raise-0.55ex\hbox{$\scriptstyle *$}\kern0.4pt}
    {\kern0pt\raise-0.5ex\hbox{$\scriptscriptstyle *$}\kern0.2pt}
}%

\newcommand{\lowershriek}{_!}

\newcommand{\name}[1]{\ulcorner #1\urcorner}
\newcommand{\coname}[1]{\raisebox{-1.5pt}{$\llcorner$}\kern0.8pt#1\kern0.8pt\raisebox{-1.5pt}{$\lrcorner$}}

\newcommand{\tr}{\operatorname{tr}}

\newcommand{\Aut}{\operatorname{Aut}}

\newcommand{\id}{\operatorname{id}}
\newcommand{\Id}{\mathsf{Id}}

\newcommand{\Decbot}[1]{\operatorname{Dec}_\bot{}\kern-2pt{#1}}
\newcommand{\Dectop}[1]{\operatorname{Dec}_\top{}\kern-2pt{#1}}

\providecommand{\kat}[1]{\text{\textbf{\textsl{#1}}}}

\newcommand{\Set}{\kat{Set}}

\newcommand{\Grpd}{\kat{Grpd}}

\newcommand{\Q}{\mathbb{Q}}

\renewcommand{\epsilon}{\varepsilon}

\newcommand{\CC}{\mathcal{C}}

\providecommand{\norm}[1]{\left| {#1}\right|}

\newcommand{\tensor}{\otimes}

\newcommand{\N}{\mathbb{N}}
\newcommand{\B}{\mathbb{B}}

\newcommand{\FF}{\mathcal{F}}
\newcommand{\MM}{\mathcal{M}}

\usepackage{amsthm}

\newtheorem{lemma}{Lemma}[subsection]
\newtheorem{prop}[lemma]{Proposition}

\newtheorem{theorem}[lemma]{Theorem}

\newtheorem{taller}[lemma]{$\!\!$}

\newenvironment{blanko}[1]%
{\begin{taller}{\normalfont\bfseries  #1}\normalfont}%
{\end{taller}}
\usepackage{graphicx}

\begin{document}

\title{The incidence comodule bialgebra of the \\[4pt]
Baez--Dolan construction}
\author{Joachim Kock\\
\footnotesize Universitat Aut\`onoma de Barcelona\\[-3pt]
\footnotesize and Centre de Recerca Matem\`atica, Barcelona\\
\footnotesize kock@mat.uab.cat}
\date{}

\maketitle

\begin{abstract}
  Starting from any operad $\PPP$, one can consider on one hand the free
  operad on $\PPP$, and on the other hand the Baez--Dolan construction on
  $\PPP$. These two new operads have the same space of operations, but with
  very different notions of arity and substitution. The main result of this
  paper is that the incidence bialgebras of the two-sided bar constructions
  of the two operads constitute together a comodule bialgebra.
  The result is objective: it concerns comodule-bialgebra structures 
  on groupoid slices, and the proof is given in terms of 
  equivalences of groupoids and homotopy pullbacks.
  Comodule bialgebras in the usual sense are obtained
  by taking homotopy cardinality.
  The simplest instances of the construction cover several comodule
  bialgebras of current interest in analysis. If $\PPP$ is the identity
  monad, then the result is the Fa\`a di Bruno comodule bialgebra (dual to
  multiplication and substitution of power series). If $\PPP$ is any monoid
  $\Omega$ (considered as a one-coloured operad with only unary
  operations), the resulting comodule bialgebra is the dual of the
  near-semiring of $\Omega$-moulds under product and composition, as
  employed in \'Ecalle's theory of resurgent functions in local dynamical systems.
  If $\PPP$ is the terminal operad, then the result is essentially the
  Calaque--Ebrahimi-Fard--Manchon comodule bialgebra of rooted trees,
  dual to composition and substitution of B-series in numerical analysis
  (Chartier--Hairer--Vilmart). 
  The full generality is of interest in category theory. As it holds for 
  {\em any} operad, the result is actually about the Baez--Dolan 
  construction itself, providing it with a new algebraic perspective.
\end{abstract}


 
\newpage

  \tableofcontents

  \newpage
  
\setcounter{section}{-1}

\section{Introduction}
  
The main result of this paper is the construction of a comodule bialgebra
from the Baez--Dolan construction on any operad. Before explaining this
result, we outline the motivation behind it,
recall the Baez--Dolan construction itself, and say a word about objective
combinatorics and homotopy linear algebra, the setting in which the results
are staged.

\subsection{Background and motivation}

\begin{blanko}{Contexts of comodule bialgebras.}
  For $B$ a commutative bialgebra, a comodule bialgebra over $B$ is a
  bialgebra object in the braided monoidal category of (left) $B$-comodules
  \cite{Abe}, \cite{Manchon:Abelsymposium} (see \ref{sub:comodulebialgebras} for some 
  recollections). Although comodule bialgebras go
  back to the genesis of Hopf algebra theory in algebraic topology --- 
  the homology of any $H$-space is a comodule bialgebra over Milnor's dual 
  Steenrod bialgebra~\cite{Milnor:1958}, 
  \cite{Switzer} 
  --- and
  also have a certain history in quantum algebra~\cite{Molnar}, the recent
  burst in interest in them essentially comes from analysis.

  --- {\em B-series} \cite{Hairer-Lubich-Wanner}. First it was discovered
  in numerical analysis by Chartier, E.~Hairer, and
  Vilmart~\cite{Chartier-Hairer-Vilmart:FCM2010} that Butcher series admit
  a second important operation after composition, namely substitution, and
  that the two operations interact in an interesting way. The interaction
  was recast in the language of combinatorial Hopf algebras by Calaque,
  Ebrahimi-Fard, and Manchon~\cite{Calaque-EbrahimiFard-Manchon:0806.2238},
  who made explicit the comodule-bialgebra structure.

  --- {\em Mould calculus} \cite{Cresson}. Second, it was realised that the
  same structure plays an important role in mould calculus, \'Ecalle's
  extensive theory for dealing with `functions in a variable number of
  variables' in local dynamical systems~\cite{Ecalle:I} (see also
  \'Ecalle--Vallet~\cite{Ecalle-Vallet} and
  Ebrahimi-Fard--Fauvet--Manchon~\cite{EbrahimiFard-Fauvet-Manchon:1609.03549}).

  The examples from these two application areas will be recovered as special
  cases of the {\em incidence comodule bialgebra of the Baez--Dolan construction}.

  --- {\em Regularity structures} \cite{Hairer:intro}. More recently,
  comodule bialgebras have surfaced in the algebraic renormalisation of
  regularity structures~\cite{Hairer:regularity}, a powerful framework for
  analysing singular stochastic PDEs. While so-called structure Hopf
  algebras (shuffle Hopf algebras and Butcher--Connes--Kreimer-like Hopf
  algebras) have played an important role for some time (see for example
  \cite{Chen:54}, \cite{Lyons:roughsignals}, \cite{Gubinelli:0610300},
  \cite{Hairer:regularity}), a second bialgebra
  structure, the {\em renormalisation bialgebra}, was introduced recently
  in a landmark paper by Bruned, M.~Hairer, and
  Zambotti~\cite{Bruned-Hairer-Zambotti:1610.08468}. A key point in their
  theory is that the structure Hopf algebra is a comodule bialgebra over
  the renormalisation bialgebra, in close analogy with the
  Calaque--Ebrahimi-Fard--Manchon situation. (See
  Manchon~\cite{Manchon:cours2016} and Foissy~\cite{Foissy:1811.07572} for
  further algebraic treatment.)

  --- {\em Moment-cumulant relations} \cite{Nica-Speicher:Lectures}.
  Finally, in Voiculescu's theory of free
  probability~\cite{Voiculescu:lectures},
  Speicher~\cite{Speicher:multiplicative}, \cite{Nica-Speicher:Lectures}
  had discovered a beautiful moment-cumulant formula in terms of M\"obius
  inversion in the incidence algebra of the noncrossing partitions lattice
  (analogous to Rota's classical moment-cumulant relations in terms of the
  ordinary partition lattice). Ebrahimi-Fard and Patras
  \cite{EbrahimiFard-Patras:1409.5664},~\cite{EbrahimiFard-Patras:1502.02748},
  inspired by methods of quantum field theory, gave a very different
  approach to the same formula, in terms of a time-ordered exponential
  coming from a half-shuffle in the double tensor algebra. Recently,
  Ebrahimi-Fard, Foissy, Kock, and
  Patras~\cite{EbrahimiFard-Foissy-Kock-Patras:1907.01190} uncovered the
  relationship between the two constructions in terms of a
  comodule-bialgebra structure on noncrossing partitions.

  \newpage
  
  The topic is now under scrutiny in combinatorics; for a survey, see
  Manchon~\cite{Manchon:Abelsymposium}. Fauvet, Foissy, and
  Manchon~\cite{Fauvet-Foissy-Manchon:1503.03820}, motivated by mould
  calculus, gave a beautiful construction of comodule bialgebras from
  finite topological spaces, building on
  \cite{Foissy-Malvenuto-Patras:1403.7488}, with important connections to
  quasi-symmetric functions. Comodule bialgebras play a prominent role in
  Foissy's treatise~\cite{Foissy:1702.05344} where it is shown how comodule
  bialgebras are induced by certain actions of the brace operad. In a
  different line of development, Carlier~\cite{Carlier:1903.07964} showed
  that every hereditary species (in the sense of
  Schmitt~\cite{Schmitt:hacs}) induces a comodule bialgebra, and found a
  link with the Batanin--Markl operadic
  categories~\cite{Batanin-Markl:1404.3886}. 
  
  The present contribution provides a rather general construction of
  comodule bialgebras, from the Baez--Dolan construction on any operad.
\end{blanko}

\begin{blanko}{The Baez--Dolan construction.}
  In their seminal 1998 paper {\em Higher-dimensional algebra III:
  $n$-categories and the algebra of opetopes}~\cite{Baez-Dolan:9702}, Baez
  and Dolan introduced the {\em opetopes}, a new family of shapes for
  defining weak higher categories in a uniform way, overcoming coherence
  problems by exploiting universal properties. The opetopes are defined
  inductively by a remarkable construction on operads $\PPP\mapsto\PPP\bd$,
  now called the Baez--Dolan construction (reviewed below in
  Subsection~\ref{sub:BD}). The colours of $\PPP\bd$ are the operations of
  $\PPP$, and the algebras for $\PPP\bd$ are the operads over $\PPP$.
  The construction leads to a notable
  combinatorial richness: starting from nothing but the identity monad
  $\Id$ on $\Set$, it produces first the natural numbers as the operations
  of $\Id{}\bd$, then the planar rooted trees as operations of
  $(\Id{}\bd)\bd$, and after that certain trees of trees, and trees of
  trees of trees\ldots The $(n+1)$-dimensional opetopes are by definition
  the operations of the $n$th iterate of this construction.
  
  The operations-to-colours shift can be described combinatorially in terms
  of trees as a shift from {\em grafting onto leaves} to {\em substituting
  into nodes}. This is the aspect of the Baez--Dolan construction explored
  in the present paper.
  
  The Baez--Dolan construction has interesting connections with logic. It
  was recast in terms of a {\em function replacement} by Hermida, Makkai,
  and Power~\cite{Hermida-Makkai-Power:I},~\cite{Hermida-Makkai-Power:II},
  motivated by first-order logic with dependent sorts.
  Cheng~\cite{Cheng:0304277}, \cite{Cheng:0304279} made important
  contributions establishing relationships between the different
  variants of the construction; see also Leinster~\cite{Leinster:0305049}.
  Kock, Joyal, Batanin, and Mascari~\cite{Kock-Joyal-Batanin-Mascari:0706}
  exploited the formalism of polynomial functors~\cite{Gambino-Kock:0906.4931}
  to give a purely
  combinatorial description of the Baez--Dolan construction, which turned
  out to have a homological interpretation~\cite{Steiner:1204.6723}, and
  more recently found further applications to logic and computer science
  \cite{Finster:opetopic!}, \cite{Finster:HoTTest},
  \cite{Curien-HoThanh-Mimram:1903.05848},
  \cite{HoThanh-Subramaniam:1911.00907}. The polynomial approach to operads
  has proven useful in homotopy theory. Batanin and
  Berger~\cite{Batanin-Berger:1305.0086} used it to construct left-proper
  Quillen model structures for algebras for a large class of operads called
  {\em tame} polynomial monads. They show that for any polynomial monad,
  the Baez--Dolan construction is tame. Recently, Batanin and 
  Markl~\cite{Batanin-Markl:1812.02935} have ported the Baez--Dolan 
  construction to the theory of 
  operadic categories~\cite{Batanin-Markl:1404.3886}, for the purpose of
  developing Koszul duality theory.

  The present contribution exploits the polynomial formalism to add a
  new algebraic perspective to the landscape of the Baez--Dolan
  construction, by studying it in the context of objective algebraic
  combinatorics.
\end{blanko}
  
\begin{blanko}{Objective combinatorics and homotopy linear algebra.}
  The idea of {\em objective combinatorics} --- the term is due to Lawvere
  --- is to work with the algebra of sets instead of numbers, in order to
  obtain native bijective proofs, and reveal and exploit functorialities
  and universal properties that cannot exist at the numerical level.
  Joyal's theory of species~\cite{Joyal:1981} is the starting point for the
  theory. For the sake of dealing seamlessly with symmetries of objects, it
  is often fruitful to upgrade further from sets to
  groupoids~\cite{Baez-Dolan:finset-feynman},
  \cite{Galvez-Kock-Tonks:1207.6404}, dealing directly with groupoids of
  combinatorial objects rather than with their sets of iso-classes. The
  numerical level is then recovered by taking homotopy cardinality.
  
  {\em Homotopy linear algebra} \cite{Baez-Hoffnung-Walker:0908.4305},
  \cite{Galvez-Kock-Tonks:1602.05082} serves as a general tool in this
  context by systematically replacing vector spaces by slice categories,
  and linear maps by linear functors; see 
  Subsections~\ref{sub:Set/I}--\ref{sub:groupoids} for a brief 
  review. If $I$ is a groupoid of combinatorial
  objects, classical algebraic combinatorics starts with the vector space
  $\Q_{\pi_0 I}$ spanned by iso-classes of objects; the objective approach
  considers instead the slice category $\Grpd_{/I}$. Linear maps are 
  replaced by spans of groupoids, and matrix multiplication by 
  certain homotopy pullbacks.
  
  At the objective level, everything is encoded with groupoids and maps of 
  groupoids. Of course these groupoids and maps should not come out of the 
  blue; they should organise themselves into nice configurations. In this
  respect simplicial structures are particularly nice, for their 
  importance in homotopy theory and category theory.
  In a recent series of papers (starting with
  \cite{Galvez-Kock-Tonks:1512.07573}), G\'alvez, Kock, and Tonks show how
  certain simplicial groupoids called {\em decomposition spaces} admit the
  construction of incidence (co)algebras at the groupoid-slice level (and a
  M\"obius inversion principle); a few key concepts are recalled in 
  Subsection~\ref{sub:simplicial}. This construction is a common
  generalisation of classical constructions of incidence (co)algebras of
  posets \cite{Joni-Rota}, \cite{Schmitt:incidence},
  monoids~\cite{Cartier-Foata}, and M\"obius categories~\cite{Leroux:1976},
  \cite{Content-Lemay-Leroux}, \cite{Lawvere-Menni}, but reveals also most other coalgebras in
  combinatorics to be of incidence type (see for
  example~\cite{Galvez-Kock-Tonks:1708.02570} and
  \cite{Galvez-Kock-Tonks:1612.09225}). Decomposition spaces are the same
  thing as the {\em $2$-Segal spaces} of Dyckerhoff and
  Kapranov~\cite{Dyckerhoff-Kapranov:1212.3563} (see
  \cite{Feller-Garner-Kock-Proulx-Weber:1905.09580} for the last piece of
  this equivalence), of importance in homological algebra and
  representation theory, notably in connection with Hall algebras and the
  Waldhausen $S$-construction~\cite{Dyckerhoff-Kapranov:1212.3563},
  \cite{Dyckerhoff:1505.06940}, \cite{Young:1611.09234},
  \cite{Penney:1710.02742}, \cite{Poguntke:1709.06510}. The theory is now
  being developed in many directions.

  The incidence comodule bialgebras of the present work will be constructed
  at the level of slice categories using the decomposition-space machinery.
\end{blanko}
  
\begin{blanko}{Incidence bialgebras of operads.}
  Central to this paper is the now-standard construction of
  a bialgebra from an operad (see
  \cite{vanderLaan:math-ph/0311013},
  \cite{vanderLaan-Moerdijk:hep-th/0210226}, and
  \cite{Chapoton-Livernet:0707.3725} for related constructions).
  Given an operad $\RRR$ (satisfying finiteness
  conditions, cf.~\ref{fin-cond} below), the associated bialgebra is
  free as an algebra on the set of iso-classes of operations. The
  comultiplication of an operation $r$ is given by summing over all
  ways $r$ can arise by operad substitution from a collection of
  operations fed into a single operation:
  $$
  \Delta(r) = \sum_{r=b \circ (a_1,\ldots,a_n)}  a_1
  \cdots a_n \tensor b  .
  $$
  
  It is fruitful to break the construction into three
  steps~\cite{Kock-Weber:1609.03276}, each of which is important in its own
  right, routing it through simplicial methods and homotopy linear algebra.
  The {\em first step} is to take the two-sided bar construction on
  the operad~\cite{Weber:1503.07585}, a classical construction in algebraic
  topology and homological algebra \cite{May:LNM271}. This produces a
  simplicial groupoid $X=\TSB{\SSS}{\RRR}$, which is a symmetric monoidal 
  Segal space~\cite{Kock-Weber:1609.03276} (cf.~\ref{sub:bar} below). The
  {\em second step} constructs from any monoidal Segal groupoid $X$ a
  bialgebra structure on the slice $\Grpd_{/X_1}$, using the machinery of
  decomposition spaces \cite{Galvez-Kock-Tonks:1512.07573},
  \cite{Galvez-Kock-Tonks:1512.07577}. The {\em third step}
  is to take homotopy cardinality~\cite{Galvez-Kock-Tonks:1602.05082}, to
  recover usual bialgebras at the vector-space level.
  
  This three-step realisation of the
  incidence-bialgebra construction is a main ingredient in this work,
  where there will be {\em two} of them in interaction.
\end{blanko}

\begin{blanko}{Polynomial functors and trees.}
  The operads used in this work are more precisely {\em finitary 
  polynomial monads over groupoids}, as explained in Subsection~\ref{sub:monads}. 

  The notion of polynomial functor has origins in topology, representation
  theory, combinatorics, logic, and computer science, but the task of
  unifying these developments has only recently
  begun~\cite{Gambino-Kock:0906.4931}, \cite{Weber:1106.1983}. As a first
  approximation, polynomial functors objectify polynomial functions and
  formal power series. More generally, the theory turns out to be a
  powerful toolbox for studying substitution and recursion. Upgraded to
  groupoids, it can subsume the notions of species and
  operads~\cite{Kock:MFPS28}, \cite{Gepner-Haugseng-Kock:1712.06469}. The
  more specific reason polynomial functors are useful in the present context is that,
  since they are represented by diagrams of groupoids
  $$
  I \leftarrow E \to B \to I,
  $$
  many of the calculations performed at the objective level can be
  expressed naturally in terms of these representing groupoids. Secondly,
  trees --- the combinatorial substrate for the theory of operads --- can
  be represented by the same shape of diagram as polynomial
  endofunctors~\cite{Kock:0807}: $I$ is then the set of edges, $B$ the set of
  nodes, and $E$ the edge-node incidence --- see \ref{sub:trees}.
  Altogether, the polynomial formalism helps articulate the interplay
  between algebra and combinatorics.
\end{blanko}

\subsection{Present contributions}

\begin{blanko}{The incidence comodule bialgebra construction, general form.}
  The present paper connects the developments outlined above to
  provide a rather general categorical construction of comodule bialgebras,
  namely as the incidence bialgebras of the Baez--Dolan construction for
  operads (polynomial monads over groupoids).
  
  For any operad $\PPP$ one can construct two new operads: the free operad
  on $\PPP$, denoted $\PPP\upperstar$ (\ref{sub:free}), and the Baez--Dolan 
  construction $\PPP\bd$ (\ref{sub:BD}). These operads have the same space of operations,
  namely the groupoid of $\PPP$-trees, but behave very differently, with
  different notions of colours, arity, and substitution. The free operad
  $\PPP\upperstar$ is about grafting of trees, whereas the Baez--Dolan
  construction $\PPP\bd$ is about substituting trees into nodes of other
  trees.
  
  The two-sided bar construction now gives simplicial groupoids $Y =
  \TSB{\SSS}{\PPP\upperstar}$ (\ref{sub:bar(free)}) and $Z = \TSB{\SSS}{\PPP\bd}$ 
  (\ref{sub:bar(BD)}). Applying the incidence bialgebra construction, one obtains two
  different bialgebra structures on the same groupoid slice $\Grpd_{/Y_1}
  \simeq \Grpd_{/Z_1}$. The constructions are categorical and rather formal
  --- no ad hoc  
  choices are involved --- but at the same time all the groupoids and maps 
  involved have compelling combinatorial interpretations in terms of trees.
  These two bialgebras have the same underlying
  algebra, but different comultiplications. The first version of the main
  result of this paper states that the two
  bialgebras form a comodule bialgebra:

  \medskip

  \noindent {\bf Theorem~\ref{thm:main}.}
  {\em
  For any operad $\PPP$, the
  two-sided bar constructions $Y:=\TSB{\SSS}{\PPP\upperstar}$ and
  $Z:=\TSB{\SSS}{\PPP\bd}$ together endow the slice $ \Grpd_{/\SSS(\tr(\PPP))} $ with
  the structure of a comodule bialgebra. Precisely, the incidence bialgebra
  of $\TSB{\SSS}{\PPP\upperstar}$ is a left comodule bialgebra over the
  incidence bialgebra of $\TSB{\SSS}{\PPP\bd}$.
  }
  
  \medskip
  
\end{blanko}

\begin{blanko}{Proofs: homotopy pullbacks, trees, and active-inert factorisations.}
  At the objective level, the comodule-bialgebra axioms state that certain
  linear functors are isomorphic. The linear functors are given by
  composites of spans of groupoids, extracted naturally from the bar
  constructions.
  The spans are composed by pullback --- this always means homotopy 
  pullback. In the end the proof (taking up Subsection~\ref{sub:mainthm}) amounts to
  establishing an equivalence of groupoids between two different 
  pullbacks, or more precisely, exhibiting a groupoid which serves as 
  pullback for both.
  
  While all this is rather formal, exploiting basic properties of
  pullbacks, the actual checks must of course use features specific to the
  situation. In the present case, the proofs end up relying on properties
  of the category of trees, and in particular its active-inert
  factorisation system, recalled in \ref{active}. The groupoid central to
  the proof is formally described as having objects pairs of composable
  active maps of $\PPP$-trees
  $$
  W \actto K \actto T
  $$
  where $W$ is a $2$-level tree.
  At the same time, this groupoid has the clean combinatorial 
  interpretation
  as a groupoid of `blobbed' trees with a compatible cut, like this:
  
  
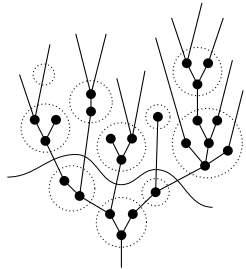
\begin{center}
  \begin{tikzpicture}
	
	\begin{scope}[shift={(0.0,0.0)}]
	  \draw[densely dotted] (0.0, 0.0) circle (0.33);
	  \coordinate (b1) at (0.0, -0.17);
	  \coordinate (b21) at (-0.15, 0.11);
	  \coordinate (b22) at (0.15, 0.11);
	  \draw (b1)--(b21) pic {onedot};
	  \draw (b1)--(b22) pic {onedot};
	\end{scope}

	\begin{scope}[shift={(-0.65,0.45)}]
	  \draw[densely dotted] (0.0, 0.0) circle (0.3);
	  \coordinate (c1) at (0.1, -0.1);
	  \coordinate (c2) at (-0.1, 0.1);
	\end{scope}

	\begin{scope}[shift={(0.45,0.4)}]
	  \draw[densely dotted] (0.0, 0.0) circle (0.18);
	  \coordinate (d) at (0.0, 0.0);
	\end{scope}
	
	\begin{scope}[shift={(-1.0,1.25)}]
	  \draw[densely dotted] (0.0, 0.0) circle (0.32);
	  \coordinate (e1) at (0.0, -0.17);
	  \coordinate (e21) at (-0.14, 0.11);
	  \coordinate (e22) at (0.14, 0.11);
	  \draw (e1)--(e21) pic {onedot} -- +(-0.2, 0.6);
	  \draw (e21) -- +(0.2, 1.00);
	  \draw[densely dotted] (e21)+(0.12, 0.60) circle (0.14);
	  \draw (e1)--(e22) pic {onedot};
	\end{scope}

	\begin{scope}[shift={(0.0,1.0)}]
	  \draw[densely dotted] (0.0, 0.0) circle (0.32);
	  \coordinate (f1) at (0.0, -0.17);
	  \coordinate (f21) at (-0.14, 0.11);
	  \coordinate (f22) at (0.14, 0.11);
	  \draw (f1)--(f21) pic {onedot};
	  \draw (f1)--(f22) pic {onedot};
	  \draw (f22) -- +(-0.2, 0.8);
	  \draw (f22) -- +(0.17, 0.9);
	\end{scope}

	\begin{scope}[shift={(-0.4,1.6)}]
	  \draw[densely dotted] (0.0, 0.0) circle (0.28);
	  \coordinate (u3) at (0.0, -0.13);
	  \coordinate (u31) at (0.0, 0.1);
	  \draw (u3) -- (u31) pic {onedot} -- +(-0.2, 0.8);
	  \draw (u31) -- +(0.2, 0.8);
	\end{scope}

   \begin{scope}[shift={(1.1,1.05)}]
	  \draw[densely dotted] (0.03, 0.0) circle (0.46);
	  \coordinate (r1) at (0.0, -0.3);
	  \coordinate (r2) at (-0.25, 0.0);
	  \coordinate (r3) at (0.05, 0.0);
	  \coordinate (r4) at (0.3, -0.1);
	  \coordinate (r5) at (-0.1, 0.3);
	  \coordinate (r6) at (0.16, 0.3);
	  \draw (r1) -- (r2) pic {onedot} -- +(-0.4, 1.4);
	  \draw (r1) -- (r3) pic {onedot};
	  \draw (r1) -- (r4) pic {onedot} -- +(0.2, 0.8);
	  \draw (r3) -- (r5) pic {onedot};
	  \draw (r3) -- (r6) pic {onedot} -- +(0.2, 0.8);
	\end{scope}
	   
	\begin{scope}[shift={(0.48,1.4)}]
	  \draw[densely dotted] (0.0, 0.0) circle (0.16);
	  \coordinate (v) at (0.0, 0.0);
	\end{scope}

	\begin{scope}[shift={(1.0,2.0)}]
	  \draw[densely dotted] (0.0, 0.0) circle (0.32);
	  \coordinate (w1) at (0.0, -0.17);
	  \coordinate (w21) at (-0.14, 0.11);
	  \coordinate (w22) at (0.14, 0.11);
	  \draw (w1)--(w21) pic {onedot} -- +(-0.2, 0.6);
	  \draw (w1)--(w22) pic {onedot} -- +(0.2, 0.6);
	  \draw (w21) -- +(0.2, 0.8);
	\end{scope}

	\draw (0.0,-0.6) -- (b1) pic {onedot};
	\draw (b21)--(c1) pic {onedot} --(c2) pic {onedot};
	\draw (b22)--(d) pic {onedot};
	\draw (c2)--(e1) pic {onedot};
	\draw (b21)--(f1) pic {onedot};
	\draw (c1)--(u3) pic {onedot};
	\draw (d)--(r1) pic {onedot};
	\draw (d)--(v) pic {onedot};
	\draw (r5)--(w1) pic {onedot};

	\draw[ultra thin] (-1.5, 0.6) to[out=0, in=180] (-0.6, 0.9);
	\draw[ultra thin] (-0.6, 0.9) to[out=0, in=180] (0.0, 0.5);
	\draw[ultra thin] (0.0, 0.5) to[out=0, in=180] (0.45, 0.7);
	\draw[ultra thin] (0.45, 0.7) to[out=0, in=180] (1.2, 0.2);

  \end{tikzpicture}
  \end{center}

\end{blanko}

\begin{blanko}{Finiteness conditions.}
  At the objective level, no finiteness conditions are required, which is 
  pleasant in terms of elbow room.
  However, in the end it is interesting
  to be able to take homotopy
  cardinality, to make contact with all the interesting 
  mathematics taking place at the level of ordinary vector spaces. 
  To this end it is necessary to impose finiteness conditions, and in fact
  some adjustments to the constructions are required, in order to keep things
  finite.
  For the incidence coalgebra of a decomposition space $X$ to admit a
  homotopy cardinality, it must be assumed locally
  finite~\cite{Galvez-Kock-Tonks:1512.07577}, which means that the two maps
  $
  X_0 \stackrel{s_0}\to X_1$ and $X_1 \stackrel{d_1}{\leftarrow} X_2
  $
  have finite homotopy fibres. For two-sided bar constructions, this in
  turn happens for {\em locally finite}
  operads~\cite{Kock-Weber:1609.03276}. For example, the terminal operad
  (the operad for commutative (unital) monoids) is {\em not} locally
  finite, whereas the terminal reduced operad (the operad for commutative
  not-necessarily-unital monoids) {\em is} locally finite.
\end{blanko}
  
\begin{blanko}{Locally finite version of the main theorem.}
  While the free operad $\PPP\upperstar$ is {\em always} locally finite
  (cf.~\cite{Kock-Weber:1609.03276} and \ref{lem:freefinite} below), the
  Baez--Dolan construction $\PPP\bd$ is {\em never} locally finite
  (\ref{lem:BDnotfinite}), and as a result its two-sided bar construction
  is not locally finite. Therefore one cannot just take homotopy
  cardinality of the general main theorem \ref{thm:main} to get a
  traditional comodule bialgebra. In Section~\ref{sec:finiteness} it is
  explained how to fix this in a canonical way. The main point is to use
  the {\em reduced} Baez--Dolan construction $\PPP\bdr:=
  \overline{\PPP\bd}$, which is simply the Baez--Dolan construction
  followed by operad reduction, whereby all nullary operations are
  excluded. This tweak is completely analogous to subtleties regarding the
  most classical comodule bialgebras: for example, while one can multiply
  arbitrary formal power series, in order to substitute one power series
  into another it must be required to have no constant term.

  The two-sided bar construction on the reduced Baez--Dolan construction
  $\PPP\bdr$ {\em is} locally finite (Lemma~\ref{lem:BDRfin}), and
  therefore admits a homotopy cardinality giving a traditional bialgebra in
  vector spaces. This is the motivation for replacing $\PPP\bd$ with
  $\PPP\bdr$, but we continue to work at the objective level. The
  interaction between $\PPP\upperstar$ and $\PPP\bdr$ is a bit more
  complicated now, since the two bialgebras are no longer supported at the
  same groupoid, and we need to consider comodules more general than just
  those underlying a bialgebra itself. The objective account of comodules
  is given by certain {\em culf} simplicial maps~\cite{Walde:1611.08241},
  \cite{Young:1611.09234}, \cite{Carlier:1801.07504}. We need a third
  simplicial groupoid which is a comodule over $\TSB{\SSS}{\PPP\bdr}$ and
  having the same base groupoid as $\TSB{\SSS}{\PPP\upperstar}$. This is
  achieved through another relative two-sided bar construction (of
  $\PPP\bdr$ relative to $\PPP\bd$), which also has another nice
  description (Proposition~\ref{prop:=Actop}): it is the symmetrisation of
  the fat nerve of the opposite of the category of $\PPP$-trees and active
  injections:
  $$
  \TSB{\PPP\bd}{\PPP\bdr} \simeq 
  \SSS \fatnerve \,\OMEGA_{\operatorname{act.inj}}(\PPP)\op .
  $$

  This comodule structure is shown to meet the finiteness requirements, and
  is the main ingredient in the proof of the following
  locally finite version of the main theorem~\ref{thm:main}.
\medskip

  \noindent {\bf Theorem~\ref{thm:main-finite}.}
  {\em 
  For any operad $\PPP$, the two-sided bar constructions
  $Y=\TSB{\SSS}{\PPP\upperstar}$ and $Z=\TSB{\SSS}{\PPP\bdr}$ together
  endow the slice $\Grpd_{/\SSS(\tr(\PPP))}$ with the structure of a {\em
  locally finite} comodule bialgebra. Precisely, the incidence bialgebra of
  $Y$ is a locally finite left comodule bialgebra over the incidence bialgebra of $Z$.  
  }
  
  \medskip
   \end{blanko}

\begin{blanko}{Examples.}
  In the final section we work out some examples.
  
  Subsection~\ref{sub:unary} deals with the special case where the operad is just a category (regarded as an 
  operad with only unary operations). Then there is an interesting 
  simplification, namely that the comodule bialgebra is actually free
  on a comodule {\em co}algebra. This feature is shared by the examples of 
  comodule bialgebras constructed by Carlier~\cite{Carlier:1903.07964}
  from hereditary species.
  
  For the more specific examples, the first is simply the incidence 
  comodule bialgebra of the Baez--Dolan construction on the identity 
  functor. This is shown (in Subsection~\ref{sub:FdB}) to be the Fa\`a di Bruno comodule bialgebra,
  dual to composition and multiplication of power series.
  
  The next class of examples is that of monoids, considered as
  one-colour operads with only unary operations. In this case we recover 
  the comodule bialgebras of mould calculus~\cite{Ecalle:I}.
  
  In Subsection~\ref{sec:CEFM} we come to the example that prompted this
  work: the Calaque--Ebrahimi-Fard--Manchon comodule bialgebra of rooted
  trees~\cite{Calaque-EbrahimiFard-Manchon:0806.2238} from the theory of
  B-series~\cite{Chartier-Hairer-Vilmart:FCM2010}. In this case there is an extra step involved, namely taking
  core~\cite{Kock:1109.5785}. This is the passage from operadic trees to combinatorial trees,
  consisting simply in omitting decorations and shaving off leaves and
  root. 
  
  Subsection~\ref{sub:freeP} deals with variations on the class of examples
  where the base operad $\PPP$ is free on a linear order, on a quiver, or
  on a polynomial endofunctor. This leads to certain comodule bialgebras of
  words, paths in quivers, and trees.
  
  In the final subsection~\ref{sub:outlook}, some non-examples are treated.
  The first is rather close to being a Baez--Dolan construction, namely the
  comodule bialgebra of noncrossing partitions of
  Ebrahimi-Fard--Foissy--Kock--Patras~\cite{EbrahimiFard-Foissy-Kock-Patras:1907.01190}.
  It is outlined here how the Baez--Dolan construction should be modified
  in order to cover this example. To finish, it is briefly discussed why
  the comodule bialgebras of Bruned,
  Hairer, and Zambotti~\cite{Bruned-Hairer-Zambotti:1610.08468} do not come
  from a Baez--Dolan construction.
\end{blanko}

\section{Preliminaries}

\label{sec:prelims}

Objective combinatorics works with objects instead of numbers. A
systematic way of achieving this is to use slice categories instead of vector
spaces~\cite{Galvez-Kock-Tonks:1602.05082}; for efficiency one
uses groupoids instead of sets, in order to take symmetries
into account automatically. After a brief introduction to homotopy linear 
algebra, we recall the basic notions of 
decomposition spaces and polynomial functors, the two main toolboxes 
employed in this work.

\subsection{Objective combinatorics and linear algebra}
\label{sub:Set/I}

In this first subsection we work with sets instead of groupoids, only to 
emphasise the main ideas in their simplest form. In Subsection~\ref{sub:groupoids}
we upgrade to groupoids, as actually needed in this work.

\begin{blanko}{Vector spaces versus slice categories.}
  For $S$ a set, we denote by $\Q_S$ the vector space spanned by $S$,
  and by $\delta_s$ the basis vector indexed by $s\in S$. A general vector
  in $\Q_S$ is a (finite) linear combination $\sum_{s\in S} \lambda_s
  \delta_s$, essentially the same thing as a (finite)
  $S$-indexed family of scalars $\lambda_s$.
  
  We want to consider $S$-indexed families of sets (or groupoids) $(X_s
  \mid s\in S)$, encoded as a single map $p:X \to S$; the members of the
  family are then the fibres $ X_s := p^{-1}(s)$,
  defined formally as the pullback
  $$
  \begin{tikzcd}
	X_s \drpullback \ar[d] \ar[r] & X \ar[d] \\
	1 \ar[r, "\name{s}"'] & S  .
  \end{tikzcd}
  $$
  Here $1$ denotes a singleton set, and $\name{s} : 1 \to S$
  is the map that picks out $s$.  
  
  The families over $S$ form the objects of the {\em slice category} 
  $\Set_{/S}$; its morphisms 
  are commutative triangles.
  The slice category $\Set_{/S}$ plays the role of $\Q_S$. Just as $\Q_S$
  is the free vector space on $S$, the slice $\Set_{/S}$ can be
  characterised as the completion of $S$ under sums (categorical 
  coproducts).
\end{blanko}

\begin{blanko}{Linear maps and linear functors.}
  Instead of linear maps $\Q_S \to \Q_T$, we have {\em linear functors}
  $\Set_{/S} \to \Set_{/T}$, which means functors that preserve
  sums. One can prove (cf.~\cite{Galvez-Kock-Tonks:1602.05082}) that every
  such functor is given by a span
  $$
  \begin{tikzcd}[sep={3em,between origins}]
	& M\ar[ld, "p"'] \ar[rd, "q"] & \\
	S && T
  \end{tikzcd}
  $$
  by pullback along $p$ followed by composing along $q$. This is denoted 
  $q\lowershriek \circ p\upperstar$.
  Note that $M$ can be regarded as an $S\times T$-indexed family of sets
  $$({}_s M_t \mid s\in S, t\in T),$$
  that is, a matrix of sets.
  The basic fact in linear algebra that composition of linear maps is given
  in coordinates by matrix multiplication has its objective analogue in the
  following fundamental lemma:
\end{blanko}
\begin{lemma}[`Beck--Chevalley']\label{lem:BC}
  For any pullback square
  \[\begin{tikzcd}
  X' \drpullback \ar[d, "p'"'] \ar[r, "g"] & X \ar[d, "p"]  \\
  S' \ar[r, "f"'] & S  ,
  \end{tikzcd}
  \]
  the canonical natural transformation of functors
  $$
  p'\lowershriek \circ g\upperstar  \Rightarrow f\upperstar \circ p\lowershriek
  $$
  is invertible.
\end{lemma}
This lemma ensures that composition of linear functors
is given by pullback composition of spans: in the situation of the
diagram
\[\begin{tikzcd}[sep = {2.8em,between origins}]
  && M\times_T N \dpullback \ar[ld] \ar[rd] && \\
  & M \ar[ld] \ar[rd] && N \ar[ld] \ar[rd] & 
  \\
  S && T && U 
\end{tikzcd}
\]
the linear functor defined by the outer span is isomorphic
to the composite of the linear functors defined by
the two small spans. Note that
 the `matrix' corresponding to the
  pullback can be written 
  $$
  \textstyle
  \big( \; \underset{t\in T}\sum \; {}_s M_t \times {}_t N_u \mid s\in S, u 
  \in U\;\big) ,
  $$
  which is the objective analogue of matrix multiplication.

\begin{blanko}{Cardinality.}
  The objective level works independently of finiteness conditions. 
  The passage from the slice-category level to the
  vector-space level consists in taking cardinality; this requires
  some finiteness conditions.
  
  The {\em cardinality} of a family
  $p:X \to S$ in $\Set_{/S}$ with $X$ finite is by definition the vector
  $$
  \norm{p} \; := \sum_{s\in S} \norm{X_s} \, \delta_s  \quad \in \Q_S ,
  $$
  where $\norm{X_s}$ is usual cardinality of the finite set $X_s$, and
  $\delta_s$ is the basis vector in $\Q_S$ corresponding to $s$. 
  The cardinality of pullback composition is matrix multiplication.
\end{blanko}

\begin{blanko}{Example: small categories are linear monads.}\label{cat=monadspan}
  Given a small category with object set $X_0$ 
  and set of arrows $X_1$, the span
  $$
  X_0 \stackrel{s}{\leftarrow} X_1 \stackrel{t}{\to} X_0
  $$
  (source and target) induces a linear endofunctor
  \begin{eqnarray*}
    \PPP: \Set_{/X_0} & \longrightarrow & \Set_{/X_0}  \\
    A{\to}X_0 & \longmapsto & t\lowershriek s\upperstar (A {\to} X_0) .
  \end{eqnarray*}
  The composite $\PPP\circ\PPP$
  is given by 
  $$
  
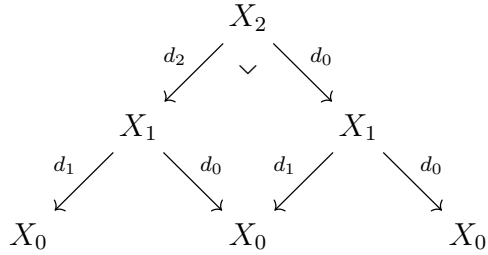
\begin{tikzcd}[sep={3.5em,between origins}]
	&& X_2 \dpullback \ar[ld, "d_2"'] \ar[rd, "d_0"]&&
	\\
	& X_1 \ar[ld, "d_1"'] \ar[rd, "d_0"] && X_1 \ar[ld, "d_1"'] \ar[rd, "d_0"] & 
	\\
	X_0 && X_0 && X_0  
  \end{tikzcd}
  $$
  with $X_2 \simeq X_1 \times_{X_0} X_1$ the set of pairs of composable arrows.
  Composition of arrows (which is the face map $d_1 : X_2 \to X_1$)
  defines a monad multiplication $\mu:\PPP \circ \PPP \Rightarrow \PPP$.
  The assignment of identity arrows to each object $s_0 : X_0 
  \to X_1$ is precisely the unit for the 
  monad, $\eta: \Id \Rightarrow \PPP$.
  The monad axioms amount precisely to the category axioms. Conversely, a 
  linear monad defines a category.
\end{blanko}

\subsection{Groupoids and homotopy pullbacks}
\label{sub:groupoids}

To account for symmetries of combinatorial objects, it is convenient
to upgrade the theory from sets to groupoids. For this to work, it is 
necessary that all notions be taken in their 
homotopy sense: pullback will thus mean homotopy pullback, fibre will mean
homotopy fibre, and so on. The passage to numbers is now given by homotopy 
cardinality of (families of) groupoids; see~\ref{card} below.
Introductions to this machinery can be found in preliminary sections of 
\cite{Galvez-Kock-Tonks:1207.6404}, or in the appendix of 
\cite{Galvez-Kock-Tonks:1612.09225}. A fuller treatment (in the case of 
$\infty$-groupoids) can be found in \cite{Galvez-Kock-Tonks:1602.05082}.
In due time, the book manuscript 
\cite{Carlier-Kock} should become a suitable reference.

Recall that a groupoid $X$ is a small category in which all arrows are
invertible. A map of groupoids is just a functor. An equivalence of
groupoids is just an equivalence of categories. Topology provides valuable
intuition and terminology: an object in a groupoid is also called a point; 
an arrow is also called a path; a
natural transformation is also called a homotopy. We write $\pi_0(X)$ for
the set of connected components of a groupoid $X$.
  
We are interested in groupoids up to equivalence. For this reason, the
usual notions of pullback, fibres, slices, adjoints, and so on, are not
appropriate --- they are not homotopy invariant. We shall need the homotopy
notions. If just they are used consistently, they behave very much like the
ordinary notions do for sets.

\begin{blanko}{Homotopy commutative squares.}
  Playing the role of commutative squares we have the {\em homotopy} 
  commutative squares: they are squares of groupoids
   \[
  \begin{tikzcd}
  P \ar[d] \ar[r] \ar[rd, phantom, pos =0.3, "\theta" description]& Y \ar[d]  \\
  X \ar[r] \ar[ru, Rightarrow, shorten <= 10pt, 
	shorten >= 10pt]& S
  \end{tikzcd}
  \]
 that commute only up to a homotopy (natural transformation). It is 
 important that this homotopy $\theta$ is specified as part of the data.
 Nevertheless, in the present paper
 the homotopy will always be clear from the context (very often simply 
 because the square actually happens to commute strictly), and we will
 suppress it throughout.
 The same convention goes for homotopy commutative triangles, and for 
 other shapes of diagram.
\end{blanko}

\begin{blanko}{Homotopy pullbacks.}\label{hopbk}
  A homotopy pullback is a (homotopy commutative) square
  \[
  \begin{tikzcd}
  P \ar[d] \ar[r] & Y \ar[d, "q"]  \\
  X \ar[r, "p"'] 
	& S
  \end{tikzcd}
  \]
  satisfying a universal property among all such squares with common $p$
  and $q$. As such it is determined uniquely up to equivalence. There are
  different (but equivalent) models for homotopy pullback. The standard
  model is given by the so-called comma category, which is
  the groupoid whose objects are triples $(x,y,\sigma)$ with $x\in X$ and
  $y\in Y$, and where $\sigma: px \isopil qy$ is an arrow in $S$
  (constituting the components of the natural transformation which is part
  of the data of the square). This model plays an important role in the
  development of the homotopy theory of groupoids~\cite{Carlier-Kock}, and
  it is good to keep in mind as a blueprint for the idea of homotopy
  pullback. Nevertheless, in this work it does not appear explicitly. The
  only homotopy pullbacks needed will be computed using the following three
  practical lemmas.
\end{blanko}

\begin{blanko}{Fibrations.}
  A map of groupoids $p:X\to B$ is a {\em fibration} when it satisfies the 
  {\em path lifting property} from topology: for each point $x\in X$ and 
  arrow $\beta: 
  b \isopil p(x)$ in $B$, there exists an arrow $\alpha:x'\isopil x$ such 
  that $p(\alpha)=\beta$.
\end{blanko}

\begin{lemma}
  An ordinary (strict) pullback is also a homotopy pullback if one of the
  two maps $p$ and $q$ is a fibration.
\end{lemma}

\begin{lemma}[Prism Lemma]\label{lem:prism}
  Given a prism diagram of groupoids
  \begin{center}
    \begin{tikzcd}
        X''   \ar[r, ""] \ar[d, ""]
          & X' \ar[r, ""]\ar[d, ""] \drpullback
          & X\ar[d, ""]\\
        Y'' \ar[r, ""'] & Y' \ar[r, ""] & Y
    \end{tikzcd}
  \end{center}
	in which the right-hand square is a (homotopy) pullback, then the outer
	rectangle is a (homotopy) pullback if and only of the left-hand square
	is a (homotopy) pullback.
\end{lemma}

\begin{blanko}{Homotopy fibre.}
  Given a map of groupoids $p: X \to S$ and a point
  $s \in S$, the \emph{(homotopy) fibre} $X_s$ of $p$ over $s$ is the 
  (homotopy) pullback
\[
    \begin{tikzcd}
        X_s \ar[r, ""] 
            \ar[d, ""'] 
            \drpullback
             & X \ar[d, "p"]\\
        1 \ar[r, "\name{s}"'] & S.
    \end{tikzcd}
\]
\end{blanko}

\begin{lemma}
  [Fibre Lemma]\label{lem:fibre}
  A square of groupoids
  \begin{center}
    \begin{tikzcd}
        P \ar[r, "u"] \ar[d] & Y \ar[d] \\
        X \ar[r, "f"'] & S
    \end{tikzcd}
  \end{center}
  is a (homotopy) pullback if and only if for each $x\in X$ the induced comparison map
  $u_x\colon P_x \to Y_{fx}$ is an equivalence.
\end{lemma}

%
%

\begin{blanko}{Linear functors.}
  For a groupoid $B$, the slice category $\Grpd_{/B}$ has objects families
  $X\to B$, but in contrast to the set case, the morphisms in $\Grpd_{/B}$
  are allowed to be triangles commuting only up to a specified isomorphism. This
  slack is required to obtain a notion invariant under equivalence, but in
  practice many triangles commute
  strictly.
  
  Every map $f: B \to A$ defines a functor $f\upperstar: \Grpd_{/A}\to 
  \Grpd_{/B}$ by homotopy pullback along $f$.  This functor has both 
  adjoints (which means adjoints up to homotopy): the left adjoint 
  $f\lowershriek :\Grpd_{/B} \to \Grpd_{/A}$ can be taken to be simply postcomposition
  with $f$. 
  
  A {\em linear functor} $\Grpd_{/I} \to \Grpd_{/J}$ is one 
  equivalent to $q\lowershriek\circ p\upperstar$ for a span
  $I \stackrel{p}\leftarrow M \stackrel{q}\to J$. These turn out to be
  precisely the cocontinuous functors.
  We get in this way a 
  category $\kat{LIN}$ whose objects are slices $\Grpd_{/I}$ and whose morphisms are 
  linear functors~\cite{Galvez-Kock-Tonks:1602.05082}.
  The category $\kat{LIN}$
  features a tensor product, defined as
  $$
  \Grpd_{/I} \tensor \Grpd_{/J} := \Grpd_{/(I \times J)},
  $$
  in analogy with the familiar fact for vector spaces
  $\Q_I \tensor \Q_J = \Q_{I\times J}$.
  The neutral object is $\Grpd\simeq \Grpd_{/1}$.
  In reality, $\kat{LIN}$ should be construed a monoidal $2$-category or as an 
  monoidal $\infty$-category (in both cases, the coherence of the tensor 
  product is a consequence of the coherence of the cartesian product), but this
  subtlety is not important in this paper.
\end{blanko}

\begin{blanko}{Remarks.}\label{weak?}
  As $f\upperstar$ is defined by homotopy pullback, in turn only defined up
  to equivalence, in principle $f\upperstar$ is only a pseudofunctor. This
  is the natural level of strictness, and the results for them look the
  same as for ordinary functors between set slices. To justify the
  manipulations rigorously, it can sometimes be necessary to invoke the
  theory of $\infty$-categories, as developed by Joyal~\cite{Joyal:CRM} and
  Lurie~\cite{Lurie:HTT} and many other people --- the paper
  \cite{Gepner-Haugseng-Kock:1712.06469} develops the theory of polynomial
  monads in this setting, and contains everything needed for the present
  purposes. 
  
  It can be felt a bit awkward to invoke this huge machinery. In
  a different approach, there are at least two ways to stay within the
  realm of strict functors. One is to take pullback to mean always the
  standard homotopy pullback, given by the comma category construction, as
  mentioned in \ref{hopbk}. This is in fact a strict functor, and restricts
  to strict slices (where the morphisms are strictly commutative
  triangles)~\cite{Carlier-Kock}. The second is to arrange always for
  pulling back only along fibrations, and then use the strict pullback
  (between strict slices). In concrete situations of combinatorial origin,
  it is usually not difficult to take one of these stricter routes, but it
  is rather cumbersome to develop the general theory in this way, as the
  strictnesses are not homotopy invariant, and can be tricky to maintain
  through various constructions.
  
  In order to focus on the main ideas,
  the present exposition generally favours the fully homotopy-invariant
  approach, which is much closer to combinatorial intuition (where one 
  freely replaces a groupoid of combinatorial structures by an equivalent 
  one), avoiding going
  into details with these issues. However, in many places it is practical to
  carry out a local argument using fibrations, as we shall see.
  The reader is invited to read the text also according to the stricter
  approaches. The follow-up remarks \ref{followup} and \ref{followup2} will
  elaborate on this issue in connection with monads, and in connection with
  the Baez--Dolan construction, respectively, and provide hints for such a
  reading.
\end{blanko}

\begin{blanko}{Homotopy cardinality of groupoids 
  \cite{Baez-Hoffnung-Walker:0908.4305}, \cite{Galvez-Kock-Tonks:1602.05082}.}
  \label{card}
  A groupoid $X$ is called finite if $\pi_0(X)$ is a finite set, and all
  $\Aut(x)$ are finite groups.
  The {\em homotopy cardinality} of a groupoid $X$ is defined to be
  $$
  \norm{X} = \sum_{x\in \pi_0 X} \frac{1}{\norm{\Aut(x)}} .
  $$
  
  More generally, the {\em homotopy cardinality} of a family $X \to B$ is
  $$
  \sum_{b\in \pi_0 B} \frac{\norm{X_b}}{\norm{\Aut(b)}} \cdot \delta_b \quad \in 
  \Q_B ,
  $$
  where $\delta_b$ is the basis vector corresponding to $b$. Homotopy 
  cardinality is a homotopy invariant notion.
\end{blanko}

\begin{blanko}{Homotopy quotients~\cite{Baez-Dolan:finset-feynman}.}\label{X/G}
  Given a group action $G \times X \to X$ (for $G$ a group and $X$ a set
  or groupoid), instead of the naive quotient (set of orbits) it is
  better to use the corresponding homotopy invariant notion of {\em 
  homotopy quotient}.
  It is an example of a 
  homotopy colimit, and as such is determined up to equivalence.
  One specific model (also called {\em weak
  quotient} or {\em action groupoid}) is obtained from $X$ by sewing in a path from
  $x$ to $g.x$ for each $x\in X$ and $g\in G$. It has much better
  properties than the naive quotient regarding interaction with the other
  basic constructions such
  as homotopy cardinality and homotopy
  pullbacks~\cite{Galvez-Kock-Tonks:1602.05082}, and it will be denoted
  simply $X/G$. (The naive quotient (which would be $\pi_0(X/G)$) is never
  employed in this work.) In particular, we have
  $$
  \norm{X/G} = \norm{X}/\norm{G}
  $$
  (where $\norm{G}$ is the order of the group $G$).
\end{blanko}

\begin{blanko}{Discrete and finite maps.}
  A map of groupoids is called {\em discrete} if all its homotopy fibres 
  are discrete. It is called {\em finite} if all its homotopy fibres are 
  finite.
\end{blanko}

From now on, the words {\em pullback}, {\em fibre}, {\em quotient} will 
always refer to the homotopy notions.

\subsection{Simplicial groupoids, decomposition spaces, and incidence coalgebras}
\label{sub:simplicial}

\begin{blanko}{The simplex category and the active-inert factorisation system.}
  \label{active-inert-Delta}
  The simplex category $\simplexcategory$ is the category of nonempty finite 
  linear orders 
  $$
  [n] := \{0<1<\cdots< n\}
  $$
  and monotone maps. It features an active-inert factorisation system: An
  arrow in $\simplexcategory$ is \emph{active}, written 
  $a:[m]\actto [n]$, when it
  preserves end-points, $a(0)=0$ and $a(m)=n$; and it is \emph{inert}, 
  written $a:[m] \into [n]$, if it
  preserves distance, meaning $a(i+1)=a(i)+1$ for $0\leq i\leq m-1$. The
  active maps are generated by the codegeneracy maps $s^i : [n{+}1] \to [n]$
  and by the {\em inner} coface maps $d^i : [n{-}1]\to [n]$, $0 < i < n$,
  while the inert maps are generated by the {\em outer} coface maps $d^0$
  and $d^n$. Every morphism in $\simplexcategory$ factors
  uniquely as an active map followed by an inert map. Furthermore, it is a
  basic fact \cite{Galvez-Kock-Tonks:1512.07573} that active and inert maps
  in $\simplexcategory$ admit pushouts along each other, and the resulting
  maps are again active and inert.
\end{blanko}

\begin{blanko}{Simplicial groupoids.}
  A simplicial groupoid is a functor $X: \simplexcategory\op\to\Grpd$.
  (In principle, in order to have a homotopy invariant notion, we allow 
  simplicial groupoids to be pseudofunctors, but in this work all 
  simplicial groupoids can actually be arranged to be strict functors.)
  Since $\simplexcategory$ is generated by coface and codegeneracy maps, a
  simplicial groupoid is conveniently described by indicating the face and
  degeneracy maps, picturing it as
  \[\begin{tikzcd}[column sep = {7em,between origins}]
  X_0
  \ar[r, pos=0.65, "s_0" on top] 
  &
  \ar[l, shift left=6pt, pos=0.65, "d_0" on top]
  \ar[l, shift right=6pt, pos=0.65, "d_1" on top]
  X_1
  \ar[r, shift right=6pt, pos=0.65, "s_0" on top]
  \ar[r, shift left=6pt, pos=0.65, "s_1" on top]  
  &
  \ar[l, shift left=12pt, pos=0.65, "d_0" on top]
  \ar[l, shift right=12pt, pos=0.65, "d_2" on top]
  \ar[l, pos=0.65, "d_1" on top]
  X_2
  \ar[r, shift left=12pt, pos=0.65, "s_2" on top]
  \ar[r, pos=0.65, "s_1" on top]
  \ar[r, shift right=12pt, pos=0.65, "s_0" on top]
  &
  \ar[l, shift right=18pt, pos=0.65, "d_3" on top]
  \ar[l, shift right=6pt, pos=0.65, "d_2" on top]
  \ar[l, shift left=6pt, pos=0.65, "d_1" on top]
  \ar[l, shift left=18pt, pos=0.65, "d_0" on top]
  X_3 
  \ar[r, phantom, "\dots" on top]
  & {}
  \end{tikzcd}\]
  The subscripts on the $d_i$ refer to which vertex of a simplex is omitted. The
  subscripts on the $s_i$ refer to which vertex is repeated. (The 
  simplicial identities, such as for example $d_0 \circ d_2 = d_1\circ d_0$,
  are not captured by the picture.)
\end{blanko}

\begin{blanko}{Segal spaces.}
  A simplicial groupoid $X: \simplexcategory\op\to\Grpd$ is called a {\em Segal 
  space} when for all $n\geq 0$ the simplicial-identity square 
  $$
  \begin{tikzcd}
  X_{n+2} \ar[d, "d_0"'] \ar[r, "d_{n+2}"] & X_{n+1} \ar[d, "d_0"]  \\
  X_{n+1} \ar[r, "d_{n+1}"'] & X_{n}
  \end{tikzcd}
  $$
  is a (homotopy) pullback.
\end{blanko}

\begin{blanko}{Fat nerve.}
  The base case $n=0$ of the Segal condition can be interpreted as saying 
  that the $2$-simplices are pairs of `composable $1$-simplices'.
  Indeed the Segal condition is motivated by categories.
  For $\CC$ a small category, the {\em fat nerve} of $\CC$ is the simplicial 
  groupoid
  \begin{eqnarray*}
    \fatnerve \CC : \simplexcategory\op & \longrightarrow & \Grpd  \\
    {}[n] & \longmapsto & \kat{Fun}([n],\CC)^{\operatorname{iso}}
  \end{eqnarray*}
  (strings of composable arrows). It is always a
  Segal space. In fact, up to homotopy, the Segal condition characterises
  the simplicial groupoids that are fat nerves of categories. In this work,
  Segal spaces arise mainly from the two-sided bar construction on an 
  operad, cf.~Section~\ref{sec:bar}.
\end{blanko}

\begin{blanko}{Decomposition spaces~\cite{Galvez-Kock-Tonks:1512.07573}/$2$-Segal spaces \cite{Dyckerhoff-Kapranov:1212.3563}.}
  \label{incidence}
  A {\em decomposition space} is a simplicial groupoid $X:
  \simplexcategory\op\to\Grpd$ that takes active-inert pushouts in
  $\simplexcategory$ to pullbacks in $\Grpd$.
  Decomposition spaces are the same thing as the
  $2$-Segal spaces of Dyckerhoff and 
  Kapranov~\cite{Dyckerhoff-Kapranov:1212.3563}.
  
  Every Segal groupoid is a decomposition
  space~\cite{Dyckerhoff-Kapranov:1212.3563},~\cite{Galvez-Kock-Tonks:1512.07573}.
\end{blanko}

\begin{blanko}{The incidence coalgebra of a decomposition space~\cite{Galvez-Kock-Tonks:1512.07573}.}
  \label{inc}
  The motivating property of decomposition spaces is that they admit the
  incidence-coalgebra construction, meaning that the functor
  $$
  \begin{tikzcd}
  \Delta: \Grpd_{/X_1} 
  \ar[r, "{(d_2,d_0)\lowershriek\circ d_1{}\upperstar}"] &
  \Grpd_{/X_1} \tensor \Grpd_{/X_1} 
  \end{tikzcd}
  $$
  given by the canonical span
  \begin{equation}\label{eq:Xspan}
  \begin{tikzcd}
	X_1 & \ar[l, "d_1"'] X_2 \ar[r, "{(d_2,d_0)}"] & X_1 \times X_1
  \end{tikzcd}
  \end{equation}
  defines a (homotopy-coherent) coassociative
  comultiplication~\cite{Dyckerhoff-Kapranov:1212.3563},
  \cite{Galvez-Kock-Tonks:1512.07573}. Note that only simplicial
  degree $1$ and $2$ are needed in order to define the
  comultiplication, but $X_3$ enters to express coassociativity, and
  (in the general homotopical setting) all the higher $X_k$ are
  needed to express coherence~\cite{Galvez-Kock-Tonks:1512.07573},
  \cite{Dyckerhoff-Kapranov:1212.3563}, \cite{Penney:1710.02742}.
\end{blanko}

\begin{blanko}{Examples: posets, monoids, categories.}
  If $\CC$ is a poset, a monoid, or more generally a category, then the
  nerve or fat nerve $X = \fatnerve \CC$ is a Segal space and hence a decomposition 
  space. The incidence coalgebra construction on $X$ now recovers the 
  classical notions in the case where $\CC$ is a locally finite poset 
  \cite{Joni-Rota}, a decomposition-finite monoid \cite{Cartier-Foata}, or 
  more generally
  a M\"obius category in the sense of Leroux~\cite{Leroux:1976}, after 
  taking homotopy cardinality. Indeed,
  Leroux's formula for comultiplication on the vector space spanned by the 
  arrows of $\CC$ is
  \begin{equation}\label{Delta(f)}
  \Delta(f) = \sum_{b \circ a = f} a \tensor b  .
  \end{equation}
  This is precisely what comes out of the general construction in 
  \ref{inc}: the sum in \eqref{Delta(f)} is taken over the
  fibre (that is, pullback along $d_1$)
  $$
  \begin{tikzcd}
  (X_2)_f \drpullback \ar[d] \ar[r] & X_2 \ar[d, "d_1"]  \\
  1 \ar[r, "\name{f}"'] & X_1
  \end{tikzcd}
  $$
  (the set of pairs of composable arrows whose composite is $f$), and 
  returning the two constituents of the decomposition is precisely to apply
  $(d_2,d_0)\lowershriek$.
\end{blanko}

Many coalgebras in combinatorics arise as the homotopy cardinality of 
the incidence coalgebra of a
decomposition space which is not a poset, monoid or
category~\cite{Galvez-Kock-Tonks:1708.02570},
\cite{Galvez-Kock-Tonks:1612.09225}.

\begin{blanko}{Example~\cite{Galvez-Kock-Tonks:1612.09225}.}\label{ex:CK}
  Pertinent to the present undertakings is the decomposition-space
  realisation of the Butcher--Connes--Kreimer Hopf algebra of rooted trees
  \cite{Butcher:1972}, \cite{Dur:1986}, \cite{Kreimer:9707029}. As an algebra it is free
  commutative on the set of iso-classes of rooted trees $T$. The
  comultiplication is defined by summing over certain admissible cuts $c$:
  \begin{equation}\label{eq:preintro}
  \Delta(T) \ = \sum_{c\in \operatorname{adm.cuts}(T)} P_c \tensor R_c .
  \end{equation}
  An {\em admissible cut} $c$ partitions the nodes of $T$ into two subsets or `layers'
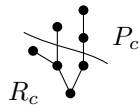
  \[
\begin{tikzpicture}
  \begin{scope}[shift={(0.0, 0.0)}]
    \draw (0.0, 0.0) pic {onedot} -- (-0.172, 0.368);
    \draw (0.0, 0.0) -- (0.172, 0.368);
  \end{scope}
  
  \begin{scope}[shift={(-0.172, 0.368)}]
    \draw (0.0, 0.0) pic {onedot} -- (-0.319, 0.196) pic {onedot};
    \draw (0.0, 0.0) -- (0.0, 0.515) pic {onedot};
  \end{scope}
  
  \begin{scope}[shift={(0.172, 0.368)}]
    \draw (0.0, 0.0) pic {onedot} -- (0.0, 0.368) pic {onedot} -- (0.0, 0.735) pic {onedot};
  \end{scope}
  
  \begin{scope}[shift={(0.0, 0.368)}]
    \draw (-0.613, 0.441) .. controls (-0.172, 0.221) and (0.172, 0.221) .. (0.49, 0.025);
  \end{scope}
  \draw (-0.613, 0.0) node {\footnotesize $R_c$};
  \draw (0.735, 0.735) node {\footnotesize $P_c$};
  \end{tikzpicture}
  \]
  One layer must form a rooted subtree $R_c$ (or be empty), and its
  complement forms the `crown', a subforest $P_c$ regarded as a monomial of
  trees. To realise this from a decomposition space
  (cf.~\cite{Galvez-Kock-Tonks:1612.09225}), let $\dstrees_k$ denote the
  groupoid of forests with $k-1$ compatible admissible cuts, partitioning
  the forest into $k$ layers (which may be empty). The $\dstrees_k$ assemble into
  a simplicial groupoid $\dstrees$, where degeneracy maps repeat a cut
  (that is, insert an empty layer), and face maps forget a cut (joining
  adjacent layers) or discard the top or bottom layer.

  The comultiplication \eqref{eq:preintro} arises from this simplicial
  groupoid by the general pull-push formula~\eqref{eq:Xspan}:
  for a tree $T\in \dstrees_1$, take the homotopy sum over the
  homotopy fibre $d_1^{-1}(T) \subset \dstrees_2$, and for each element $c$
  in this fibre return the pair $(d_2 c, d_0 c)$ consisting of the two
  layers. Finally take homotopy cardinality to arrive at $P_c \tensor R_c$.
  This simplicial groupoid
  is not a Segal groupoid, since it is not possible to reconstruct a tree 
  from the layers of a cut. One readily checks that it is a decomposition 
  space~\cite{Galvez-Kock-Tonks:1708.02570}, 
  \cite{Galvez-Kock-Tonks:1612.09225}.
  We shall come back to this example in \ref{sec:CEFM}.
\end{blanko}

\begin{blanko}{Culf maps~\cite{Galvez-Kock-Tonks:1512.07573}.}
  A simplicial map is {\em culf} (short for conservative and
  with unique lifting of factorisations) if, considered as a natural
  transformation of functors $\simplexcategory\op\to\Grpd$, it is cartesian
  on active maps (that is, the naturality squares on active maps are 
  (homotopy) pullbacks). Culf maps induce coalgebra
  homomorphisms~\cite{Galvez-Kock-Tonks:1512.07573}. In the present paper,
  the main use of this concept is that monoidal structures on decomposition
  spaces are required to be culf: this is what ensures that the
  multiplication resulting from taking cardinality is comultiplicative so
  as to yield a bialgebra. The second use of culfness is that the 
  objective analogue of comodules is given by certain culf
  maps~\cite{Walde:1611.08241}, \cite{Young:1611.09234},
  \cite{Carlier:1801.07504}, cf.~\ref{comoduleconf} below.
\end{blanko}

\subsection{Polynomial functors}

A standard reference for polynomial functors is
\cite{Gambino-Kock:0906.4931}; the long
manuscript~\cite{Kock:NotesOnPolynomialFunctors} aims at eventually
becoming a unified reference. While polynomial functors have often been studied
in the context of sets, for the present purposes it is necessary to deal
with polynomial functors over groupoids~\cite{Kock:MFPS28}, which
constitute a convenient language for operads~\cite{Weber:1412.7599}.
The upgrade from sets to groupoids is possible, provided
all notions are taken to be the homotopy notions. The full
$\infty$-categorical theory is developed in
\cite{Gepner-Haugseng-Kock:1712.06469}. For the present needs,
the introductions in \cite{Kock:MFPS28} and \cite{Kock:1512.03027}
should suffice as supplement to the following brief review.

\bigskip

A small category can be seen as a linear monad (\ref{cat=monadspan}), or as
an operad with only unary operations. The multi aspect of general operads
can be accounted for by passing from linear functors to polynomial
functors~\cite{Gambino-Kock:0906.4931},
\cite{Kock:NotesOnPolynomialFunctors}. Where linear functors are given by
spans $I \leftarrow M \to J$, polynomial functors incorporate a nonlinear
aspect by means of an extra `middle map':

\begin{blanko}{Polynomial functors.}
  A {\em polynomial} is a diagram of groupoids
  $$
  I \stackrel s \longleftarrow E \stackrel p \longrightarrow B \stackrel t 
  \longrightarrow J .
  $$
  The associated {\em polynomial functor}
  is the composite
  $$
  \Grpd_{/I} \overset{s\upperstar}\longrightarrow
  \Grpd_{/E} \overset{p\lowerstar}\longrightarrow
  \Grpd_{/B} \overset{t\lowershriek}\longrightarrow
  \Grpd_{/J}  .
  $$
  (Here of course we are talking about homotopy slices, upperstar is
  homotopy pullback, and lowerstar and lowershriek are the homotopy
  adjoints to homotopy pullback \cite{Kock:MFPS28},
  \cite{Kock:1512.03027}, \cite{Gepner-Haugseng-Kock:1712.06469}.)
  Generally, functors isomorphic to one of this form are called polynomial 
  functors. They can be characterised intrinsically as those preserving
  connected limits~\cite{Gambino-Kock:0906.4931}, 
  \cite{Gepner-Haugseng-Kock:1712.06469}, such as most notably pullbacks.
    
  We shall be concerned only with {\em finitary} polynomial functors.
  Abstractly this means polynomial functors that furthermore 
  preserve sifted colimits;
  in terms of the representing diagrams, they are those for which the
  middle map has finite discrete fibres~\cite{Kock:MFPS28},
  \cite{Gepner-Haugseng-Kock:1712.06469}. In this case, there is the
  following explicit formula for the polynomial functor~\cite{Kock:MFPS28}:
  $$
  (X_i \mid i\in I) \longmapsto (  \sum_{b\in \pi_0(B_j)} \prod_{e\in E_b} 
  X_{se}/\Aut(b) \mid j\in J)  ,
  $$
  where the quotient is of course a homotopy quotient (of the canonical 
  action of the group $\Aut(b)$), as in \ref{X/G}.
\end{blanko}

\begin{blanko}{Examples.}\label{ex:IMS-endo}
  (The monad structures on these examples will be dealt with in 
  \ref{sub:monads}.)
  The identity monad $\Id: \Grpd\to\Grpd$ is represented by $1 
  \leftarrow 1 \to 1 \to 1$.
  
  The {\em free-monoid monad} (also called the {\em word monad})
  \begin{eqnarray*}
    \MMM: \Grpd & \longrightarrow & \Grpd  \\
    X & \longmapsto & X\upperstar = \sum_{n\in \N} X^n
  \end{eqnarray*}
  is represented by the polynomial diagram
  $$
  1 \leftarrow \N' \to \N \to 1,
  $$
  where $\N'$ denotes the 
  set $\{ (n,i) \mid i\leq n \}$ so that the fibre over $n$ is an
  $n$-element set. The elements in $\MMM(X)$ are finite words in $X$.
  
  The {\em free-symmetric-monoidal-category monad}
  \begin{eqnarray*}
    \SSS:\Grpd & \longrightarrow & \Grpd  \\
    X & \longmapsto & \sum_{\underline n\in \pi_0\B} X^{\underline n} / 
	\Aut(\underline n) 
  \end{eqnarray*}
  is polynomial, represented by
  $$
  1\leftarrow \B' \to \B \to 1 .
  $$
  Here $\B =\SSS 1$ is the groupoid of finite sets and bijections, and $\B'$ is the
  groupoid of finite pointed sets and basepoint-preserving bijections. To 
  be specific we take $\B$ to mean the skeleton consisting of the finite 
  sets $\underline n := \{1,2,\ldots,n\}$.
  Then $\SSS X$ is the groupoid whose objects are the words in $X$, and
  whose morphisms from $(x_i)_{i \in \underline n}$ to $(y_i)_{i \in
  \underline n}$ are given by {\em decorated permutations}, meaning pairs
  $(\rho,(f_i)_{i \in \underline n})$ where $\rho \in \mathfrak{S}_n$ is a
  permutation of $\underline{n}$ and $f_i : x_i \to y_{\rho i}$ is an
  arrow in $X$ for each $i \in \underline{n}$.
  
  The objects of $\SSS X$ are called {\em monomials} of objects in $X$.
\end{blanko}

\begin{blanko}{Morphisms of polynomial functors~\cite{Gambino-Kock:0906.4931}.}
  \label{morphisms}
  {\em Morphisms} of polynomial functors are essentially cartesian natural
  transformations, but involving also change of colours. Precisely, they
  are given by diagrams
  $$
  \begin{tikzcd}
  \PPP': &   I' \ar[d, "F"']& \ar[l]  E'\ar[d] \drpullback \ar[r] & B'\ar[d] \ar[r] & 
  I'\ar[d, "F"] \\
  \PPP: &  I  &\ar[l] E\ar[r] & B \ar[r]  &I ,
  \end{tikzcd}
  $$
  where the middle square is a pullback (expressing arity preservation) 
  \cite{Gambino-Kock:0906.4931}.
  These correspond to cartesian natural transformations
  $$
  \begin{tikzcd}
	\Grpd_{/I'} \ar[r, "\PPP'"] \ar[d, "F\lowershriek"'] 
	& \Grpd_{/I'} \ar[d, "F\lowershriek"] \ar[ld, Rightarrow, shorten <= 16pt, 
	shorten >= 12pt, "\phi"']
	\\
	\Grpd_{/I} \ar[r, "\PPP"'] & \Grpd_{/I} .
  \end{tikzcd}
  $$
  We should now explain the notions colour and arity used here:
\end{blanko}

\begin{blanko}{Graphical interpretation of polynomial endofunctors~\cite{Kock:NotesOnPolynomialFunctors}.}
  \label{graphical}
  Given a polynomial endofunctor $\PPP$ represented by 
  $$
  I \leftarrow E \stackrel p \to B \to I  ,
  $$
  $I$ is interpreted as the groupoid of {\em colours} (or {\em objects}),
  and $B$ as the groupoid of {\em operations}. The {\em arity} of an
  operation $b\in B$ is the fibre $E_b$ (at the numerical level, the arity
  is its homotopy cardinality; assuming $\PPP$ is finitary, this is a
  natural number), and each operation is typed: the output colour of $b$ is
  $t(b)$, and the input colours are the $s(e)$ for $e\in E_b$. The groupoid
  $E$ is thus the groupoid of operations with a marked input slot. The map
  $p$ just forgets the mark. We may picture an element in $B$ as a corolla
  with node labelled by $b\in B$, and with leaves and root decorated by $I$
  according to this scheme. This is called a $\PPP$-corolla.
  
  Evaluation of $\PPP$ on an object $X\to I$ has the following
  combinatorial interpretation~\cite{Kock:NotesOnPolynomialFunctors}. 
  $\PPP(X)$ is the
  groupoid of $\PPP$-corollas as before, but furthermore 
  with leaves decorated in $X$ (subject to a compatibility condition:
  an element $x\in X$ may decorate leaf $\ell$ only
  if the colour of $x$ (under $X \to I$) matches the colour of $\ell$
  (under $E\to I$)).
  
  In particular, the endofunctor $\PPP\circ \PPP$ has as operations
  $\PPP$-corollas whose input edges are decorated by other $\PPP$-corollas,
  with compatible colours. This data is naturally interpreted as a
  `$2$-level $\PPP$-tree':
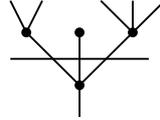
    \[
  \begin{tikzpicture}[line width=0.25mm]
  \begin{scope} 
	\draw (0.0, 0.28) -- (0.0, 0.7);
	\fill (0.0, 0.7) circle[radius=0.065];
	\draw (0.0, 0.7) -- (-0.7, 1.4);
	\fill (-0.7, 1.4) circle[radius=0.065];
	\draw (-0.7, 1.4) -- (-0.91, 1.82);
	\draw (-0.7, 1.4) -- (-0.49, 1.82);
	\draw (0.0, 0.7) -- (0.0, 1.4);
	\fill (0.0, 1.4) circle[radius=0.065];
	\draw (0.0, 0.7) -- (0.7, 1.4);
	\fill (0.7, 1.4) circle[radius=0.065];
	\draw (0.7, 1.4) -- (0.28, 1.82);
	\draw (0.7, 1.4) -- (0.7, 1.82);
	\draw (0.7, 1.4) -- (1.12, 1.82);	
	\draw (-0.91, 1.05) -- (0.91, 1.05);
  \end{scope}
  \end{tikzpicture}
  \]

\end{blanko}

We now formalise this in terms of trees.

\subsection{Trees}

\label{sub:trees}

\begin{blanko}{Trees.}\label{polytree-def}
  The trees relevant to the present context are {\em operadic} trees,
  i.e.~admitting open-ended edges for leaves and root, such as the
  following:

  
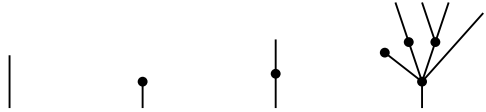
\begin{center}
  \begin{tikzpicture}[line width=0.25mm]
  \begin{scope}[shift={(-1.75, 0.0)}]
    \draw (0.0, 0.0) -- (0.0, 0.7);
  \end{scope}
  
  \begin{scope}[shift={(0.0, 0.0)}]
    \draw (0.0, 0.0) -- (0.0, 0.35);
	\fill (0.0, 0.35) circle[radius=0.065];
  \end{scope}
  
  \begin{scope}[shift={(1.75, 0.0)}]
    \draw (0.0, 0.0) -- (0.0, 0.91);
	\fill (0.0, 0.455) circle[radius=0.065];
  \end{scope}
  
  \begin{scope}[shift={(3.675, 0.0)}]
    \draw (0.0, 0.0) -- (0.0, 0.35);
	\fill (0.0, 0.35) circle[radius=0.065];
	\draw (0.0, 0.35) -- (-0.175, 0.875);
	\fill (-0.175, 0.875) circle[radius=0.065];
	\draw (-0.175, 0.875) -- (-0.35, 1.4);
    \draw (0.0, 0.35) -- (-0.49, 0.735);
	\fill (-0.49, 0.735) circle[radius=0.065];
    \draw (0.0, 0.35) -- (0.175, 0.875);
	\fill (0.175, 0.875) circle[radius=0.065];
	\draw (0.175, 0.875) -- (0.0, 1.4);
    \draw (0.175, 0.875) -- (0.35, 1.4);
    \draw (0.0, 0.35) -- (0.805, 1.26);
  \end{scope}
  \end{tikzpicture}
  \end{center}

  It was observed in \cite{Kock:0807} that operadic trees can be
  conveniently encoded by diagrams of the same shape as polynomial endofunctors. 
  By definition, a
  {\em (finite, rooted) tree} is a diagram of finite sets
\begin{equation}\label{tree}
  A \stackrel{s}\longleftarrow M \stackrel{p}\longrightarrow N 
  \stackrel{t}\longrightarrow A
\end{equation}
satisfying the following three conditions:
  
  (1) $t$ is injective
  
  (2) $s$ is injective with singleton complement (called the {\em 
  root} and denoted $1$).
  
  \noindent With $A=1+M$, 
  define the walk-to-the-root function
  $\sigma: A \to A$ by $1\mapsto 1$ and $e\mapsto t(p(e))$ for
  $e\in M$. 
  
  (3)  $\forall x\in A : \exists k\in \N : \sigma^{k}(x)=1$.
  
  The elements of $A$ are called {\em edges}.  The elements of $N$
  are called {\em nodes}.  For $b\in N$, the edge $t(b)$ is called
  the {\em output edge} of the node.  That $t$ is injective is just to
  say that each edge is the output edge of at most one node.  For
  $b\in N$, the elements of the fibre $M_b:= p^{-1}(b)$ are
  called {\em input edges} of $b$.  Hence the whole set
  $M=\sum_{b\in N} M_b$ can be thought of as the set of
  nodes-with-a-marked-input-edge, i.e.~pairs $(b,e)$ where $b$ is a
  node and $e$ is an input edge of $b$.  The map $s$ returns the
  marked edge.  Condition (2) says that every edge is the input edge
  of a unique node, except the root edge.
  Condition (3) says that if you walk towards the root, in a finite 
  number of steps you arrive there.
  The edges not in the image of $t$ are called {\em leaves}.
  
  The tree $1 \leftarrow 0 \to 0 \to 1$ is the {\em trivial tree}
  \inlineDotlessTree .
  A {\em corolla} is a tree \!\inlinecorolla of the form
  $n{+}1 \leftarrow n \to 1 \to n{+}1$ (one node and $n$ input edges).
\end{blanko}

\begin{blanko}{Inert maps of trees (tree embeddings).}\label{inert}
  An {\em inert map} of trees is a diagram
  $$    \begin{tikzcd}
    A' \ar[d, "\alpha"']& \ar[l]  M'\ar[d] \drpullback \ar[r] & N'\ar[d] \ar[r] & 
    A'\ar[d, "\alpha"] \\
    A  &\ar[l] M\ar[r] & N \ar[r]  &A ,
  \end{tikzcd}
  $$
  where the middle square is a pullback.  It is a consequence of
  the tree axioms that inert maps of trees are necessarily injective~\cite{Kock:0807}.
  The pullback condition means (in view of the Fibre Lemma~\ref{lem:fibre})
  that a node must be mapped to a node of the same 
  arity.  In conclusion, the inert maps are precisely the full subtree  
  inclusions (called tree embeddings in \cite{Kock:0807}).
  
  The category of trees and inert maps has nice geometric features,
  including a Grothen\-dieck topology~\cite{Kock:0807}, useful to formalise 
  notions of gluing. Presently, it is of
  interest that grafting of trees is expressed as colimits in this
  category: every tree is canonically the colimit of its one-node subtrees.
  We shall later need more general maps of trees, which will be generated
  by the free-monad monad, cf.~\ref{active} below.
\end{blanko}


\begin{blanko}{$n$-level trees.}\label{levelled}
  The {\em height} of an edge $x\in A$ is defined as
  $$
  h(x) := \min \{ k\in \N : \sigma^{k}(x)=1 \} 
  $$
  (with reference to the walk-to-the-root function $\sigma$ of Axiom~3).
  In particular, the root edge has height $0$.
  An {\em  $n$-level tree} is a tree where all edges have 
  height $\leq n$ and all leaves have height
  precisely $n$. The trivial tree is thus a $0$-level tree, and a
  corolla is a $1$-level tree.  Note that a tree without leaves is
  an $n$-level tree for all sufficiently big $n$.
  
  To give an $n$-level tree is equivalent to giving a sequence of $n$
  maps of finite sets $A_n \to A_{n-1} \to \dots \to A_1 \to A_0=1$. 
  The corresponding tree is given by
  $$
  \sum_{i=0}^n A_i \longleftarrow \sum_{i=1}^n A_i \longrightarrow 
  \sum_{i=0}^{n-1} A_i\longrightarrow \sum_{i=0}^n A_i ,
  $$
  readily checked to satisfy the tree axioms, and being 
  $n$-levelled by construction --- its set of leaves is $A_n$.
\end{blanko}
 
\begin{blanko}{$\PPP$-trees.}\label{Ptree}
  Having trees and polynomials on the same footing makes it easy to
  deal with decorations of trees \cite{Kock:0807} (see also
  \cite{Kock:1109.5785}, \cite{Kock:MFPS28}, \cite{Kock-Joyal-Batanin-Mascari:0706}). With
  a polynomial endofunctor $\PPP$ fixed, given by a diagram $I \leftarrow E
  \to B \to I$,
  a {\em $\PPP$-tree} is  by definition a diagram
  $$    \begin{tikzcd}
    A \ar[d, "\alpha"']& \ar[l]  M\ar[d] \drpullback \ar[r] & N\ar[d] \ar[r] & 
    A\ar[d, "\alpha"] \\
    I  &\ar[l] E\ar[r] & B \ar[r]  &I ,
  \end{tikzcd}
  $$
  where the top row is a tree.  Hence nodes are decorated by elements in $B$,
  and edges are decorated by elements in $I$, subject to obvious 
  compatibilities.  That the middle square is a
  pullback expresses that $n$-ary nodes of the tree have to be decorated by
  $n$-ary operations, and that a specific bijection is given.
\end{blanko}

\begin{blanko}{Examples of $\PPP$-trees.}
  With reference to Example~\ref{ex:IMS-endo}, 
  $\Id$-trees are linear trees, $\MMM$-trees
  are planar trees, and $\SSS$-trees are abstract trees (naked trees).
  More exotically, if $\PPP(X)= 1+X^2$ then $\PPP$-trees are planar binary 
  trees.
\end{blanko}

Denote by $\tr(\PPP)$ the groupoid of $\PPP$-trees, by 
$\operatorname{cor}(\PPP)$ the groupoid of $\PPP$-corollas, and by 
$\operatorname{triv}(\PPP)$ the groupoid of trivial $\PPP$-trees.
\begin{lemma}[\cite{Kock:1512.03027}]\label{lem:cor(P)}
  There are canonical equivalences
  $$
  I \simeq \operatorname{triv}(\PPP)
  \qquad \text{ and } \qquad B \simeq \operatorname{cor}(\PPP).
  $$
\end{lemma}

\subsection{Polynomial monads and operads}

\label{sub:monads}

\begin{blanko}{Polynomial monads.}\label{monadcontract}\label{residue}
  A {\em polynomial monad} is a polynomial endofunctor $\PPP$ equipped with
  cartesian natural transformations $\Id\Rightarrow \PPP \Leftarrow \PPP
  \circ \PPP$, required to satisfy the associative and unital laws 
  \cite{MacLane:categories}.
  The multiplication law $\mu: \PPP\circ \PPP \Rightarrow \PPP$ on a
  polynomial endofunctor $\PPP$ can be seen as a rule prescribing how to
  contract each $2$-level $\PPP$-tree to a $\PPP$-corolla, preserving
  arities. The unit law $\eta: \Id \Rightarrow \PPP$ assigns to each colour
  a special unary corolla; we may think of this as contraction of trivial 
  $\PPP$-trees 
  \inlineDotlessTree\, to \inlineonetree. More generally, iteration of
  these laws gives the unbiased view on monads, which can be seen as a law
  prescribing how to obtain from a whole tree configuration of
  $\PPP$-operations (that is, a $\PPP$-tree) a single $\PPP$-operation (a 
  $\PPP$-corolla),  called the {\em residue} of the $\PPP$-tree.
\end{blanko}

\begin{blanko}{Example: the free-monoid monad.}\label{ex:M}
  The free-monoid endofunctor from Example~\ref{ex:IMS-endo},
  $\MMM(X) = \sum_{n\in \N} X^n$, represented by
  $$
  1 \leftarrow \N'\stackrel p \to \N \to 1,
  $$
  has a canonical monad structure. For $X$ a set (or more generally a
  groupoid), $\MMM(X)$ is the set of words in $X$, and $\MMM(\MMM(X))$ is
  the set of words of words in $X$. The monad multiplication is
  concatenation of words (that is, removal of parentheses). The unit 
  interprets an element in $X$ as a word of length $1$.
  
  The element $n\in \N$ can be represented as a planar $n$-corolla. Then
  the composite endofunctor $\MMM\circ \MMM$ has as operations $2$-level
  planar trees, and the monad multiplication simply contracts such a
  $2$-level tree to the corolla with the same number of leaves.
  Preservation of leaf number is precisely to say that the natural
  transformations $\Id \Rightarrow \MMM \Leftarrow\MMM\circ \MMM$ are
  cartesian.
\end{blanko}

The following example is fundamental to our undertakings.

\begin{blanko}{The free-symmetric-monoidal-category monad.}\label{S} 
  See \cite{Weber:1106.1983} for details.  The 
  free-symmetric-monoidal-category endofunctor $\SSS : \Grpd\to\Grpd$
  of Example~\ref{ex:IMS-endo}, represented by
  $$
  1\leftarrow \B' \to \B \to 1 ,
  $$
  is naturally a monad. For $X$ a groupoid, $\SSS X$ is the groupoid of
  {\em monomials} of objects in $X$. Explicitly these are finite sequences of
  objects $(x_i)_{i \in \underline n}$, with morphisms permutations
  decorated by arrows in $X$ (see \ref{ex:IMS-endo}). The multiplication
  $\mu_X : \SSS\SSS X \to \SSS X$ is given by concatenation 
  (disjoint union). The unit $\eta_X : X
  \to \SSS X$ interprets an object as a length-$1$ sequence. These natural 
  transformations are readily seen to be cartesian.
  
%
\end{blanko}

\begin{blanko}{Ordinary coloured symmetric operads as polynomial monads.}
  By \emph{operad} we mean coloured symmetric operad in $\Set$. We will not
  reproduce the standard definition here, because in this work we shall only  
  consider operads in the form of polynomial monads. The basic result in
  this direction is Weber's theorem:
\end{blanko}
  
\begin{theorem}[Weber~\cite{Weber:1412.7599}, Theorem~3.3]
  Operads with colour set $I$
  are essentially the same thing as polynomial monads $\PPP: \Grpd_{/I} 
  \to \Grpd_{/I}$ cartesian
  over the free-symmetric-monoidal-category monad, as in
  $$    
  \begin{tikzcd}
  I \ar[d, "F"']& \ar[l, "s"']  E\ar[d] \drpullback \ar[r, "p"] & B\ar[d] 
  \ar[r, "t"] & 
  I\ar[d, "F"] \\
  1  &\ar[l] \B'\ar[r] & \B \ar[r]  &1 ,
  \end{tikzcd}
  $$
  for which ($I$ is a set and) $B \to \B$ is a discrete fibration.
\end{theorem}
(For the notion of monad morphism expressed by the diagram, some further 
details are provided in \ref{opf-ax} below.)  

  Briefly, the equivalence goes as follows~\cite{Kock-Weber:1609.03276}.
  For an operad $\PPP$ with colour set $I$, the symmetric groups $\mathfrak
  S_n$ act on the sets of $n$-ary operations. Let $B$ be the disjoint union of
  the homotopy quotients of these actions.
  There is a canonical projection map to $\B$, itself the
  disjoint union of the $1$-object groupoids $\mathfrak S_n$. This is a
  discrete fibration, whose fibre over $n$ is the set $\PPP_n$ of $n$-ary 
  operations. The groupoid $E$ is given
  by (strict) pullback.
  The fibre of $E \to B$ over an operation $r$ is the set of its input
  slots. The monad structure on the polynomial endofunctor comes precisely
  from the substitution law of the operad $\PPP$.
  
  Conversely, given a polynomial monad cartesian over $\SSS$ as above, the
  discrete fibration $B \to \B$ induces a $\mathfrak S$-set which is the set of
  operations of an operad.  The set of $n$-ary operations is the (homotopy) fibre of $B \to
  \B$ over $n$. 
  The operad substitution law comes from the monad multiplication.
  
\begin{blanko}{Polynomial monads and operads as needed in this paper.}
  \label{ouroperads}
  It is often convenient to work with operads allowed to have a
  {\em groupoid} of colours rather than just a {\em set} of colours. 
  In this paper, the reason is simple: we shall see that  the
  Baez--Dolan construction naturally 
  produces operads with a groupoid of colours, even if given as input
  an operad with a set of colours.
  We
  also need to give up the requirement that the classifying map $B \to \B$
  be a discrete fibration. Accordingly we define an {\em operad} (with
  groupoids of colours $I$) to be a finitary polynomial monad $\PPP :
  \Grpd_{/I} \to \Grpd_{/I}$ represented by a diagram of groupoids
  $$
  I \stackrel s \leftarrow E \stackrel p \to B \stackrel t  \to I .
  $$
  (The monad map to $\SSS$ as in Weber's theorem is not necessary in the 
  homotopical setting, because $\SSS$ is homotopy terminal 
  among finitary polynomial monads~\cite{Gepner-Haugseng-Kock:1712.06469}.)
\end{blanko}

\begin{blanko}{Remark.}
  It is not unlikely that this notion of operad is actually
  (bi)equivalent to the classical notion of symmetric operad featured in
  Weber's theorem. Discreteness of colours does not seem significant: it
  transpires from \cite{Batanin-Kock-Weber:1510.08934} that every 
  operad with groupoid colours should be equivalent to an 
  operad with discrete colours, the idea being that the symmetries of
  colours can be incorporated into the groupoid of operations as invertible
  unary operations. The details have not been worked out, though. On the
  other hand, the discrete-fibration condition on the map $B \to \B$ might
  be a more delicate issue.
\end{blanko}

\begin{blanko}{Example.}\label{MoverS}
  The free-monoid monad $\MMM$ (see \ref{ex:IMS-endo} and \ref{ex:M}) is
  an operad in the sense of \ref{ouroperads}, by means of the diagram
    $$    
  \begin{tikzcd}
  \MMM : &
  1 \ar[d, "="']& \ar[l]  \N'\ar[d] \drpullback \ar[r] & \N\ar[d] \ar[r] & 
  1\ar[d, "="] \\
  \SSS : & 
  1  &\ar[l] \B'\ar[r] & \B \ar[r]  &1 .
  \end{tikzcd}
  $$
  Note that the map $\N\to\B$ is not a fibration, but it could easily be 
  replaced by a fibration, by letting $\N$ denote the equivalent groupoid 
  of all finite 
  linear orders and monotone maps.
  
  A {\em nonsymmetric operad} is a polynomial monad 
  over $\MMM$.  Thereby it is also over $\SSS$; this is its symmetrisation.
  Note that in the polynomial formalism, the polynomial monad itself 
  does not change under symmetrisation.
\end{blanko}

\begin{blanko}{Follow-up remark (continuation of discussion in \ref{weak?}).}\label{followup}
  Since in the fully homotopical setting the polynomial functors are
  actually pseudofunctors, the monads we deal with are allowed to be
  pseudomonads. The theory of polynomial monads
  works well in the fully homotopical setting --- in fact it works well in
  the setting of $\infty$-groupoids~\cite{Gepner-Haugseng-Kock:1712.06469},
  and leads to a theory of $\infty$-operads. (In a precise sense, as
  explained there, it works better than over $\Set$!)
  
  On the other hand, Weber's polynomial approach to operads showcases the
  strict approach: since the polynomial monads in question are cartesian 
  over $\SSS$, and since the middle map $\B' \to \B$ is a fibration, 
  the cartesianness over it can be taken to be strict and thereby also
  all middle maps $p$ of polynomial functors involved are fibrations, and 
  one then works with strict pullbacks $p\upperstar$
  (and its strict right adjoint $p\lowerstar$), and in fact one can then 
  work over strict slices $\Grpd/I$, and  the polynomial monads are then ordinary monads 
  or $2$-monads. As noted (and as will be clear in the next section), 
  for the present purposes 
  it is not practical to 
  maintain all Weber's strictness conditions. 
  Instead, it is possible to work with the level of strictness where the
  monad map to $\SSS$ is maintained as part of the structure and is
  required to be strictly commutative, and where $s$, $p$ and $t$ as well
  as the arity map $B \to \B$ are all required to be fibrations. The
  viability of this approach will be substantiated in \ref{followup2} 
  below, once we have seen how the free-monad and Baez--Dolan 
  constructions work.
\end{blanko}

\section{Free monads, Baez--Dolan construction, and two-sided bar 
construction}
\label{sec:bar}

\subsection{The free monad $\PPP\upperstar$}

\label{sub:free}

For any polynomial endofunctor $\PPP$, one can construct the free
monad on $\PPP$: it is the (least) solution to the fixpoint equation of 
endofunctors
$$
\QQQ \simeq \Id + \PPP\circ \QQQ ,
$$
and it exists for general categorical reasons. More importantly, there is a
neat explicit polynomial representation:

\begin{theorem}[\cite{Kock:0807}, \cite{Kock:1512.03027}, \cite{Gepner-Haugseng-Kock:1712.06469}]
  \label{thm:P*}
  For a finitary polynomial endofunctor $\PPP$ represented by
  $I\leftarrow E \to B \to I$, the free monad $\PPP\upperstar$ is
  represented by
  $$
  I \longleftarrow \tr'(\PPP) \longrightarrow \tr(\PPP) \longrightarrow I ,
  $$
  where $\tr(\PPP)$ is the groupoid of $\PPP$-trees, and 
  $\tr'(\PPP)$ is the groupoid of $\PPP$-trees with a marked 
  leaf.
\end{theorem}

Graphically:
  
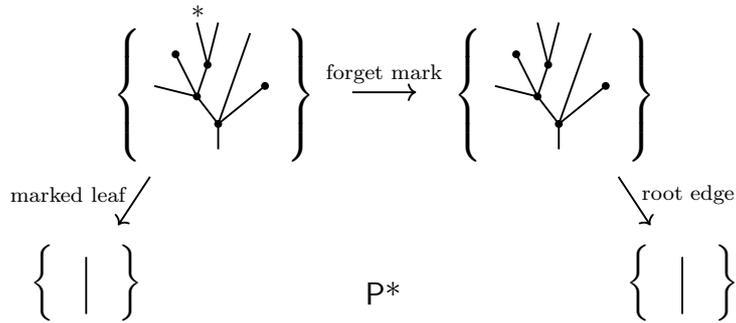
\[
  \begin{tikzpicture}[line width=0.25mm,scale=0.8]
	
    \begin{scope}[shift={(-2.73, 0.28)}] 
      \draw (0.0, 0.175) -- (0.0, 0.595);
	  \fill (0.0, 0.595) circle[radius=0.065];
	  \draw (0.0, 0.595) -- (-0.35, 1.05);
	  \fill (-0.35, 1.05) circle[radius=0.065];
	  \draw (-0.35, 1.05) -- (-1.05, 1.225);
      \draw (-0.35, 1.05) -- (-0.7, 1.75);
	  \fill (-0.7, 1.75) circle[radius=0.065];
      \draw (-0.35, 1.05) -- (-0.175, 1.575);
	  \fill (-0.175, 1.575) circle[radius=0.065];
	  \draw (-0.175, 1.575) -- (-0.35, 2.275);
      \draw (-0.175, 1.575) -- (0.0, 2.275);
      \draw (0.0, 0.595) -- (0.525, 2.1);
      \draw (0.0, 0.595) -- (0.77, 1.225);
	  \fill (0.77, 1.225) circle[radius=0.065];
      \draw (-0.33, 2.4) node {\footnotesize *};
    \end{scope}
	
    \draw (-4.2, 1.4) node {$\bigleftbrace{30}$};
    \draw (-1.4, 1.4) node {$\bigrightbrace{30}$};
	  
    \begin{scope}[shift={(2.87, 0.28)}] 
      \draw (0.0, 0.175) -- (0.0, 0.595);
	  \fill (0.0, 0.595) circle[radius=0.065];
	  \draw (0.0, 0.595) -- (-0.35, 1.05);
	  \fill (-0.35, 1.05) circle[radius=0.065];
	  \draw (-0.35, 1.05) -- (-1.05, 1.225);
      \draw (-0.35, 1.05) -- (-0.7, 1.75);
	  \fill (-0.7, 1.75) circle[radius=0.065];
      \draw (-0.35, 1.05) -- (-0.175, 1.575);
	  \fill (-0.175, 1.575) circle[radius=0.065];
	  \draw (-0.175, 1.575) -- (-0.35, 2.275);
      \draw (-0.175, 1.575) -- (0.0, 2.275);
      \draw (0.0, 0.595) -- (0.525, 2.1);
      \draw (0.0, 0.595) -- (0.77, 1.225);
	  \fill (0.77, 1.225) circle[radius=0.065];
    \end{scope}

	\draw (1.4, 1.4) node {$\bigleftbrace{30}$};
    \draw (4.2, 1.4) node {$\bigrightbrace{30}$};

	\begin{scope}[shift={(-4.9, -2.45)}] 
      \draw (0.0, 0.175) -- (0.0, 1.12);
    \end{scope}
	
    \draw (-5.6, -1.75) node {$\bigleftbrace{15}$};
    \draw (-4.2, -1.75) node {$\bigrightbrace{15}$};
    
	\begin{scope}[shift={(4.9, -2.45)}] 
      \draw (0.0, 0.175) -- (0.0, 1.12);
    \end{scope}
	
    \draw (4.2, -1.75) node {$\bigleftbrace{15}$};
    \draw (5.6, -1.75) node {$\bigrightbrace{15}$};   
    \draw (0.0, -1.925) node {$\PPP\upperstar$};
    \draw[->] (-0.525, 1.4) -- (0.525, 1.4);
    \draw[->] (3.85, 0.0) -- (4.375, -0.805);
    \draw[->] (-3.85, 0.0) -- (-4.375, -0.805);
	  \draw (0.0, 1.7) node {\scriptsize forget mark};
	  \draw (5.0, -0.3) node {\scriptsize root edge};
	  \draw (-5.2, -0.3) node {\scriptsize marked leaf};

  \end{tikzpicture}
  \]
  (The trees in the picture are $\PPP$-trees, but the $\PPP$-decorations 
  have been suppressed to avoid clutter.)
  
\begin{blanko}{Example: free monad on a tree.}
  Since a tree $A \leftarrow M \to N \to A$ is in particular a polynomial 
  endofunctor $\TTT$, one can consider the free monad on it. But according 
  to \ref{inert}, $\TTT$-trees are subtrees in $\TTT$, so that the free 
  monad on $\TTT$ is given by
  $$
  A \longleftarrow \operatorname{sub}'(\TTT) \longrightarrow \operatorname{sub}(\TTT)
  \longrightarrow A 
  $$  
  (with evident notation).
\end{blanko}

\begin{blanko}{Active maps and the Moerdijk--Weiss category of trees $\OMEGA$.}\label{active} 
  Moerdijk and Weiss~\cite{Moerdijk-Weiss:0701293} defined the category
  $\OMEGA$ of operadic trees to be the full subcategory of the category of
  operads spanned by the free operads on trees. Formally, a 
  map $T'\to T$ in $\OMEGA$ is a morphism of polynomial endofunctors $T' \to 
  T\upperstar$, that is, a diagram
  $$    \begin{tikzcd}[column sep={5em,between origins}]
    A' \ar[d, "\alpha"']& \ar[l]  M'\ar[d] \drpullback \ar[r] & N'\ar[d] \ar[r] & 
    A'\ar[d, "\alpha"] \\
    A  &\ar[l] \operatorname{sub}'(T) \ar[r] & \operatorname{sub}(T) 
	\ar[r]  &A .
  \end{tikzcd}
  $$
  
  In addition to the inert maps $\into$ already described in \ref{inert},
  $\OMEGA$ has {\em active} maps, denoted
  $\actto$, formally generated by the free-monad
  monad~\cite{Weber:TAC18}.
  In elementary terms they are {\em node refinements}
  \cite{Kock:0807}: for each corolla $C$ there is an active map to any tree
  with the same number of leaves:
  
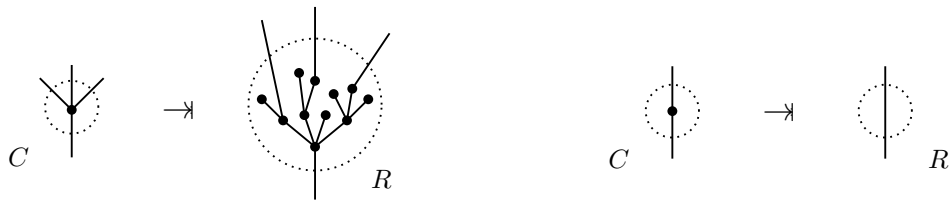
\begin{center}
  \begin{tikzpicture}[line width=0.25mm]
  \footnotesize
  \begin{scope}[shift={(0.0, 0.28)}]
    \draw (0.0, 0.0) -- (0.0, 1.225);
    \draw (0.0, 0.63) pic {onedot} -- (-0.42, 1.05);
    \draw (0.0, 0.63) -- (0.42, 1.05);
    \draw (-0.7, 0.0) node {$C$};
    \draw[dotted] (0.0, 0.665) circle (0.35);
  \end{scope}
  
  \draw (1.4, 0.875) node {\normalsize $\actto$};
  
  \begin{scope}[shift={(3.2, 0.0)}]
    \draw (0.875, 0.0) node {$R$};
    \draw (0.0, -0.28) -- (0.0, 0.42) pic {onedot}
    -- (-0.42, 0.77) pic {onedot}
    -- (-0.7, 1.05) pic {onedot};
    \draw (-0.42, 0.77) -- (-0.7, 2.1);
    \draw (0.0, 0.42) -- (-0.14, 0.84) pic {onedot}
    -- (-0.21, 1.4) pic {onedot};
    \draw (-0.14, 0.84) -- (0.0, 1.295) pic {onedot} -- (0.0, 2.275);
    \draw (0.0, 0.42) -- (0.14, 0.84) pic {onedot};
    \draw (0.0, 0.42) -- (0.42, 0.77) pic {onedot} -- (0.245, 1.12) pic {onedot};
    \draw (0.42, 0.77) -- (0.49, 1.19) pic {onedot} -- (0.98, 1.925);
    \draw (0.42, 0.77) -- (0.7, 1.05) pic {onedot};
    \draw[dotted] (0.0, 0.98) circle (0.875);
  \end{scope}
\end{tikzpicture}
\qquad\qquad\qquad
\begin{tikzpicture}[baseline=-1.5ex, line width=0.25mm]
  \footnotesize
  \begin{scope}[shift={(0.0, 0.28)}]
    \draw (0.0, 0.0) -- (0.0, 1.225);
    \draw (0.0, 0.63) pic {onedot};
    \draw (-0.7, 0.0) node {$C$};
    \draw[dotted] (0.0, 0.63) circle (0.35);
  \end{scope}
  
  \draw (1.4, 0.875) node {\normalsize $\actto$};
  
  \begin{scope}[shift={(2.8, 0.28)}]
    \draw (0.0, 0.0) -- (0.0, 1.225);    
    \draw (0.7, 0.0) node {$R$};
    \draw[dotted] (0.0, 0.63) circle (0.35);
  \end{scope}
\end{tikzpicture}
\end{center}
  (The second picture here illustrates the important special case
  where a unary node is refined into a trivial tree.)  Conversely, for a 
  fixed tree $T$, there is a unique active map from a corolla, namely from 
  the corolla with the same number of leaves. (More precisely, all the 
  existing such maps are uniquely isomorphic.)

  General active
  maps $K \actto T$ are described by the following lemma. For $K$ a
  tree, denote by $\operatorname{Act}(K)$ the groupoid of all active
  maps out of $K$.
\end{blanko}  

\begin{lemma}[\cite{Gepner-Haugseng-Kock:1712.06469}, 
  Lemma~5.3.8]\label{lem:Act(K)}
  If $C_1,\ldots,C_n$ are the nodes of a tree $K$, then there is a 
  canonical equivalence
  $$
  \operatorname{Act}(K) \simeq \prod_{i=1}^n \operatorname{Act}(C_i) .
  $$
\end{lemma}
\noindent  
In other words, an active map out of $K$ is given by
  refining each node of $K$ into a
  tree of matching arity as above, and then gluing together all these
  refinement maps according to the same gluing recipe that describes
  the tree $K$ as a colimit of its nodes, like this:  
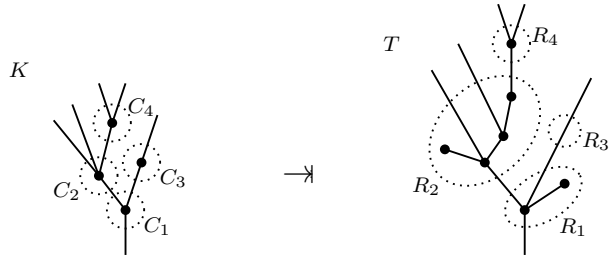
\begin{center}
  \begin{tikzpicture}[line width=0.25mm]
  \scriptsize
  \begin{scope}[shift={(5.25, 0.0)}]
    \draw (-1.4, 2.45) node {$K$};
    \draw (0.0, 0.0) -- (0.0, 0.595) pic {onedot}
    -- (-0.35, 1.05) pic {onedot} -- (-0.945, 1.785);
    \draw (-0.35, 1.05) -- (-0.7, 1.995);
    \draw (-0.35, 1.05) -- (-0.175, 1.75) pic {onedot} -- (-0.35, 2.275);
    \draw (-0.175, 1.75) -- (0.0, 2.275);
    \draw (0.0, 0.595) -- (0.42, 1.855);
    \draw (0.21, 1.225) pic {onedot};
    \draw (-0.77, 0.84) node {$C_2$};
    \draw (0.42, 0.385) node {$C_1$};
    \draw (0.63, 1.015) node {$C_3$};
    \draw (0.21, 1.96) node {$C_4$};
    \draw[dotted] (-0.35, 1.05) circle (0.245);
    \draw[dotted] (0.0, 0.595) circle (0.245);
    \draw[dotted] (0.21, 1.225) circle (0.245);
    \draw[dotted] (-0.175, 1.75) circle (0.245);
  \end{scope}
  
  \draw (7.525, 1.05) node {\normalsize $\actto$};
  
  \begin{scope}[shift={(10.5, 0.0)}]
	\draw (-1.75, 2.8) node {$T$};
	\draw (0.0, 0.0) -- (0.0, 0.595) pic {onedot}
	-- (-0.525, 1.225) pic {onedot} -- (-1.225, 2.45);
	\draw (-0.525, 1.225) -- (-1.05, 1.4) pic {onedot};
	\draw (-0.28, 1.575) -- (-0.875, 2.8);
	\draw (-0.525, 1.225) -- (-0.28, 1.575) pic {onedot}
	-- (-0.175, 2.1) pic {onedot} -- (-0.175, 2.8) pic {onedot} -- (-0.35, 3.325);
	\draw (-0.175, 2.8) -- (0.0, 3.325);
	\draw (-1.295, 0.91) node {$R_2$};
	\draw (0.64, 0.37) node {$R_1$};
	\draw (0.94, 1.505) node {$R_3$};
	\draw (0.27, 2.905) node {$R_4$};
	\draw (0.0, 0.595) -- (0.875, 2.345);
	\draw (0.0, 0.595) -- (0.525, 0.945) pic {onedot};
	\draw[dotted,rotate around={45:(-0.525,1.645)}] (-0.525,1.645) ellipse (0.84 and 0.595);
	\draw[dotted,rotate around={32:(0.245,0.77)}] (0.245,0.77) ellipse (0.56 and 0.315);
	\draw[dotted] (-0.175, 2.8) circle (0.245);
	\draw[dotted] (0.525, 1.645) circle (0.21);
  \end{scope}
\end{tikzpicture}
\end{center}
  
General maps in $\OMEGA$ look essentially the same but can land in
  bigger trees. More precisely, one is allowed to postcompose with an
  inclusion (inert map) of trees of $T$ into a bigger tree. This leads to:

\begin{blanko}{The active-inert factorisation system in $\OMEGA$ \cite{Kock:0807}.}
  \label{active-inert-Omega}
  Every arrow in $\OMEGA$ factors as an active map followed by an inert
  map, constituting the {\em active-inert factorisation system}. Restricted to the subcategory of linear trees
  $\simplexcategory\subset \OMEGA$, this gives the active-inert
  factorisation system on $\simplexcategory$ already described in
  \ref{active-inert-Delta}.

  In particular, we shall need to refactor as follows.
  Given an inert
  map followed by an active map (drawn as solid arrows):

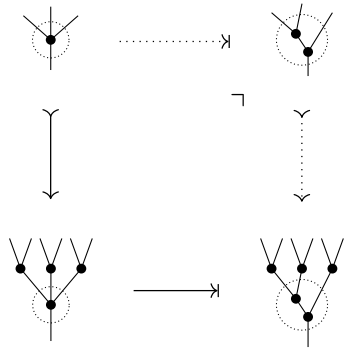
\newcommand{\trA}{
\begin{tikzpicture}[scale=0.8, baseline=10]
    \coordinate (b) at (0.0, 0.6);
    \coordinate (l1) at (-0.45, 1.0);
    \coordinate (l2) at (0.0, 1.15);
    \coordinate (l3) at (0.45, 1.0);
    \draw (0.0, 0.1) -- (b) pic {onedot};
	\draw (b) -- (l1);
	\draw (b) -- (l2);
	\draw (b) -- (l3);
	\draw[densely dotted] (b) circle (0.3);
  \end{tikzpicture}
}
\newcommand{\trB}{
  \begin{tikzpicture}[scale=0.8, baseline=10]
    \coordinate (b1) at (0.1, 0.4);
    \coordinate (b2) at (-0.1, 0.7);
    \coordinate (l1) at (-0.5, 1.05);
    \coordinate (l2) at (0.0, 1.2);
    \coordinate (l3) at (0.5, 1.05);
    \draw (0.1, 0.0) -- (b1) pic {onedot} -- (b2) pic {onedot};
	\draw (b2) -- (l1);
	\draw (b2) -- (l2);
	\draw (b1) -- (l3);
	\draw[densely dotted] (0.0, 0.6) circle (0.42);
  \end{tikzpicture}
}
\newcommand{\trC}{
  \begin{tikzpicture}[scale=0.8, baseline=16]
    \coordinate (b) at (0.0, 0.6);
    \coordinate (t1) at (-0.5, 1.2);
    \coordinate (t2) at (0.0, 1.2);
    \coordinate (t3) at (0.5, 1.2);
    \draw (0.0, 0.0) -- (b) pic {onedot};
	\draw (b) -- (t1) pic {onedot};
	\draw (b) -- (t2) pic {onedot};
	\draw (b) -- (t3) pic {onedot};
	\draw (t1) -- +(-0.18, 0.5);
	\draw (t1) -- +(0.18, 0.5);
	\draw (t2) -- +(-0.18, 0.5);
	\draw (t2) -- +(0.18, 0.5);
	\draw (t3) -- +(-0.18, 0.5);
	\draw (t3) -- +(0.18, 0.5);
	\draw[densely dotted] (b) circle (0.3);
  \end{tikzpicture}
}
\newcommand{\trD}{
  \begin{tikzpicture}[scale=0.8, baseline=16]
    \coordinate (b1) at (0.1, 0.4);
    \coordinate (b2) at (-0.1, 0.7);
    \coordinate (t1) at (-0.5, 1.2);
    \coordinate (t2) at (0.0, 1.2);
    \coordinate (t3) at (0.5, 1.2);
    \draw (0.1, -0.1) -- (b1) pic {onedot} -- (b2) pic {onedot};	
	\draw (b2) -- (t1) pic {onedot};
	\draw (b2) -- (t2) pic {onedot};
	\draw (b1) -- (t3) pic {onedot};
	\draw (t1) -- +(-0.18, 0.5);
	\draw (t1) -- +(0.18, 0.5);
	\draw (t2) -- +(-0.18, 0.5);
	\draw (t2) -- +(0.18, 0.5);
	\draw (t3) -- +(-0.18, 0.5);
	\draw (t3) -- +(0.18, 0.5);
	\draw[densely dotted] (0.0, 0.6) circle (0.42);
  \end{tikzpicture}
}

\[
\begin{tikzcd}[sep={8em,between origins}]
  \trA
  \ar[r, dotted, -act, shorten <= 10pt, shorten >= 10pt] 
  \ar[d, into, shorten <= 10pt, shorten >= 10pt] 
  & \trB \dlactinert
  \ar[d, dotted, into, shorten <= 8pt, shorten >= 9pt]
  \\
  \trC
  \ar[r, -act, shorten <= 10pt, shorten >= 10pt] & \trD
\end{tikzcd}
\]
\noindent
 there is a unique
  way of refactoring it into an active map followed by an inert map, as 
  indicated.
\end{blanko}

\begin{blanko}{Remark.}
  This active-inert factorisation system is a special case of the general
  notion of generic-free factorisation system introduced and studied
  deeply by Weber~\cite{Weber:TAC13}, \cite{Weber:TAC18}. Many of the
  important properties in the abstract theory of operads can be described
  in terms of such factorisation systems. In fact, recently Chu and
  Haugseng~\cite{Chu-Haugseng:1907.03977} have taken this viewpoint to the
  extreme, developing an axiomatic theory of operad-like structures based
  on this notion.
\end{blanko}

\begin{blanko}{Induction of $\PPP$-structure along maps in 
  $\OMEGA$.}\label{induce}
  Let $\PPP$ be any polynomial endofunctor (not required to have monad 
  structure). For an inert map of trees $S\into T$, if $T$ has 
  $\PPP$-tree structure, then there is induced a $\PPP$-tree structure on 
  $S$, simply by composition $S \into T \to \PPP$ (which is composition 
  of diagrams as in \ref{morphisms}).  
  This defines the category $\OMEGA_{\operatorname{inert}}(\PPP)$ of
  $\PPP$-trees and inert maps (for any polynomial endofunctor $\PPP$).
  
  If $\PPP$ is furthermore a monad, then the functoriality extends to
  active maps $K \actto T$. Indeed, by construction
  (\ref{active-inert-Omega}) an active map $K\actto T$ is a morphism of
  polynomials $K \to T\upperstar$, and since $\PPP$ is a monad, the
  $\PPP$-tree structure $T\to\PPP$ gives by adjunction a morphism of
  polynomials $T\upperstar \to \PPP$. The composite $K\to T\upperstar \to
  \PPP$ is the induced $\PPP$-structure on $K$. Altogether this defines,
  for any polynomial monad $\PPP$, a category of $\PPP$-trees
  $\OMEGA(\PPP)$ (which in fancier terms can be described as the category
  of elements of the dendroidal nerve of $\PPP$, cf.~\cite{Kock:0807}.)
\end{blanko}

\subsection{The Baez--Dolan construction $\PPP \bd$}

\label{sub:BD}

We now come to the central notion of this work, the Baez--Dolan
construction~\cite{Baez-Dolan:9702}. Curiously, Baez and Dolan defined the
construction for operads with {\em categories} of colours, but used it only
for operads with {\em sets} of colours. The set version is not optimal for
the purposes of opetope theory. It was adjusted by
Cheng~\cite{Cheng:0304277}, simply by taking the original definition
seriously and allowing a {\em groupoid} of colours. An alternative
adjustment was provided by Kock, Joyal, Batanin, and
Mascari~\cite{Kock-Joyal-Batanin-Mascari:0706}, using polynomial monads
over $\Set$. These cannot account for general operads, only for so-called
sigma-cofibrant operads~\cite{Kock:0807}, but that is enough for the
purpose of defining opetopes. The following version of the Baez--Dolan
construction is the polynomial version from
\cite{Kock-Joyal-Batanin-Mascari:0706}, but upgraded from sets to
groupoids, in the spirit of Cheng~\cite{Cheng:0304277}.

\begin{blanko}{Set-up.}\label{BD-setup}
  Let $\PPP$ be a polynomial monad, represented by a diagram of 
  groupoids
  $$
  I \leftarrow E \to B \to I  .
  $$
  Consider the category $\kat{PolyEnd}_{/\PPP}$ of polynomial
  endofunctors of $\Grpd_{/I}$ cartesian over $\PPP$. This category is
  monoidal under composition: given $\QQQ{\Rightarrow}\PPP$ and
  $\QQQ'{\Rightarrow}\PPP$, the composite is $\QQQ\circ \QQQ'
  \Rightarrow \PPP \circ \PPP \Rightarrow \PPP$, using the monad
  multiplication of $\PPP$. The neutral object for the composition is
  $\eta: \Id \Rightarrow \PPP$ (the monad unit). Consider also the
  category $\kat{PolyMnd}_{/\PPP}$ of polynomial monads cartesian over
  $\PPP$. The forgetful functor $U$ admits a left adjoint, the
  free-monad-over-$\PPP$ functor:
  $$\begin{tikzcd}
  \kat{PolyMnd}_{/\PPP} \ar[d, shift left=10pt, "U"] \ar[d, phantom, description, 
  "\isleftadjointto"]\\
  \kat{PolyEnd}_{/\PPP}  .  \ar[u, shift left=10pt]
  \end{tikzcd}$$
  Denote by $\TTT$ the monad generated by this adjunction. Now the key 
  point is that there is a natural equivalence
  \begin{equation}\label{eq:key-equiv}
  \kat{PolyEnd}_{/\PPP} \simeq \Grpd_{/B}  .
  \end{equation}
  This is because polynomial endofunctors cartesian over $\PPP$ are given by diagrams
  $$
  \begin{tikzcd}[row sep={4.5ex,between origins}]
	& U \ar[r] \ar[dd] \ar[ld] \drpullbackS  & V \ar[dd] \ar[rd] & \\
	I &&{}& I \,,\\
	& E \ar[r] \ar[ul] & B \ar[ur] &
  \end{tikzcd}
  $$
  but clearly this data is completely determined by the single map $V 
  \to B$.  
\end{blanko}

\begin{blanko}{The Baez--Dolan construction, polynomial version.}
  With notation as in \ref{BD-setup}, the {\em Baez--Dolan construction} on
  $\PPP$ is by definition the monad
  $$
  \begin{tikzcd}
  \Grpd_{/B} \ar[r, "\PPP\bd"] & \Grpd_{/B} 
  \end{tikzcd}
  $$
  obtained by transporting the monad
  $\TTT$ along the key equivalence $\kat{PolyEnd}_{/\PPP}
  \simeq \Grpd_{/B}$.
\end{blanko}

\begin{blanko}{Note.}
  In the literature the Baez--Dolan construction on $\PPP$ is usually
  denoted $\PPP^+$, and sometimes called the plus construction
  \cite{Baez-Dolan:9702}, \cite{Leinster:0305049},
  \cite{Kock-Joyal-Batanin-Mascari:0706}. The change in notation here is
  motivated by the fact that $\PPP\bd$ will relate to $\PPP\upperstar$ in
  the same way as substitution
  relates to multiplication (for power series), as will become clear.
\end{blanko}

One can now follow through the explicit description of the
free-monad-over-$\PPP$ monad and the key equivalence~\eqref{eq:key-equiv},
to obtain the following.

\begin{theorem}[\cite{Kock-Joyal-Batanin-Mascari:0706}]
  \label{thm:Pbd}
  The Baez--Dolan construction $\PPP\bd$ is polynomial, represented by
  $$
  B \longleftarrow \tr^\bullet(\PPP) \longrightarrow \tr(\PPP) 
  \longrightarrow B  .
  $$
\end{theorem}
\noindent
Here $\tr(\PPP)$ is the groupoid of $\PPP$-trees, and
  $\tr^\bullet(\PPP)$ is the groupoid of $\PPP$-trees with a
  marked node. The last map takes a $\PPP$-tree and contracts it to a
  $\PPP$-corolla (i.e.~an object of $B$, cf.~Lemma~\ref{lem:cor(P)})
  using the original monad structure on $\PPP$ 
  (cf.~\ref{monadcontract}).  The leftmost map returns the marked node
  of the $\PPP$-tree, considered as a $\PPP$-corolla.
  The middle map
  simply forgets the mark.
  
  Graphically:  
  
  \[
  \begin{tikzpicture}[line width=0.25mm,scale=0.8]
	
    \begin{scope}[shift={(-2.73, 0.28)}] 
      \draw (0.0, 0.175) -- (0.0, 0.595);
	  \fill (0.0, 0.595) circle[radius=0.065];
	  \draw (0.0, 0.595) -- (-0.35, 1.05);
	  \fill (-0.35, 1.05) circle[radius=0.065];
	  \draw (-0.35, 1.05) -- (-1.05, 1.225);
      \draw (-0.35, 1.05) -- (-0.7, 1.75);
	  \fill (-0.7, 1.75) circle[radius=0.065];
      \draw (-0.35, 1.05) -- (-0.175, 1.575);
	  \fill (-0.175, 1.575) circle[radius=0.065];
	  \draw (-0.175, 1.575) -- (-0.35, 2.275);
      \draw (-0.175, 1.575) -- (0.0, 2.275);
      \draw (0.0, 0.595) -- (0.525, 2.1);
      \draw (0.0, 0.595) -- (0.77, 1.225);
	  \fill (0.77, 1.225) circle[radius=0.065];
      \draw (-0.525, 0.84) node {\footnotesize *};
    \end{scope}
	
    \draw (-4.2, 1.4) node {$\bigleftbrace{30}$};
    \draw (-1.4, 1.4) node {$\bigrightbrace{30}$};
	  
    \begin{scope}[shift={(2.87, 0.28)}] 
      \draw (0.0, 0.175) -- (0.0, 0.595);
	  \fill (0.0, 0.595) circle[radius=0.065];
	  \draw (0.0, 0.595) -- (-0.35, 1.05);
	  \fill (-0.35, 1.05) circle[radius=0.065];
	  \draw (-0.35, 1.05) -- (-1.05, 1.225);
      \draw (-0.35, 1.05) -- (-0.7, 1.75);
	  \fill (-0.7, 1.75) circle[radius=0.065];
      \draw (-0.35, 1.05) -- (-0.175, 1.575);
	  \fill (-0.175, 1.575) circle[radius=0.065];
	  \draw (-0.175, 1.575) -- (-0.35, 2.275);
      \draw (-0.175, 1.575) -- (0.0, 2.275);
      \draw (0.0, 0.595) -- (0.525, 2.1);
      \draw (0.0, 0.595) -- (0.77, 1.225);
	  \fill (0.77, 1.225) circle[radius=0.065];
    \end{scope}

	\draw (1.4, 1.4) node {$\bigleftbrace{30}$};
    \draw (4.2, 1.4) node {$\bigrightbrace{30}$};

	\begin{scope}[shift={(-4.9, -2.45)}] 
      \draw (0.0, 0.175) -- (0.0, 0.595);
	  \fill (0.0, 0.595) circle[radius=0.065];
	  \draw (0.0, 0.595) -- (-0.42, 0.945);
      \draw (0.0, 0.595) -- (0.0, 1.12);
      \draw (0.0, 0.595) -- (0.42, 0.945);
    \end{scope}
	
    \draw (-5.6, -1.75) node {$\bigleftbrace{15}$};
    \draw (-4.2, -1.75) node {$\bigrightbrace{15}$};
    
	\begin{scope}[shift={(4.9, -2.45)}] 
      \draw (0.0, 0.175) -- (0.0, 0.595);
	  \fill (0.0, 0.595) circle[radius=0.065];
	  \draw (0.0, 0.595) -- (-0.42, 0.945);
      \draw (0.0, 0.595) -- (-0.175, 1.05);
      \draw (0.0, 0.595) -- (0.175, 1.05);
      \draw (0.0, 0.595) -- (0.42, 0.945);
    \end{scope}
	
    \draw (4.2, -1.75) node {$\bigleftbrace{15}$};
    \draw (5.6, -1.75) node {$\bigrightbrace{15}$};   
    \draw (0.0, -1.925) node {$\PPP\bd$};
    \draw[->] (-0.525, 1.4) -- (0.525, 1.4);
    \draw[->] (3.85, 0.0) -- (4.375, -0.805);
    \draw[->] (-3.85, 0.0) -- (-4.375, -0.805);
	  \draw (0.0, 1.7) node {\scriptsize forget mark};
	  \draw (5.3, -0.3) node {\scriptsize monad mult.};
	  \draw (-5.3, -0.3) node {\scriptsize marked node};

  \end{tikzpicture}
  \]
(The trees in the picture are $\PPP$-trees, but the $\PPP$-decorations 
  have been suppressed to avoid clutter.)  

\begin{blanko}{Blobbed trees.}\label{blobbing}
  An operation of the endofunctor $\PPP\bd \circ \PPP\bd$ is a $\PPP$-tree
  $K$ whose nodes are decorated by $\PPP$-trees $R_i$ in a compatible way:
  the tree that decorates a node with local structure $b\in B$ must have
  $b$ as its residue~\ref{residue} (note that the notion of residue for
  $\PPP$-trees involves the monad multiplication of $\PPP$, to compose
  the whole tree configuration of operations to a single operation).
  An operation of $\PPP\bd\circ\PPP\bd$ is thus an active
  map of $\PPP$-trees
  $$
  K \actto T ,
  $$
  mapping each node in $K$ to a subtree $R_i \subset T$, as described 
  in \ref{active}.
  A pictorial way of encoding this is to draw each node of $K$ as a big blob and
  draw its decorating tree inside the blob, so that its leaves
  and root match the boundary of the blob. Altogether the configuration can
  be seen as a {\em blobbed $\PPP$-tree} (called {\em constellation} in 
  \cite{Kock-Joyal-Batanin-Mascari:0706}) ---
  a $\PPP$-tree $T$ 
  with some blobs partitioning the tree into smaller trees:
  
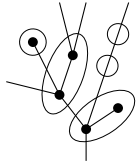
\begin{equation}
	\begin{tikzpicture}
	  \begin{scope} 
		\draw (0.0, 0.175) -- (0.0, 0.595);
		\fill (0.0, 0.595) circle[radius=0.065];
		\draw (0.0, 0.595) -- (-0.35, 1.05);
		\fill (-0.35, 1.05) circle[radius=0.065];
		\draw (-0.35, 1.05) -- (-1.05, 1.225);
		\draw (-0.35, 1.05) -- (-0.7, 1.75);
		\fill (-0.7, 1.75) circle[radius=0.065];
		\draw (-0.35, 1.05) -- (-0.175, 1.575);
		\fill (-0.175, 1.575) circle[radius=0.065];
		\draw (-0.175, 1.575) -- (-0.35, 2.275);
		\draw (-0.175, 1.575) -- (0.0, 2.275);
		\draw (0.0, 0.595) -- (0.56, 2.275);
		\draw (0.0, 0.595) -- (0.42, 0.875);
		\fill (0.42, 0.875) circle[radius=0.065];
		\draw[line width=0.01, rotate around={70:(-0.28,1.33)}] (-0.28,1.33)  ellipse (0.56 and 0.245);
		\draw[line width=0.01, rotate around={34:(0.21,0.77)}] (0.21,0.77)  ellipse (0.49 and 0.245);
		\draw[line width=0.01, rotate around={70:(-0.7,1.75)}] (-0.7,1.75)  ellipse (0.175 and 0.175);
		\draw[line width=0.01, rotate around={70:(0.28,1.435)}] (0.28,1.435)  ellipse (0.14 and 0.14);
		\draw[line width=0.01, rotate around={70:(0.42,1.855)}] (0.42,1.855)  ellipse (0.14 and 0.14);
	  \end{scope}
	\end{tikzpicture}
  \end{equation}
  Each blob contains a tree, and each node is contained in precisely
  one blob. It must be stressed that there may be blobs containing
  only a trivial tree. This occurs naturally because the tree $K$ may
  contain unary nodes \inlineonetree , and the decorating tree of such
  a unary node may be the trivial tree \inlineDotlessTree . For
  details, see \cite{Kock-Joyal-Batanin-Mascari:0706}. Let us stress
  that $\inlineDotlessTree$ and $\inlineblobtriv$ are examples of blobbed trees,
  corresponding, respectively, to the active maps
  $\inlineDotlessTree\!\actto \inlineDotlessTree$ and $\inlineonetree
  \!\actto \inlineDotlessTree$. The combinatorial description is
  attractive and useful for grasping constructions, but for the
  proofs below, it is the active maps $K \actto T \to \PPP$ we
  shall actually work with.
\end{blanko}

\begin{blanko}{Monad structure on $\PPP\bd$ and operad interpretation.}
  The monad structure is canonically given by the adjunction. 
  The monad multiplication $\mu:
  \PPP\bd\circ \PPP\bd \Rightarrow \PPP\bd$ simply takes the tree of
  $\PPP$-trees and returns the total $\PPP$-tree obtained by gluing
  together all the decorating $\PPP$-trees $R_i$ according to the recipe
  provided by $K$, to get a big $\PPP$-tree $T$. In terms of blobbed trees,
  the monad multiplication consists in erasing the blobs (not their
  content) and retaining the total $\PPP$-tree $T$. The unit for the monad
  $\PPP\bd$ is given by regarding a $\PPP$-corolla as a $\PPP$-tree.

  We consider $\PPP\bd$ as a coloured operad, in the sense of
  \ref{ouroperads}: its colours are the operations of $\PPP$. The
  operations of $\PPP\bd$ are the $\PPP$-trees, and the arity of such
  a $\PPP$-tree is its set of nodes.
  One can substitute a whole $\PPP$-tree (that is, a
  $\PPP\bd$-operation) into a node of another $\PPP$-tree if the
  residue of the tree matches the local structure of the node, as 
  pictured here:
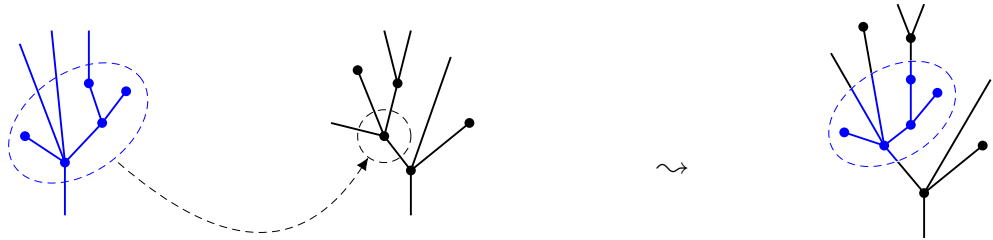
\begin{equation*}\label{BD-subst}
  \begin{tikzpicture}[line width=0.25mm]
    \begin{scope}[shift={(0.0, 0.0)}, blue]
	  \draw (0.0, 0.0) -- (0.0, 0.7);
	  \fill (0.0, 0.7) circle[radius=0.065];
	  \draw (0.0, 0.7) -- (-0.525, 1.05);
	  \fill (-0.525, 1.05) circle[radius=0.065];
	  \draw (0.0, 0.7) -- (-0.595, 2.275);
	  \draw (0.0, 0.7) -- (-0.175, 2.45);
	  \draw (0.0, 0.7) -- (0.49, 1.225);
	  \fill (0.49, 1.225) circle[radius=0.065];
	  \draw (0.49, 1.225) -- (0.315, 1.75);
	  \fill (0.315, 1.75) circle[radius=0.065];
	  \draw (0.315, 1.75) -- (0.315, 2.45);
	  \draw (0.49, 1.225) -- (0.805, 1.645);
	  \fill (0.805, 1.645) circle[radius=0.065];
	  \draw[line width=0.01, densely dashed, rotate around={35:(0.175,1.225)}] 
	    (0.175,1.225)  ellipse (1.015 and 0.665);
	  \draw[-{Latex[width=3.5pt,length=4pt]}
	  , line width=0.02, densely dashed, black]
	  (0.7, 0.7) .. controls (2.1, -0.525) and (3.15, -0.525) .. (3.98, 0.75);
    \end{scope}
    \begin{scope}[shift={(4.55, 0.0)}]
	  \draw (0.0, 0.0) -- (0.0, 0.595);
	  \fill (0.0, 0.595) circle[radius=0.065];
	  \draw (0.0, 0.595) -- (-0.35, 1.05);
	  \fill (-0.35, 1.05) circle[radius=0.065];
	  \draw (-0.35, 1.05) -- (-1.05, 1.225);
	  \draw (-0.35, 1.05) -- (-0.7, 1.925);
	  \fill (-0.7, 1.925) circle[radius=0.065];
	  \draw (-0.35, 1.05) -- (-0.175, 1.75);
	  \fill (-0.175, 1.75) circle[radius=0.065];
	  \draw (-0.175, 1.75) -- (-0.35, 2.45);
	  \draw (-0.175, 1.75) -- (0.0, 2.45);
	  \draw (0.0, 0.595) -- (0.525, 2.1);
	  \draw (0.0, 0.595) -- (0.77, 1.225);
	  \fill (0.77, 1.225) circle[radius=0.065];
	  \draw[line width=0.01, densely dashed] (-0.35, 1.05) circle (0.35);
	\end{scope}
    
	\draw (8.0, 0.6) node {\normalsize $\leadsto$};
	
	\begin{scope}[shift={(11.3, -0.3)}]
	  \scriptsize
	  \draw (0.0, 0.0) -- (0.0, 0.595);
	  \fill (0.0, 0.595) circle[radius=0.065];
	  \draw (0.0, 0.595) -- (-0.525, 1.225);
	  \fill (-0.525, 1.225) circle[radius=0.065];
	  \draw (-0.525, 1.225) -- (-1.225, 2.45); 
	  \draw (-0.525, 1.225) -- (-1.05, 1.4);
	  \fill (-1.05, 1.4) circle[radius=0.065];
	  \draw (-0.525, 1.225) -- (-0.8, 2.8);
	  \fill (-0.8, 2.8) circle[radius=0.065];
	  \draw (-0.525, 1.225) -- (-0.175, 1.5);
	  \fill (-0.175, 1.5) circle[radius=0.065];
	  \draw (-0.175, 1.5) -- (-0.175, 2.1);
	  \fill (-0.175, 2.1) circle[radius=0.065];
	  \draw (-0.175, 2.1) -- (-0.175, 2.65);
	  \fill (-0.175, 2.65) circle[radius=0.065];
	  \draw (-0.175, 2.65) -- (-0.35, 3.1); 
	  \draw (-0.175, 2.65) -- (0.0, 3.1);
	  \draw (-0.175, 1.5) -- (0.175, 1.925);
	  \fill (0.175, 1.925) circle[radius=0.065];
	  \draw (0.0, 0.595) -- (0.875, 2.1);
	  \draw (0.0, 0.595) -- (0.77, 1.225);
	  \fill (0.77, 1.225) circle[radius=0.065];
	  
	  \draw[fill, white, rotate around={32:(-0.42,1.645)}]
	    (-0.42,1.645) ellipse (0.91 and 0.595);
	  \draw[blue, line width=0.01,densely dashed, rotate around={32:(-0.42,1.645)}]
	    (-0.42,1.645) ellipse (0.91 and 0.595);
	  \clip[rotate around={32:(-0.42,1.645)}]
	    (-0.42,1.645) ellipse (0.91 and 0.595);
	  \draw[blue] (0.0, 0.0) -- (0.0, 0.595);
	  \fill[blue] (0.0, 0.595) circle[radius=0.065];
	  \draw[blue] (0.0, 0.595) -- (-0.525, 1.225);
	  \fill[blue] (-0.525, 1.225) circle[radius=0.065];
	  \draw[blue] (-0.525, 1.225) -- (-1.225, 2.45); 
	  \draw[blue] (-0.525, 1.225) -- (-1.05, 1.4);
	  \fill[blue] (-1.05, 1.4) circle[radius=0.065];
	  \draw[blue] (-0.525, 1.225) -- (-0.8, 2.8);
	  \fill[blue] (-0.8, 2.8) circle[radius=0.065];
	  \draw[blue] (-0.525, 1.225) -- (-0.175, 1.5);
	  \fill[blue] (-0.175, 1.5) circle[radius=0.065];
	  \draw[blue] (-0.175, 1.5) -- (-0.175, 2.1);
	  \fill[blue] (-0.175, 2.1) circle[radius=0.065];
	  \draw[blue] (-0.175, 2.1) -- (-0.175, 2.65);
	  \fill[blue] (-0.175, 2.65) circle[radius=0.065];
	  \draw[blue] (-0.175, 2.65) -- (-0.35, 3.1); 
	  \draw[blue] (-0.175, 2.65) -- (0.0, 3.1);
	  \draw[blue] (-0.175, 1.5) -- (0.175, 1.925);
	  \fill[blue] (0.175, 1.925) circle[radius=0.065];
	  \draw[blue] (0.0, 0.595) -- (0.875, 2.1);
	  \draw[blue] (0.0, 0.595) -- (0.77, 1.225);
	  \fill[blue] (0.77, 1.225) circle[radius=0.065];
      
	\end{scope}

  \end{tikzpicture}
\end{equation*}
  
  The neutral operation
  (of colour $b \in B$) is the tree consisting of just that corolla $b$.

  Note that 
  for each original
  colour $i\in I$, the trivial tree of colour $i$ is a nullary operation of
  $\PPP\bd$.
\end{blanko}

\begin{blanko}{Example.}\label{ex:Id}
  Let $\PPP$ be the identity monad $\Id: \Grpd \to \Grpd$, which is
  represented by the diagram
  $$
  1 \leftarrow \underline 1 \to \underline 1 \to 1 .
  $$
  The ones are all singleton; the two middle ones have been underlined
  just to stress that the play a different role than the non-underlined.
  According to the general graphical interpretation (\ref{graphical}), we
  picture the unique operation (unique element in $B=\underline 1$) as a
  \inlineonetree . The free monad on $\Id$ is represented by
  $$
  1 \leftarrow \N \stackrel=\to \N \to 1 ,
  $$
  where $\N$ is regarded as the set of linear trees (including the 
  trivial tree, corresponding to $0\in \N$).  The Baez--Dolan 
  construction $\Id\bd$ is instead represented by
  $$
  \underline 1 \leftarrow \N^\bullet \to \N \to \underline 1 .
  $$
  Here $\N^\bullet$ denotes the set of (iso-classes of) linear trees 
  with a marked node.  The monad multiplication is given by 
  substituting linear trees into the nodes of linear trees.
  Since $\Id\bd$ has one operation in each degree, it is easily seen that
  $\Id\bd$ is the free-monoid monad from Example~\ref{ex:M}, which is
  thus exhibited as the Baez--Dolan construction on the identity 
  monad. We shall return to this example in \ref{sub:FdB}, where there is 
  also a picture.
\end{blanko}

\begin{blanko}{Example (continued from Example~\ref{ex:M}).}
  Let $\PPP=\MMM$ be the free-monoid monad on $\Grpd$, $X\mapsto X\upperstar$,
  represented by 
  $$
  1 \leftarrow \N'\to \N \to 1 .
  $$
  We regard $\N$ as the set of corollas (one for each arity
  $n\in \N$) and $\N'$ as the set of corollas with a marked leaf. The
  monad law now deals with grafting of corollas onto leaves.
  (This is a different graphical interpretation than the one just
  produced in the previous example with $\MMM = 
  \Id\bd$.)
  
  The free monad on $\MMM$ is represented by
  $$
  1 \leftarrow \tr'(\MMM) \to \tr(\MMM) \to 1  ,
  $$
  having as operations the set of all $\MMM$-trees, which means planar trees.
  The Baez--Dolan construction $\MMM\bd$ is represented by
  $$
  \N \leftarrow \tr^\bullet(\MMM) \to \tr(\MMM) \to \N ,
  $$
  where $\tr^\bullet(\MMM)$ is the set of planar trees with a marked
  node. The monad multiplication is given by substituting trees into
  nodes. This monad $\MMM\bd$ is the free-nonsymmetric-single-coloured-operad
  monad~\cite{Leinster:0305049}.
\end{blanko}

\begin{blanko}{Aside: opetopes.}\label{opetopes}
  The original motivation \cite{Baez-Dolan:9702} for the
  Baez--Dolan construction was to iterate it and use it to define the
  {\em opetopes}. These are shapes parametrising higher-dimensional
  many-in/one-out operations, in turn devised by Baez and Dolan to define
  weak $n$-categories for all $n$
  (see
  Leinster's book~\cite{Leinster:0305049} for a detailed development of the
  theory).
  
  The $n$-dimensional opetopes are by
  definition the colours of the $n$th iteration of the Baez--Dolan
  construction starting with $\Id$:
  $$
  \operatorname{Op}_0 = 1 \qquad
  \operatorname{Op}_1 = \underline 1 \qquad
  \operatorname{Op}_2 = \N \qquad 
  \operatorname{Op}_3 = \tr(\Id\bd) = \tr(\MMM) ,
  $$
  the set of planar trees. In dimension $4$ one finds the set of blobbed
  planar trees, and after that it becomes increasing complicated. An
  elementary combinatorial description was given in
  \cite{Kock-Joyal-Batanin-Mascari:0706} in terms of something called zoom
  complexes, certain trees of trees of trees.
\end{blanko}

\subsection{Two-sided bar construction (of one operad relative to another)}

\label{sub:bar}

The two-sided bar construction is a classical construction in algebraic
topology and homological algebra (see \cite{May:LNM271}), where it serves,
among many other things, to construct classifying spaces, approximations
and resolutions, and deloopings. In $2$-dimensional category theory it
serves to model the PROP envelope of an operad, and more generally internal
algebra classifiers via codescent objects~\cite{Weber:1503.07585}; see also
\cite{Batanin-Berger:1305.0086}.

The relative polynomial version \cite{Batanin:0207281}, 
\cite{Batanin-Berger:1305.0086}, \cite{Weber:1503.07585} used here 
concerns the situation where one
polynomial monad is cartesian over another. The most important case is that
of polynomial monads cartesian over $\SSS$, the
free-symmetric-monoidal-category monad \ref{S}, because these are
essentially symmetric operads. We describe the constructions in that
case, but $\SSS$ could be replaced by any other polynomial monad (as we
shall do in Subsection~\ref{sec:furtherbar}).

\begin{blanko}{Operad morphisms (monad opfunctors).}\label{opf-ax}
  The natural notion of morphism of coloured operads does not fix the
  colours, and therefore at the monad level must be a bit more
  involved than just natural transformations.
  To say that a monad $\RRR: \Grpd_{/I} \to \Grpd_{/I}$ is {\em cartesian} over 
  $\SSS$ means there is a diagram of polynomial monads
  $$    \begin{tikzcd}
    I \ar[d, "F"']& \ar[l]  E\ar[d] \drpullback \ar[r] & B\ar[d] \ar[r] & 
    I\ar[d, "F"] \\
    1  &\ar[l] \B'\ar[r] & \B \ar[r]  &1 ,
  \end{tikzcd}
  $$
  intertwining the  two monads $\RRR$ and $\SSS$ by means of the functor
  $F\lowershriek$ and a natural transformation
  $$
  \phi : F\lowershriek \RRR \Rightarrow \SSS F\lowershriek  ,
  $$
  forming together a monad opfunctor in the sense of 
  Street~\cite{Street:formal-monads}, a monad adjunction in the sense of 
  Weber~\cite{Weber:1503.07585},
  or a vertical morphism of horizontal monads in the double-category 
  setting~\cite{Fiore-Gambino-Kock:1006.0797}.
  This means we have the following two compatibilities with the monad structures
  (monad opfunctor structure on $F\lowershriek$): for any 
  object $A$ in $\Grpd_{/I}$ we have commutative diagrams
  \begin{equation}\label{eq:colax}
  \begin{tikzcd}[column sep = {2.5em,between origins}]
	 F\lowershriek A \ar[d, "{F\lowershriek \eta^{\RRR}_A}"'] 
	 \ar[rr, equal]&&
	 F\lowershriek A
	 \ar[d, "{\eta^{\SSS}_{F\lowershriek A}}"]  \\
	 F\lowershriek \RRR A \ar[rr, "{\phi_A}"'] && \SSS F\lowershriek A
  \end{tikzcd}
  \qquad\qquad
  \begin{tikzcd}[column sep = {3em,between origins}]
	 F\lowershriek \RRR\RRR A \ar[d, "{F\lowershriek \mu^{\RRR}_A}"'] 
	 \ar[rr, "{\phi_{\RRR A}}"] && \SSS F\lowershriek \RRR A \ar[rr, "{\SSS 
	 \phi_A}"] 
	 &&\SSS\SSS F\lowershriek A \ar[d, "{\mu^{\SSS}_{F\lowershriek A}}"] \\
	F\lowershriek \RRR A \ar[rrrr, "{\phi_A}"'] &&&& \SSS F\lowershriek A  .
  \end{tikzcd}
  \end{equation}

  It is a general fact that in the polynomial setting the
  natural transformation $\phi$ is cartesian.  This follows
  because its ingredients are the unit and counit of lowershriek-upperstar 
  adjunctions and an instance of the Beck--Chevalley isomorphism.  See 
  \cite[\S~3.3]{Weber:1503.07585} for details.
\end{blanko}

\begin{blanko}{One-sided bar construction.}\label{one-sided}
  Let $1 := \id_I {:} I {\to} I$ denote the terminal object in
  $\Grpd_{/I}$, and let $\alpha: \RRR 1 \to 1$ denote the unique map from
  $\RRR 1$ in $\Grpd_{/I}$. (Note that $\alpha$ has underlying map of
  groupoids $B \to I$.) The pair $(1,\alpha)$ is the terminal
  $\RRR$-algebra.

  In $\Grpd_{/I}$ there is induced a natural 
  simplicial-object-with-missing-top-face-maps
  \begin{equation}\label{eq:R1}
  \begin{tikzcd}[column sep = 5em]
  1
  \ar[r, pos=0.65, "s_0" on top] 
  &
  \ar[l, phantom, shift right=6pt, "\dots\dots" on top]
  \ar[l, shift left=6pt, pos=0.65, "d_0" on top]
  \RRR 1
  \ar[r, shift right=6pt, pos=0.65, "s_0" on top]
  \ar[r, shift left=6pt, pos=0.65, "s_1" on top]  
  &
  \ar[l, phantom, shift right=12pt, "\dots\dots" on top]
  \ar[l, pos=0.65, "d_1" on top]
  \ar[l, shift left=12pt, pos=0.65, "d_0" on top]
  \RRR\RRR 1
  \ar[r, shift left=12pt, pos=0.65, "s_2" on top]
  \ar[r, pos=0.65, "s_1" on top]
  \ar[r, shift right=12pt, pos=0.65, "s_0" on top]
  &
  \ar[l, phantom, shift right=18pt, "\dots\dots" on top]
  \ar[l, shift right=6pt, pos=0.65, "d_2" on top]
  \ar[l, shift left=6pt, pos=0.65, "d_1" on top]
  \ar[l, shift left=18pt, pos=0.65, "d_0" on top]
  \RRR\RRR\RRR 1
    \ar[r, phantom, "\dots" on top]
  & {}
  \end{tikzcd}
  \end{equation}
  The bottom face maps come from the action $\alpha: \RRR 1 \to 1$,
  and the remaining face and degeneracy maps come from the monad structure.
  The lowest-degree maps are 
  \[
  \begin{tikzcd}[column sep = {8em,between origins}]
  1
  \ar[r, pos=0.65, "\eta_1" on top] 
  &
  \ar[l, phantom, shift right=6pt, "\dots\dots" on top]
  \ar[l, shift left=6pt, pos=0.65, "\alpha" on top]
  \RRR 1
  \ar[r, shift right=6pt, pos=0.65, "\RRR\eta_1" on top]
  \ar[r, shift left=6pt, pos=0.65, "\eta_{\RRR 1}" on top]  
  &
  \ar[l, phantom, shift right=12pt, "\dots\dots" on top]
  \ar[l, pos=0.65, "\mu_1" on top]
  \ar[l, shift left=12pt, pos=0.65, "\RRR \alpha" on top]
  \RRR\RRR 1
    \ar[r, phantom, "\dots" on top]
  & {}
  \end{tikzcd}
  \]
  and in general, $s_k  : \RRR^n 1 \to  \RRR^{n+1}1$  is given by $\RRR^{n-k} 
  \eta_{\RRR^k 1}$ and
  $d_k  : \RRR^{n+1}1 \to \RRR^n 1$ is given by $\RRR^{n-k} \mu_{\RRR^{k-1} 1}$, 
  with the convention that
  $\mu_{\RRR^{-1}1} = \alpha$.
\end{blanko}
 
\begin{blanko}{Two-sided bar construction.}
  We now apply $\SSS F\lowershriek$ to the diagram~\eqref{eq:R1} above, to
  obtain inside $\Grpd$ a genuine simplicial object: the diagram now
  acquires the missing top face maps (written 
  \begin{tikzcd}[cramped] {} & \ar[l, wavy] {}\end{tikzcd}), 
  and constitutes altogether a simplicial object in $\Grpd$
  (\cite{Weber:1503.07585}, Lemma 4.3.2) which is the {\em two-sided 
  bar construction} on $\RRR$, denoted
  $\TSB{\SSS}{\RRR}$:
  \[\begin{tikzcd}[column sep = 5em]
  \ar[r, phantom, "\TSB{\SSS}{\RRR} :" description] &
  \SSS  F\lowershriek 1 
  \ar[r, pos=0.65, "s_0" on top] 
  &
  \ar[l, shift left=6pt, pos=0.65, "d_0" on top]
  \ar[l, wavy, shift right=6pt, pos=0.65, "d_1" on top]
  \SSS  F\lowershriek \RRR 1
  \ar[r, shift right=6pt, pos=0.65, "s_0" on top]
  \ar[r, shift left=6pt, pos=0.65, "s_1" on top]  
  &
  \ar[l, shift left=12pt, pos=0.65, "d_0" on top]
  \ar[l, wavy, shift right=12pt, pos=0.65, "d_2" on top]
  \ar[l, pos=0.65, "d_1" on top]
  \SSS  F\lowershriek \RRR\RRR 1
  \ar[r, shift left=12pt, pos=0.65, "s_2" on top]
  \ar[r, pos=0.65, "s_1" on top]
  \ar[r, shift right=12pt, pos=0.65, "s_0" on top]
  &
  \ar[l, wavy, shift right=18pt, pos=0.65, "d_3" on top]
  \ar[l, shift right=6pt, pos=0.65, "d_2" on top]
  \ar[l, shift left=6pt, pos=0.65, "d_1" on top]
  \ar[l, shift left=18pt, pos=0.65, "d_0" on top]
  \SSS  F\lowershriek  \RRR\RRR\RRR  1 
  \ar[r, phantom, "\dots" on top] & {}
  \end{tikzcd}\]
  The new top face maps are given by 
  $$
  d_\top := \mu^\SSS \circ \SSS(\phi) .
  $$
  For example, the lowest-degree top face map is given as
  \begin{equation}\label{eq:dtop}
	\begin{tikzcd}[sep = 2em]
      & \ar[ld, "{\mu^S_{F\lowershriek 1}}"'] \SSS\SSS F\lowershriek 1 & \\
    \SSS F\lowershriek 1 && \ar[ll, wavy, "{d_1}"] \ar[lu, "{\SSS(\phi_1)}"']
    \SSS F\lowershriek \RRR 1 .
  \end{tikzcd}
  \end{equation}
\end{blanko}
  
\begin{blanko}{Remarks.}  
  Note that $F\lowershriek$ only serves to pass from $\Grpd_{/I}$ to
  $\Grpd$, so that the simplicial object really lives in the category
  $\Grpd$. It should further be noted that the terminal object in
  $\Grpd_{/I}$ (as always denoted $1$) is the identity map $\id:I \to I$,
  so that we have
  $$
  F\lowershriek 1 = I \qquad \text{ and } \qquad 
  F\lowershriek \RRR 1 = B .
  $$
  Furthermore, $F\lowershriek \RRR\RRR 1$ is the groupoid of $2$-level
  $\RRR$-trees.
  
  The $\SSS$ in front of everything means we are talking about
  monomials of objects. Monomials of degree $1$ (that is, original
  objects) are called {\em connected}. In the cases of interest, the
  objects will be various kinds of trees, and it makes sense also to
  say {\em forest} for monomials (disjoint unions) of trees. All the
  face and degeneracy maps --- except the top face maps --- are $\SSS$
  of a map, meaning that they send connected objects to connected
  objects. In contrast, the top face maps generally send connected
  objects to non-connected ones, which is what the wavy arrows
  indicate.
  
  It should be noted that the standard notation in the literature for this 
  two-sided bar construction is
  $$
  B(\SSS F\lowershriek, \RRR, 1) .
  $$
\end{blanko}

\begin{blanko}{Interpretation of the groupoids.}
  The two-sided bar construction is completely formal. At the same time,
  all the constituents of $X = \TSB{\SSS}{\RRR}$ have clean combinatorial
  interpretations. For the groupoids:
  
  \begin{quote} \em $X_k = \SSS \tr_k(\RRR)$ is the groupoid of
  monomials of $k$-level $\RRR$-trees.
  \end{quote}
  In particular, 
  
  $\bullet$ \
  $X_0 = \SSS \tr_0(\RRR) = \SSS I$ is the groupoid of monomials of 
  colours;
  
  $\bullet$ \
  $X_1 = \SSS \tr_1(\RRR) = \SSS\operatorname{cor}(\RRR) = \SSS B$ is the groupoid of
  monomials of operations of $\RRR$ ;
  
  $\bullet$ \
  $X_2 = \SSS \tr_2(\RRR) = \SSS \RRR B$ is the groupoid of monomials 
  of $2$-level $\RRR$-trees.
  
  As to the maps, they are all about joining or deleting levels (or inserting
  trivial levels), using the monad structure. In the lowest degree
  \begin{tikzcd}[column sep={3.8em}] X_0 \ar[r, pos=0.65, "s_0" on top] & \ar[l, shift
  left=6pt, pos=0.65, "d_0" on top] \ar[l, wavy, shift right=6pt,
  pos=0.65, "d_1" on top] X_1 \end{tikzcd}, 
  
  $\bullet$ \ 
  the bottom face map $d_0$
  returns the root edges (of a monomial of corollas);
  
  $\bullet$ \
  the top face map
  $d_1$ returns the
  monomial of leaves. 
  
  
  $\bullet$ \
  The degeneracy map $s_0$ sends a colour $i$ to the identity operation on
  $i$.
  
  \noindent 
  In the next degree, \begin{tikzcd}[column sep={3.8em}]
    X_1
  \ar[r, shift right=6pt, pos=0.65, "s_0" on top]
  \ar[r, shift left=6pt, pos=0.65, "s_1" on top]  
  &
  \ar[l, shift left=12pt, pos=0.65, "d_0" on top]
  \ar[l, wavy, shift right=12pt, pos=0.65, "d_2" on top]
  \ar[l, pos=0.65, "d_1" on top]
  X_2
\end{tikzcd}
  
  $\bullet$ \
  $d_0$
  returns the bottom level of nodes; 
  
  $\bullet$ \ 
  $d_1$ composes the $2$-level forest to
  a $1$-level forest using monad multiplication;
  
  $\bullet$ \ 
  $d_2$
  returns the monomial of corollas resulting from deleting the bottom
  nodes;
  
  $\bullet$ \ 
  $s_0$ sends a corolla to the $2$-level tree obtained by 
  grafting identity corollas onto all leaves;
  
  $\bullet$ \
  $s_1$ sends a corolla to 
  the $2$-level tree obtained by grafting that corolla onto a single 
  identity corolla. 
  
  And so on.
  The beginning of the simplicial groupoid $\TSB{\SSS}{\RRR}$ can thus be 
  pictured like this:

\begin{equation}\label{twosidedpicture}
  \begin{tikzpicture}[line width=0.25mm]

  \begin{scope}[shift={(0.525, 0.42)}] 
	\draw (0.0, 0.525) -- (0.0, 1.225);
	\draw (-0.91, 0.875) node {$\SSS$};
	\draw (-0.56, 0.875) node {$\bigleftbrace{15}$};
	\draw (0.56, 0.875) node {$\bigrightbrace{15}$};
  \end{scope}

  \begin{scope}[shift={(0.98, 0.0)}]
    \draw[->, wavy] (2.1, 1.645) -- +(-1.4, 0.0);
    \draw[->] (2.1, 0.945) -- +(-1.4, 0.0);
	\draw (1.4, 1.86) node {\scriptsize leaves};
	\draw (1.4, 1.16) node {\scriptsize root};
  \end{scope}
	
  \begin{scope}[shift={(5.075, 0.105)}] 
	\draw (0.0, 0.525) -- (0.0, 1.05);
	\fill (0.0, 1.05) circle[radius=0.065];
	\draw (0.0, 1.05) -- (-0.63, 1.5); 
	\draw (0.0, 1.05) -- (-0.33, 1.7);
	\draw (0.0, 1.05) -- (0.0, 1.8);
	\draw (0.0, 1.05) -- (0.33, 1.7);
	\draw (0.0, 1.05) -- (0.63, 1.5);
	\draw (-1.33, 1.225) node {$\SSS$};
	\draw (-0.98, 1.225) node {$\bigleftbrace{20}$};
	\draw (0.98, 1.225) node {$\bigrightbrace{20}$};
  \end{scope}

  \begin{scope}[shift={(6.335, 0.0)}]
    \draw[->, wavy] (2.1, 1.995) -- +(-1.4, 0.0);
    \draw[->] (2.1, 1.295) -- +(-1.4, 0.0);
    \draw[->] (2.1, 0.595) -- +(-1.4, 0.0);
	\draw (1.4, 2.21) node {\scriptsize delete bottom corolla};
	\draw (1.4, 1.47) node {\scriptsize compose};
	\draw (1.4, 0.82) node {\scriptsize delete top forest};
  \end{scope}

  \begin{scope}[shift={(10.815, 0.28)}] 
	\draw (0.0, 0.28) -- (0.0, 0.7);
	\fill (0.0, 0.7) circle[radius=0.065];
	\draw (0.0, 0.7) -- (-0.7, 1.4);
	\fill (-0.7, 1.4) circle[radius=0.065];
	\draw (-0.7, 1.4) -- (-0.91, 1.82);
	\draw (-0.7, 1.4) -- (-0.49, 1.82);
	\draw (0.0, 0.7) -- (0.0, 1.4);
	\fill (0.0, 1.4) circle[radius=0.065];
	\draw (0.0, 0.7) -- (0.7, 1.4);
	\fill (0.7, 1.4) circle[radius=0.065];
	\draw (0.7, 1.4) -- (0.28, 1.82);
	\draw (0.7, 1.4) -- (0.7, 1.82);
	\draw (0.7, 1.4) -- (1.12, 1.82);	
	\draw (-0.91, 1.05) -- (0.91, 1.05);
	\draw (-1.61, 1.05) node {$\SSS$};
	\draw (-1.225, 1.05) node {$\bigleftbrace{26}$};
	\draw (1.4, 1.05) node {$\bigrightbrace{26}$};
  \end{scope}

    \begin{scope}[shift={(13, 0.28)}] 
	  	\draw (1.4, 1.05) node {$\dots$};

  \end{scope}

\end{tikzpicture}\end{equation}
  (The corollas in the picture are $\RRR$-operations.
The degeneracy maps are not rendered in the picture, to avoid clutter.)
\end{blanko}

\begin{prop}[Weber~\cite{Weber:1503.07585}, Prop.~4.4.1]\label{ho-pbk}
  For any operad $\RRR$, the simplicial groupoid $\TSB{\SSS}{\RRR}$ is a
  strict category object and a Segal space.
\end{prop}
 
\begin{prop}[\cite{Kock-Weber:1609.03276}, Prop.~3.7]
  For any operad $\RRR$, the two-sided bar construction
  $X=\TSB{\SSS}{\RRR}$ is an $\SSS$-algebra in Segal spaces. Furthermore,
  the structure map $\SSS X \to X$ is culf, and hence al\-together $X$ is a
  symmetric monoidal decomposition space.
\end{prop}
\noindent
Note that $\SSS$ preserves (both strict and) homotopy pullbacks, so if 
$X$ is Segal then so is $\SSS X$.
  The $\SSS$-algebra structure is a consequence of the fact 
  that degree-wise, $X$ is $\SSS$ of something. The $\SSS$-algebra 
  structure $\SSS X \to X$ is given in degree $k$ by monad multiplication
  $
  \begin{tikzcd}
	\SSS F\lowershriek \RRR^k 1 \ar[r, "\mu^{\SSS}_{F\lowershriek 
	\RRR^k 1}"']
	&
  \SSS \SSS F\lowershriek \RRR^k 1 .
  \end{tikzcd}
  $
  Culfness is a consequence of the fact that $\SSS$ is cartesian.

\begin{blanko}{Incidence bialgebra.}
  Since the two-sided bar construction $\TSB{\SSS}{\RRR}$ is a Segal
  groupoid, and in particular a decomposition
  space~\cite{Galvez-Kock-Tonks:1512.07573}, one can now apply the
  incidence coalgebra construction (\ref{incidence}) to obtain a coalgebra
  structure on the slice $\Grpd_{/X_1}$, with comultiplication given by the
  span
  $$
  \begin{tikzcd}
	X_1 & \ar[l, "d_1"'] X_2 \ar[r, "{(d_2,d_0)}"] & X_1 \times X_1  .
  \end{tikzcd}
  $$
  Recall that $X_1$ is the groupoid of monomials of operations of $\RRR$,
  and $X_2$ is groupoid of monomials of $2$-level trees in $\RRR$. The
  symmetric monoidal structure is disjoint union. Thanks to its culfness
  property, the induced multiplication functor is compatible with the
  comultiplication so as to give altogether a
  bialgebra~\cite{Galvez-Kock-Tonks:1512.07573}.

  Shortly we shall impose the finiteness conditions  necessary to take 
  homotopy cardinality (\ref{fin-cond}), which will finally give an ordinary 
  bialgebra in $\Q$-vector spaces,
  where an $\RRR$-operation $r$ is comultiplied as
$$
  \Delta(r) = \sum_{r=b \circ (a_1,\ldots,a_n)}  a_1\cdots a_n \tensor b  .
$$  
  (Because of the bialgebra structure, it is enough to specify the
  comultiplication on connected elements, i.e.~the operations themselves,
  rather than monomials of operations.)
\end{blanko}

\subsection{Two-sided bar construction on $\PPP \upperstar$}
\label{sub:bar(free)}

\begin{blanko}{Set-up.}
  We fix an operad in the form of a polynomial monad $\PPP$,
  represented by $I \leftarrow E \to B \to I$. Then, as we have seen
  in \ref{thm:P*}, the free monad $\PPP\upperstar$ is represented by
  $I \leftarrow \tr'(\PPP) \to \tr(\PPP) \to I$. We denote by $F:I \to
  1$ the unique map to the terminal groupoid.

  Let $Y = \TSB{\SSS}{\PPP\upperstar}$ denote the two-sided bar construction
  of $\PPP \upperstar$. We have
  \[
  \begin{tikzcd}[column sep = 5em]
	Y=\TSB{\SSS}{\PPP\upperstar} : &
  \SSS  F\lowershriek 1 
  \ar[r, pos=0.65, "s_0" on top] 
  &
  \ar[l, shift left=6pt, pos=0.65, "d_0" on top]
  \ar[l, wavy, shift right=6pt, pos=0.65, "d_1" on top]
  \SSS  F\lowershriek \PPP\upperstar 1
  \ar[r, shift right=6pt, pos=0.65, "s_0" on top]
  \ar[r, shift left=6pt, pos=0.65, "s_1" on top]  
  &
  \ar[l, shift left=12pt, pos=0.65, "d_0" on top]
  \ar[l, wavy, shift right=12pt, pos=0.65, "d_2" on top]
  \ar[l, pos=0.65, "d_1" on top]
  \SSS  F\lowershriek \PPP\upperstar\PPP\upperstar 1
  \ar[r, phantom, "\cdots" on top] & {}
  \end{tikzcd}
  \]
  Thus $Y_0 = \SSS F\lowershriek 1 = \SSS I$ is the groupoid of monomials
  of colours, and $Y_1 = \SSS F\lowershriek \PPP\upperstar 1 = \SSS
  \tr(\PPP)$ is the groupoid of monomials of $\PPP$-trees (which we may
  interpret as $\PPP $-forests).
\end{blanko}

\begin{blanko}{Interpretation of the groupoids $Y_n$.}
  The general description of the two-sided bar construction in
  \ref{sub:bar} tells us that $Y_n$ is the groupoid of (monomials of)
  $n$-level trees $W$ decorated by $\PPP\upperstar$-trees. This is the same
  as (monomials of) active maps
  $$
  W \actto T ,
  $$
  where $W$ is an $n$-level tree and $T$ is any $\PPP $-tree.

  (Note that a $1$-level $\PPP$-tree is a corolla, and to
  give an active map from a corolla to a $\PPP$-tree is the same thing as
  giving the $\PPP$-tree (each $\PPP$-tree receives a unique active map
  from a $\PPP$-corolla (namely its residue~\ref{residue})).)
\end{blanko}

\begin{blanko}{Inner face maps as precomposition.}
  The
  inner face maps $Y_{n-1} \stackrel{d_i}\leftarrow Y_n$ ($0<i<n$) are
  described as follows. Let $W'$ be the $(n{-}1)$-level tree obtained
  from $W$ by contracting all the edges between levels $n-1$ and $n$, then
  there is a unique active map $W' \actto W$. Now $d_i(W{\actto}T)$ is the
  composite $W'\actto W \actto T$.
  
  The same kind of description works for the degeneracy maps.
\end{blanko}

\begin{blanko}{Outer face maps in terms of active-inert factorisation in $\OMEGA$.}
  \label{faceY}
  The outer face maps of $Y$ 
  involve active-inert factorisation in $\OMEGA$
  (cf.~\ref{active-inert-Omega}). We describe the top face maps. Given a
  (connected) $n$-simplex $W \actto T$, consider the $(n{-}1)$-level forest
  $W'$ obtained by deleting the root node of $W$, with the canonical inert
  map of forests $W' \into W$. The top face map $Y_{n-1}
  \stackrel{d_n}\leftarrow Y_n$ returns
  the $(n-1)$-simplex $W'\actto T'$ appearing in the active-inert
  factorisation
  \[
  \begin{tikzcd}[sep = {4em,between origins}]
  W' \ar[d, into] \ar[r, dotted, -act] & T' \ar[d, dotted, into] \dlactinert \\
  W \ar[r, ->|] & T .
  \end{tikzcd}
  \]
  
  The bottom face maps $Y_{n-1} \stackrel{d_0}\leftarrow Y_n$ are described 
  similarly, by deleting the leaf level instead of the root level.
\end{blanko}

\begin{blanko}{Layerings and cuts.}\label{layerings-cuts}
  We define an {\em $n$-layering} of a $\PPP$-tree $T$ to be an active map from 
  an $n$-level tree
  $$
  W \actto T.
  $$
  We also call it a $\PPP$-tree with $n-1$ {\em compatible cuts}. For the
  actual computations, we shall stick with active maps, so in a sense
  layerings and cuts are just redundant terminology. However, they serve to
  strengthen the combinatorial intuition with further pictures, and to
  relate to the notion of directed restriction
  species~\cite{Galvez-Kock-Tonks:1207.6404}\footnote{For 
  nontrivial trees, the layering notion
  agrees with the general notion of layering of directed restriction 
  species of \cite{Galvez-Kock-Tonks:1207.6404}. In the presence of 
  trivial trees, the situation is a more subtle.} and the
  Butcher--Connes--Kreimer bialgebra of Example~\ref{ex:CK},
  as we shall formalise in \ref{sec:CEFM}.
  
  In particular, an active map $W \actto T$ from a $2$-level tree 
  encodes equivalently the following data: a bottom tree (the image of the 
  root-level
  node of $W$) whose leaves are decorated with other trees (the images of
  the leaf-level nodes of $W$). This is conveniently interpreted as a tree
  with a cut, as illustrated here:
  
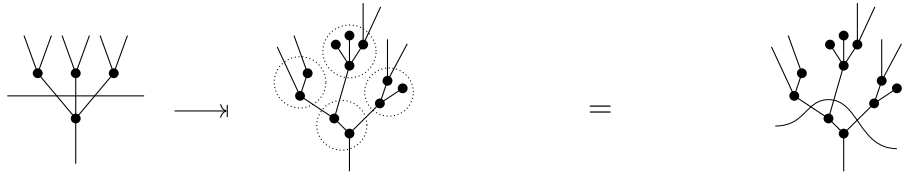
\begin{center}
\begin{tikzpicture}
  
  \begin{scope}[shift={(0.0, 0.0)}]
    \coordinate (b) at (0.0, 0.6);
    \coordinate (t1) at (-0.5, 1.2);
    \coordinate (t2) at (0.0, 1.2);
    \coordinate (t3) at (0.5, 1.2);
    \draw (0.0, 0.0) -- (b) pic {onedot};
	\draw (b) -- (t1) pic {onedot};
	\draw (b) -- (t2) pic {onedot};
	\draw (b) -- (t3) pic {onedot};
	\draw (t1) -- +(-0.18, 0.5);
	\draw (t1) -- +(0.18, 0.5);
	\draw (t2) -- +(-0.18, 0.5);
	\draw (t2) -- +(0.18, 0.5);
	\draw (t3) -- +(-0.18, 0.5);
	\draw (t3) -- +(0.18, 0.5);
	
	\draw (-0.9, 0.9) -- +(1.8, 0);
  \end{scope}
  
  \begin{scope}[shift={(3.5, 0.0)}]
    \coordinate (b1) at (0.1, 0.4);
    \coordinate (b2) at (-0.1, 0.6);
	\draw[densely dotted] (0.0, 0.51) circle (0.33);
	
	\begin{scope}[shift={(-0.55, 0.9)}]
	  \coordinate (x1) at (0.0, 0.0);
	  \coordinate (x2) at (0.1, 0.3);
	  \draw[densely dotted] (0.0, 0.18) circle (0.34);	
	\end{scope}
	
	\begin{scope}[shift={(0.1, 1.3)}]
	  \coordinate (y1) at (0.0, 0.0);
	  \coordinate (y2) at (-0.18, 0.28);
	  \coordinate (y3) at (0.0, 0.4);
	  \coordinate (y4) at (0.18, 0.28);
	  \draw[densely dotted] (0.0, 0.19) circle (0.35);
	\end{scope}
	
	\begin{scope}[shift={(0.5, 0.8)}]
	  \coordinate (z1) at (0.0, 0.0);
	  \coordinate (z2) at (0.1, 0.3);
	  \coordinate (z3) at (0.3, 0.2);
	  \draw[densely dotted] (0.12, 0.15) circle (0.32);
	\end{scope}
	
    \draw (0.1, -0.1) -- (b1) pic {onedot} -- (b2) pic {onedot};	
	\draw (b2) -- (x1) pic {onedot} -- (x2) pic {onedot};
	\draw (b2) -- (y1) pic {onedot};
	\draw (b1) -- (z1) pic {onedot};

	\draw (x1) -- +(-0.3, 0.65);
	\draw (x2) -- +(-0.18, 0.5);
	
	\draw (y1) -- (y2) pic {onedot};
	\draw (y1) -- (y3) pic {onedot};
	\draw (y1) -- (y4) pic {onedot};
	\draw (y4) -- +(0.0, 0.55);
	\draw (y4) -- +(0.23, 0.5);
	
	\draw (z1) -- (z2) pic {onedot};
	\draw (z1) -- (z3) pic {onedot};
	\draw (z2) -- +(-0.0, 0.55);
	\draw (z2) -- +(0.26, 0.49);
  \end{scope}
  
  \draw [->|] (1.3, 0.7) -- +(0.7, 0); 

  \node at (6.9, 0.7) {$=$};
  
  \begin{scope}[shift={(10.0, 0.0)}]
    \coordinate (b1) at (0.1, 0.4);
    \coordinate (b2) at (-0.1, 0.6);
	
	\begin{scope}[shift={(-0.55, 0.9)}]
	  \coordinate (x1) at (0.0, 0.0);
	  \coordinate (x2) at (0.1, 0.3);
	\end{scope}
	
	\begin{scope}[shift={(0.1, 1.3)}]
	  \coordinate (y1) at (0.0, 0.0);
	  \coordinate (y2) at (-0.18, 0.28);
	  \coordinate (y3) at (0.0, 0.4);
	  \coordinate (y4) at (0.18, 0.28);
	\end{scope}
	
	\begin{scope}[shift={(0.5, 0.8)}]
	  \coordinate (z1) at (0.0, 0.0);
	  \coordinate (z2) at (0.1, 0.3);
	  \coordinate (z3) at (0.3, 0.2);
	\end{scope}
	
    \draw (0.1, -0.1) -- (b1) pic {onedot} -- (b2) pic {onedot};	
	\draw (b2) -- (x1) pic {onedot} -- (x2) pic {onedot};
	\draw (b2) -- (y1) pic {onedot};
	\draw (b1) -- (z1) pic {onedot};

	\draw (x1) -- +(-0.3, 0.65);
	\draw (x2) -- +(-0.18, 0.5);
	
	\draw (y1) -- (y2) pic {onedot};
	\draw (y1) -- (y3) pic {onedot};
	\draw (y1) -- (y4) pic {onedot};
	\draw (y4) -- +(0.0, 0.55);
	\draw (y4) -- +(0.23, 0.5);
	
	\draw (z1) -- (z2) pic {onedot};
	\draw (z1) -- (z3) pic {onedot};
	\draw (z2) -- +(-0.0, 0.55);
	\draw (z2) -- +(0.26, 0.49);
	
	\draw[ultra thin] (-0.8, 0.5) .. controls (-0.4, 0.5) and (-0.4, 0.85) .. (-0.1, 0.85);
	\draw[ultra thin] (-0.1, 0.85) .. controls (0.3, 0.85) and (0.3, 0.2) .. (0.8, 0.2);

  \end{scope}
  
\end{tikzpicture} 
\end{center}
\end{blanko}

\begin{blanko}{Layer interpretation of the face and degeneracy maps.}
  Trees (and forests) should be read from leaves to root, in accordance with
  the operad interpretation, where leaves are inputs and root is output.
  The face maps are best understood from this viewpoint: in the lowest 
  degree
$
\begin{tikzcd}[column sep={3.8em}]
  Y_0 
  \ar[r, pos=0.65, "s_0" on top] & 
  \ar[l, shift left=6pt, pos=0.65, "d_0" on top] 
  \ar[l, wavy, shift right=6pt, pos=0.65, "d_1" on top] 
  Y_1
\end{tikzcd}$

$\bullet$ \
$d_0$ deletes the leaf level, retaining only the root colour;

$\bullet$ \
$d_1$ deletes the root level, retaining only the monomial of leaf 
colours.

\noindent
  In the next degree, 
  $
  \begin{tikzcd}[column sep={3.8em}]
    Y_1
  \ar[r, shift right=6pt, pos=0.65, "s_0" on top]
  \ar[r, shift left=6pt, pos=0.65, "s_1" on top]  
  &
  \ar[l, shift left=12pt, pos=0.65, "d_0" on top]
  \ar[l, wavy, shift right=12pt, pos=0.65, "d_2" on top]
  \ar[l, pos=0.65, "d_1" on top]
  Y_2
\end{tikzcd}
$

$\bullet$ \
$d_0$ deletes the leaf layer, leaving only 
a tree containing the root;

$\bullet$ \
$d_1$ joins the two layers, 
keeping the underlying $\PPP$-tree fixed;

$\bullet$ \
$d_2$ deletes the root 
layer, leaving only the crown forest.

\end{blanko}
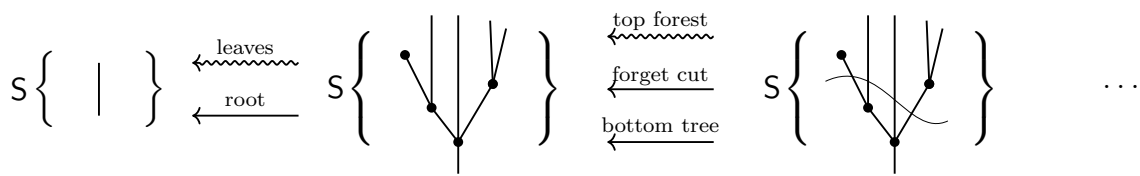
\begin{equation*}
  \begin{tikzpicture}[line width=0.25mm]

  \begin{scope}[shift={(0.525, 0.42)}] 
	\draw (0.0, 0.525) -- (0.0, 1.225);
	\draw (-1.05, 0.875) node {$\SSS$};
	\draw (-0.7, 0.875) node {$\bigleftbrace{15}$};
	\draw (0.7, 0.875) node {$\bigrightbrace{15}$};
  \end{scope}

  \begin{scope}[shift={(1.05, 0.0)}]
    \draw[->, wavy] (2.1, 1.645) -- +(-1.4, 0.0);
    \draw[->] (2.1, 0.945) -- +(-1.4, 0.0);
	\draw (1.4, 1.86) node {\scriptsize leaves};
	\draw (1.4, 1.16) node {\scriptsize root};
  \end{scope}
	
  \begin{scope}[shift={(5.25, 0.0)}] 
	\draw (0.0, 0.175) -- (0.0, 0.595);
	\fill (0.0, 0.595) circle[radius=0.065];
	\draw (0.0, 0.595) -- (-0.35, 1.05);
	\fill (-0.35, 1.05) circle[radius=0.065];
	\draw (-0.35, 1.05) -- (-0.7, 1.75);
	\fill (-0.7, 1.75) circle[radius=0.065];
	\draw (-0.35, 1.05) -- (-0.35, 2.275);
	\draw (0.0, 0.595) -- (0.0, 2.275);
	\draw (0.0, 0.595) -- (0.455, 1.365);
	\fill (0.455, 1.365) circle[radius=0.065];
	\draw (0.455, 1.365) -- (0.42, 2.205);
	\fill (0.455, 1.365) circle[radius=0.065];
	\draw (0.455, 1.365) -- (0.63, 2.1);
	\draw (-1.61, 1.33) node {$\SSS$};
	\draw (-1.26, 1.33) node {$\bigleftbrace{26}$};
	\draw (1.155, 1.33) node {$\bigrightbrace{26}$};
  \end{scope}

  \begin{scope}[shift={(6.51, 0.0)}]
    \draw[->, wavy] (2.1, 1.995) -- +(-1.4, 0.0);
    \draw[->] (2.1, 1.295) -- +(-1.4, 0.0);
    \draw[->] (2.1, 0.595) -- +(-1.4, 0.0);
	
	\draw (1.4, 2.21) node {\scriptsize top forest};
	\draw (1.4, 1.47) node {\scriptsize forget cut};
	\draw (1.4, 0.82) node {\scriptsize bottom tree};
  \end{scope}

  \begin{scope}[shift={(10.99, 0.0)}] 
	\draw (0.0, 0.175) -- (0.0, 0.595);
	\fill (0.0, 0.595) circle[radius=0.065];
	\draw (0.0, 0.595) -- (-0.35, 1.05);
	\fill (-0.35, 1.05) circle[radius=0.065];
	\draw (-0.35, 1.05) -- (-0.7, 1.75);
	\fill (-0.7, 1.75) circle[radius=0.065];
	\draw (-0.35, 1.05) -- (-0.35, 2.275);
	\draw (0.0, 0.595) -- (0.0, 2.275);
	\draw (0.0, 0.595) -- (0.455, 1.365);
	\fill (0.455, 1.365) circle[radius=0.065];
	\draw (0.455, 1.365) -- (0.42, 2.205);
	\fill (0.455, 1.365) circle[radius=0.065];
	\draw (0.455, 1.365) -- (0.63, 2.1);
	\draw[line width=0.01] (-0.91, 1.4) .. controls (-0.175, 1.75) and (0.175, 0.63) .. (0.7, 0.875);
	\draw (-1.61, 1.33) node {$\SSS$};
	\draw (-1.26, 1.33) node {$\bigleftbrace{26}$};
	\draw (1.155, 1.33) node {$\bigrightbrace{26}$};
  \end{scope}
  
  \draw (14, 1.33) node {$\dots$};

\end{tikzpicture}\end{equation*}

The degeneracy maps (not pictured) insert empty layers:

$\bullet$ \
$s_0 : Y_0 \to Y_0$ sends a colour (trivial tree)
to the trivial tree of that colour;

$\bullet$ \
$s_0 : Y_1 \to Y_2$ adds the cut of all the leaves;

$\bullet$ \
$s_1 : Y_1 \to Y_2$ adds the cut of the root edge.

The top face maps generally take connected objects to 
non-connected ones, as exemplified by $Y_1 \stackrel{d_2}\leftarrow Y_2$
which takes a tree with a cut to the top forest.
All the non-top face maps as well as the degeneracy maps send 
connected objects to connected objects.

\subsection{Two-sided bar construction on $\PPP \bd$}
\label{sub:bar(BD)}

Let $\PPP$ be any operad, represented by $I \leftarrow E \to B \to I$.
We have seen in \ref{thm:Pbd} that the Baez--Dolan construction $\PPP\bd$ is represented by
$B \leftarrow \tr^\bullet(\PPP) \to \tr(\PPP) \to B$. We denote by $F:B 
\to 1$ the unique map (apologising for the notation clash: previously $F$ 
denoted the map $I \to 1$, relevant for the free-monad construction).

Denote by $Z:= \TSB{\SSS}{\PPP\bd}$ the two-sided bar construction of 
$\PPP\bd$. The beginning of $Z$ looks like this (picturing $\PPP$-trees
as plain trees, and omitting degeneracy maps):
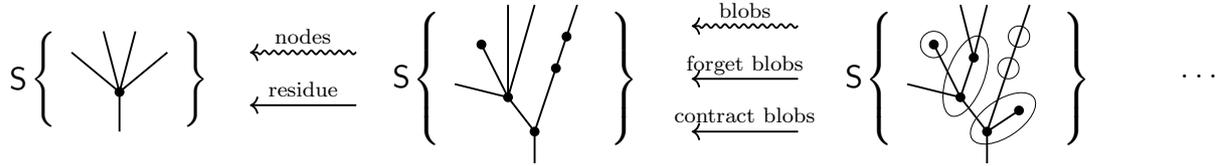
\begin{equation*}
  \begin{tikzpicture}[line width=0.25mm]

	\begin{scope}[shift={(0.77, 0.07)}] 
	  \draw (0.0, 0.525) -- (0.0, 1.05);
	  \fill (0.0, 1.05) circle[radius=0.065];
	  \draw (0.0, 1.05) -- (-0.63, 1.575); 
	  \draw (0.0, 1.05) -- (-0.21, 1.855);
	  \draw (0.0, 1.05) -- (0.21, 1.855);
	  \draw (0.0, 1.05) -- (0.63, 1.575);
	  \draw (-1.33, 1.225) node {$\SSS$};
	  \draw (-0.98, 1.225) node {$\bigleftbrace{20}$};
	  \draw (0.98, 1.225) node {$\bigrightbrace{20}$};
	\end{scope}

	\begin{scope}[shift={(1.785, 0.0)}]
	  \draw[->, wavy] (2.1, 1.645) -- +(-1.4, 0.0);
	  \draw[->] (2.1, 0.945) -- +(-1.4, 0.0);
	  \draw (1.4, 1.86) node {\scriptsize nodes};
	  \draw (1.4, 1.16) node {\scriptsize residue};
	\end{scope}

	\begin{scope}[shift={(6.23, 0.0)}] 
	  \draw (0.0, 0.175) -- (0.0, 0.595);
	  \fill (0.0, 0.595) circle[radius=0.065];
	  \draw (0.0, 0.595) -- (-0.35, 1.05);
	  \fill (-0.35, 1.05) circle[radius=0.065];
	  \draw (-0.35, 1.05) -- (-1.05, 1.225);
	  \draw (-0.35, 1.05) -- (-0.7, 1.75);
	  \fill (-0.7, 1.75) circle[radius=0.065];
	  \draw (-0.35, 1.05) -- (-0.35, 2.275);
	  \draw (-0.35, 1.05) -- (0.0, 2.275);
	  \draw (0.0, 0.595) -- (0.56, 2.275);
	  \fill (0.28, 1.435) circle[radius=0.065];
	  \fill (0.42, 1.855) circle[radius=0.065];
	  \draw (-1.75, 1.33) node {$\SSS$};
	  \draw (-1.4, 1.33) node {$\bigleftbrace{30}$};
	  \draw (1.155, 1.33) node {$\bigrightbrace{30}$};
	\end{scope}

	\begin{scope}[shift={(7.595, 0.0)}]
	  \draw[->, wavy] (2.1, 1.995) -- +(-1.4, 0.0);
	  \draw[->] (2.1, 1.295) -- +(-1.4, 0.0);
	  \draw[->] (2.1, 0.595) -- +(-1.4, 0.0);
	  \draw (1.4, 2.20) node {\scriptsize blobs};
	  \draw (1.4, 1.47) node {\scriptsize forget blobs};
	  \draw (1.4, 0.81) node {\scriptsize contract blobs};
	\end{scope}

	\begin{scope}[shift={(12.18, 0.0)}] 
	  \draw (0.0, 0.175) -- (0.0, 0.595);
	  \fill (0.0, 0.595) circle[radius=0.065];
	  \draw (0.0, 0.595) -- (-0.35, 1.05);
	  \fill (-0.35, 1.05) circle[radius=0.065];
	  \draw (-0.35, 1.05) -- (-1.05, 1.225);
	  \draw (-0.35, 1.05) -- (-0.7, 1.75);
	  \fill (-0.7, 1.75) circle[radius=0.065];
	  \draw (-0.35, 1.05) -- (-0.175, 1.575);
	  \fill (-0.175, 1.575) circle[radius=0.065];
	  \draw (-0.175, 1.575) -- (-0.35, 2.275);
	  \draw (-0.175, 1.575) -- (0.0, 2.275);
	  \draw (0.0, 0.595) -- (0.56, 2.275);
	  \draw (0.0, 0.595) -- (0.42, 0.875);
	  \fill (0.42, 0.875) circle[radius=0.065];
	  \draw[line width=0.01, rotate around={70:(-0.28,1.33)}] (-0.28,1.33)  ellipse (0.56 and 0.245);
	  \draw[line width=0.01, rotate around={34:(0.21,0.77)}] (0.21,0.77)  ellipse (0.49 and 0.245);
	  \draw[line width=0.01, rotate around={70:(-0.7,1.75)}] (-0.7,1.75)  ellipse (0.175 and 0.175);
	  \draw[line width=0.01, rotate around={70:(0.28,1.435)}] (0.28,1.435)  ellipse (0.14 and 0.14);
	  \draw[line width=0.01, rotate around={70:(0.42,1.855)}] (0.42,1.855)  ellipse (0.14 and 0.14);
	  \draw (-1.75, 1.33) node {$\SSS$};
	  \draw (-1.4, 1.33) node {$\bigleftbrace{30}$};
	  \draw (1.155, 1.33) node {$\bigrightbrace{30}$};
	\end{scope}
	  	\draw (15, 1.33) node {$\dots$};

  \end{tikzpicture}
\end{equation*}

\noindent
as we now formalise.

\begin{blanko}{Description of the face maps of $Z=\TSB{\SSS}{\PPP\bd}$.}
  \label{faceZ}
  All the non-top face maps, as well as the degeneracy maps, send connected
  configurations to connected configurations. The top face maps generally
  take a connected configuration to a non-connected one (as indicated with
  wavy arrows), exemplified by $Z_1 \stackrel{d_2}\leftarrow Z_2$ which
  takes a blobbed tree to the forest consisting of the individual
  blobs. More formally, given an element $K \actto T$ in $Z_2$, there is a
  canonical cover of $K$ by corollas (one for each node), more 
  precisely a bijective-on-nodes inert map of forests
  $$
  \textstyle{\sum_i C_i \into K .}
  $$
  For each node $C_i$, take active-inert factorisation as in the diagram
  \[
  \begin{tikzcd}[sep={4em,between origins}]
  C_i \ar[d, into] \ar[r, dotted, -act]  & 
  S_i \ar[d, dotted, into] \dlactinert  \\
  K \ar[r, -act] & T  .
  \end{tikzcd}
  \]
  Altogether
  $$
  d_2( K{\actto}T) = S_1\cdots S_k.
  $$

  The general case is described as follows.
  $Z_n$ is the groupoid of monomials of sequences of active maps of $\PPP$-trees
  $$
  K^{(n)} \actto K^{(n-1)} \actto \cdots \actto K^{(1)} \actto K^{(0)}=T \to \PPP   ,
  $$
  with the condition that $K^{(n)}$ is a corolla. The last map $T \to \PPP
  $ is just to say that $T$ is a $\PPP $-tree. By \ref{induce}, this
  induces $\PPP$-structure on all trees $K^{(i)}$ in the sequence. Note
  that $K^{(n)}$ is redundant information, since any tree admits a unique active
  map from a corolla (more precisely: admits such a map and any two such are 
  uniquely isomorphic).
  Including $K^{(n)}$ is convenient because it gives a nerve flavour to the
  simplicial structure, as we shall now see:

  The bottom face map $d_0$ consists in deleting $T$.

  The middle face maps $d_i$ ($0<i<n$)
  consist in composing two consecutive maps.  The degeneracy maps $s_j$ just 
  insert identity maps.

  The top face map $d_n$ cannot
  just delete $K^{(n)}$, because of the requirement that the sequence begins
  with a corolla.  What it does instead is to `look into the
  nodes of $K^{(n-1)}$ and for each node return the remaining sequence seen through that 
  node'.  More precisely, each node of $K^{(n-1)}$ defines an inert map from a corolla
  $C\into K^{(n-1)}$;
  let $C^{(n-1)}$ be the forest of all nodes in $K^{(n-1)}$, so that 
  $C^{(n-1)} \into 
  K^{(n-1)}$
  is an inert map with the further properties that it is bijective on 
  nodes and its domain is a forest of corollas.  Now the top face map 
  is defined by returning the sequence obtained by active--inert factorising
  the whole configuration:
  $$\begin{tikzcd}[column sep = {5.2em,between origins}, row sep = {4.7em,between origins}]
  C^{(n-1)} \ar[d, into] \ar[r, dotted, -act] & 
  C^{(n-2)} \ar[d, dotted, into] \ar[r, dotted, -act] \dlactinert & 
  \cdots \ar[r, dotted, -act] & 
  C^{(1)} \dlactinert \ar[d, dotted, into] \ar[r, dotted, -act] &
  C^{(0)} \dlactinert \ar[d, dotted, into] 
  \\
  K^{(n-1)} \ar[r, -act] & K^{(n-2)} \ar[r, -act] & 
  \cdots \ar[r, -act] &  K^{(1)} \ar[r, -act] & K^{(0)}   .
  \end{tikzcd}$$
\end{blanko}

\begin{blanko}{Follow-up remarks on strictness (cf.~\ref{weak?} and \ref{followup}).}
\label{followup2}
  Let us comment on the possibility of carrying out the constructions in a 
  stricter setting. The challenge is to find a level of strictness for 
  polynomial monads 
  strict enough to give a strict simplicial groupoid for its two-sided 
  bar construction but still Segal, and at the same time reproduce itself under the 
  free-monad and Baez--Dolan constructions. 
  It was checked in \cite[\S 6.1]{Kock-Weber:1609.03276} that the two-sided
  bar construction can be kept strict and Segal if just $B \to I$ is a fibration. 

  One possibility is to work with polynomial monads
    $$    
  \begin{tikzcd}
  I \ar[d]& \ar[l, "s"']  E\ar[d] \drpullback \ar[r, "p"] & B\ar[d, 
  "f"] 
  \ar[r, "t"] & 
  I\ar[d] \\
  1  &\ar[l] \B'\ar[r] & \B \ar[r]  &1 
  \end{tikzcd}
  $$
  for which all squares are required strict, and {\em $f$, $s$, and $t$ are
  required to be fibrations}.
  (Weber's setting meets these conditions, but
  requires furthermore that $I$ is
  discrete and that $f$ is a discrete fibration.) Intuitively, what the
  conditions say is that for any operation one can change the output and
  input colours to isomorphic colours and get a new operation.

  It is not difficult to check that the conditions are reproduced by the
  free-monad construction: roughly, the leaf map $I \leftarrow \tr'(\PPP)$
  is again a fibration because one can change colour at the marked leaf by
  changing colour at the adjacent node, and the root map $\tr(\PPP) \to I$
  is a fibration by the same argument applied to the root. For the
  Baez--Dolan construction it is more difficult. The
  return-the-marked-node map $B \leftarrow\tr^\bullet(\PPP)$ is a
  fibration because one can replace the marked node by an isomorphic node
  by gluing along the colour changes at the gluing locus. Similarly, $\mu:
  \tr(\PPP) \to B$ is {\em almost} a fibration: one can replace the residue
  of a tree --- but this does {\em not} work for trivial trees. Indeed
  given an isomorphism of colours $i\simeq j$, the trivial tree
  \inlineDotlessTree $\!_i$\; has residue \inlineonetree $\!^i_i$\; which
  is isomorphic to \inlineonetree $\!^i_j$\; (by the fibration condition),
  but there is no trivial tree with this residue. It is necessary here to
  perform a fibrant replacement of $\mu:\tr(\PPP) \to B$, by
  standard-homotopy-pullback along $\id: B \isopil B$. This replaces the
  groupoid of $\PPP$-trees with the equivalent groupoid of $\PPP$-trees
  equipped with an active map from some corolla. Note that we already
  performed this replacement silently in \ref{faceZ} when we described
  $Z_1$ as the groupoid of configurations
  $$
  K^{(1)} \actto K^{(0} = T \to \PPP,
  $$
  including the redundant corolla $K^{(1)}$.
\end{blanko}

\section{Incidence comodule bialgebras}

\label{sec:main-general}

So far we have seen that the simplicial groupoids given by two-sided
bar construction on $\PPP \upperstar$ and $\PPP \bd$, respectively, have the same 
groupoid in degree $1$, namely the groupoid $\SSS(\tr(\PPP))$ of monomials of 
$\PPP$-trees, and therefore define two different bialgebra
structures on $\Grpd_{/Y_1} \simeq \Grpd_{/Z_1}$.  In this section we 
establish the first version of the Main Theorem (\ref{thm:main}), showing 
that these two bialgebras form a comodule bialgebra. We first need to
set up the language required.

\subsection{Comodule bialgebras}
\label{sub:comodulebialgebras}

\begin{blanko}{Classical comodule bialgebras.}
  Recall first the definition (see for example Abe~\cite[\S~3.2]{Abe} and
  Manchon~\cite{Manchon:Abelsymposium}). Fix a bialgebra $B$ (over $\Q$).
  A {\em comodule
  bialgebra} over $B$ is a bialgebra in the (braided) monoidal category 
  $B\kat{-Comod}$
  of
  (left) $B$-comodules. Note that the notion of $B$-comodule uses only the
  coalgebra structure of $B$, not the algebra structure, but that it is the
  algebra structure of $B$ that endows the category of $B$-comodules with a
  {\em monoidal} structure, given as follows. If $M$ and $N$ are left
  $B$-comodules, then there is a left $B$-comodule structure on $M\tensor
  N$ is given by the composite map
  $$
  M \tensor N \to B \tensor M \tensor B \tensor N 
  \stackrel{\mu_{13}}\to B \tensor M \tensor N .
  $$
  Here $\mu_{13}$ is the map that first swaps the two middle tensor
  factors and then uses the multiplication of $B$ in the two now
  adjacent $B$-factors. It follows from the bialgebra axioms that this
  is a valid left $B$-comodule structure, giving the monoidal
  structure on the category of left $B$-comodules; the unit object is
  the $B$-comodule $\Q$ (with structure map the unit of $B$).
  Furthermore, one checks that the braiding on the underlying category
  of vector spaces lifts to the category of $B$-comodules --- this
  depends on the bialgebra $B$ being commutative.

  We now have a braided monoidal structure on $B\kat{-Comod}$, and it makes
  sense to consider bialgebras in here. A bialgebra in $B\kat{-Comod}$ is a
  $B$-comodule $M$ together with structure maps
  \begin{xalignat*}{2}
  \Delta_M : M &\to M \tensor M  & \varepsilon_M : M &\to \Q \\
  \mu_M: M \tensor M &\to M & \eta_M : \Q &\to M 
  \end{xalignat*}
  all required to be $B$-comodule maps and
   to satisfy the usual bialgebra axioms.
  We shall be concerned in particular with the requirement that
  $\Delta$ and $\varepsilon$ be $B$-comodule maps, which is to say 
  that
  they are compatible with the coaction $\gamma: M \to B \tensor M$:
  \begin{equation}\label{comodulebialg-axiom}
  \begin{tikzcd}
  M \ar[r, "\Delta_M"] \ar[dd, "\gamma"'] & M \tensor M \ar[d, "\gamma \tensor 
  \gamma"] \\
  & B \tensor M \tensor B \tensor M \ar[d, "\mu_{13}"] \\
  B \tensor M \ar[r, "B \tensor \Delta_M"'] & B \tensor M \tensor M
  \end{tikzcd}
  \qquad
  \begin{tikzcd}
  M \ar[d, "\gamma"'] \ar[r, "\varepsilon_M"] & \Q \ar[d, "\eta_B"] \\
  B \tensor M \ar[r, "B\tensor \varepsilon"'] & B .
  \end{tikzcd}
  \end{equation}
  In the main theorem, the two other axioms will automatically be 
  satisfied, because it will be the case that
  
  $\bullet$ \
  as a
  comodule, $M$ coincides with $B$ itself (with coaction $=$ comultiplication),
    
  $\bullet$ \
  the algebra structure of $M$ coincides with that of $B$ (both will 
  be free commutative).

  We now turn to the objective version of these structures.
\end{blanko}

\begin{blanko}{Comodule configurations.}\label{comoduleconf}
  Comodules arising in combinatorics are often the cardinality of
  certain slice-level comodules given by so-called comodule
  configurations of simplicial groupoids, first studied by
  Walde~\cite{Walde:1611.08241} and Young~\cite{Young:1611.09234}. We
  follow the terminology of Carlier~\cite{Carlier:1801.07504}, 
  \cite{Carlier:1903.07964}.
  For $X$ a decomposition space, a {\em left $X$-comodule configuration} 
  is a culf map 
  $$
  u:M \to X,
  $$
  where $M$ is a Segal space. It is useful to stare at the diagram:
  \[\begin{tikzcd}[column sep = 5em]
  X_0
  \ar[r, pos=0.65, "s_0" on top] 
  &
  \ar[l, shift left=6pt, pos=0.65, "d_0" on top]
  \ar[l, shift right=6pt, pos=0.65, "d_1" on top]
  X_1
  \ar[r, shift right=6pt, pos=0.65, "s_0" on top]
  \ar[r, shift left=6pt, pos=0.65, "s_1" on top]  
  &
  \ar[l, shift left=12pt, pos=0.65, "d_0" on top]
  \ar[l, shift right=12pt, pos=0.65, "d_2" on top]
  \ar[l, pos=0.65, "d_1" on top]
  X_2
  &
  \cdots
  \\ 
  \\
  M_0
  \ar[uu, "u"]
  \ar[r, pos=0.65, "s_0" on top] 
  &
  \ar[l, shift left=6pt, pos=0.65, "d_0" on top]
  \ar[l, shift right=6pt, pos=0.65, "d_1" on top]
  M_1
  \ar[uu, "u"]
  \ar[r, shift right=6pt, pos=0.65, "s_0" on top]  
  \ar[r, shift left=6pt, pos=0.65, "s_1" on top]  
  &
  \ar[l, shift left=12pt, pos=0.65, "d_0" on top]
  \ar[l, shift right=12pt, pos=0.65, "d_2" on top]
  \ar[l, pos=0.65, "d_1" on top]
  M_2
  \ar[uu, "u"]
  &
  \cdots
  \end{tikzcd}\] 
  The actual comodule is then $\Grpd_{/M_0}$, and the 
  coaction by $\Grpd_{/X_1}$ is the linear functor
  $$
  \gamma: \Grpd_{/M_0} \to \Grpd_{/X_1} \tensor \Grpd_{/M_0}
  $$
  given by the span
  $$
  M_0 \stackrel{d_1}{\longleftarrow} M_1 \stackrel{(u,d_0)}{\longrightarrow} X_1 
  \times M_0 .
  $$
  
  A comodule configuration $u:M \to X$ is called {\em locally finite} when
  $X$ is locally finite and the face map $M_0 
  \stackrel{d_1}{\leftarrow} M_1$ is finite. This is the map whose fibres 
  are summed over, so the condition is necessary and sufficient to be able
  to take homotopy cardinality to arrive at a comodule at the level of 
  vector spaces, called the {\em incidence comodule}.
\end{blanko}

\begin{blanko}{Example: decalage.}\label{dec-map}
  Recall that for any simplicial groupoid $X$, the decalage $\Dectop X$
  (also called the path-space construction~\cite{Dyckerhoff:1505.06940})
  is the simplicial groupoid obtained from $X$ by shifting all the
  groupoids down one degree, and omitting the top face and the top
  degeneracy maps. The original top face maps serve to give a simplicial
  map $u : \Dectop X \to X$ called the {\em dec map}.
  It is a general fact (\cite[Prop.~4.9]{Galvez-Kock-Tonks:1512.07573})
  that if $X$ is a decomposition space then $\Dectop X$ is a Segal space
  and the dec map $u : \Dectop X \to X$ is culf. This is to say that
  $\Dectop X$ is a left comodule configuration over $X$. The corresponding
  incidence comodule is simply the incidence coalgebra of $X$ considered as a
  left comodule over itself.
\end{blanko}

\begin{blanko}{Comodule bialgebras, objectively.}
  Given a symmetric monoidal decomposition space $Z$, there is induced
  a symmetric bialgebra structure on $\Grpd_{/Z_1}$ (meaning that the 
  multiplication is a symmetric 
  monoidal structure). To provide
  comodule-bialgebra structure on some slice $\Grpd_{/A}$ we need to make
  the groupoid $A$ appear simultaneously as $A=Y_1$ for a monoidal
  decomposition space $Y$, and as $A= M_0$ for a left $Z$-comodule
  configuration $u:M\to Z$. Then we need to check the axioms. In the 
  first case of interest, the underlying comodule, as well as its monoidal 
  structure, will be $Z$ again. As a comodule configuration, this is  
  more precisely the upper dec $u: \Dectop Z \to Z$ from \ref{dec-map}.
 \end{blanko}

\subsection{Main theorem, general form}
\label{sub:mainthm}

\begin{theorem}\label{thm:main}
  For any operad $\PPP$, the
  two-sided bar constructions $\TSB{\SSS}{\PPP\upperstar}$ and
  $\TSB{\SSS}{\PPP\bd}$ together endow the slice $ \Grpd_{/\SSS(\tr(\PPP))} $ with
  the structure of a comodule bialgebra. Precisely, the incidence bialgebra
  of $\TSB{\SSS}{\PPP\upperstar}$ is a left comodule bialgebra over the
  incidence bialgebra of $\TSB{\SSS}{\PPP\bd}$.
\end{theorem}

\begin{blanko}{Set-up.}
  As usual, we assume $\PPP$ is represented by $I \leftarrow E \to B \to
  I$, and employ the following notation:
  
  $\bullet$ \
  $Y = \TSB{\SSS}{\PPP\upperstar}$ is the two-sided bar
  construction on $\PPP\upperstar$;

  $\bullet$ \
  $Z = \TSB{\SSS}{\PPP\bd}$ is the two-sided bar construction on $\PPP\bd$.

  \noindent
  These simplicial objects have the same groupoid in degree $1$, which we
  give a special name:
  $$
  A := Z_1 = Y_1 = \SSS(\tr(\PPP)),
  $$
  the `basis' for the comodule bialgebra $\Grpd_{/A}$.
  In
  all diagrams following, comultiplication in $Z$ (as well as the coaction) is
  written vertically, whereas comultiplication in $Y$ is written
  horizontally.
\end{blanko}

\begin{proof}[Proof of Theorem~\ref{thm:main}]
  We show below in \ref{lem:comult=comodmap} that {\em the comultiplication
  of $Y$ is a $Z$-comodule map}, and in \ref{lem:counit=comodmap} that {\em
  the counit of $Y$ is a $Z$-comodule map}. The two remaining axioms, that
  the algebra structure maps are $Z$-comodule maps, are automatically
  satisfied since the algebra structure is the same as that of $Z$, which
  is compatible with the comodule structure by the bialgebra axioms for $Z$.
\end{proof}
  
\begin{lemma}\label{lem:comult=comodmap}
  The comultiplication of $Y$ is a $Z$-comodule map.
\end{lemma}

\begin{proof}
  We must show that the following diagram commutes up to natural isomorphism:
  \begin{equation}\label{eq:Grpd-diagram}
  \begin{tikzcd}
	 \Grpd_{/A} \ar[rr, "\Delta_Y"] \ar[dd, "\gamma"'] && \Grpd_{/A} \tensor \Grpd_{/A}
	 \ar[d, "\gamma\tensor \gamma"] \\
	 && \Grpd_{/A} \tensor \Grpd_{/A} \tensor \Grpd_{/A} \tensor \Grpd_{/A} \ar[d, "\mu_{13}"] \\
	 \Grpd_{/A} \tensor \Grpd_{/A} \ar[rr, "\id \tensor \Delta_Y"'] && 
	 \Grpd_{/A} \tensor \Grpd_{/A} \tensor \Grpd_{/A}   .
  \end{tikzcd}
  \end{equation}
  Spelling out the spans that define these functors, we are faced with
  the solid diagram
  \begin{equation}\label{eq:Q}
  \begin{tikzcd}[sep=large]
	 A & \ar[l, "d_1^Y"'] Y_2 \ar[r, "{(d_2^Y, d_0^Y)}"] & A \times A
	 \\
	 Z_2 \ar[u, "d_1^Z"] \ar[dd, "{(d_2^Z,d_0^Z)}"']  & 
	 Q \dlpullback \urpullback 
	 \ar[l, dotted] \ar[u, dotted] \ar[r, dotted] \ar[dd, dotted] &
	 Z_2 \times Z_2 \ar[d, "{(d_2^Z,d_0^Z)\times(d_2^Z,d_0^Z)}"] \ar[u, "d_1^Z\times d_1^Z"'] 
	 \\
	  {}&& A \times A \times A \times A \ar[d, "\mu_{13}"] 
	 \\
	 A \times A& A \times Y_2 \ar[l, "\id \times d_1^Y"] \ar[r, "{\id 
	 \times (d_2^Y,d_0^Y)}"']& A \times A \times A  .
  \end{tikzcd}
  \end{equation}
  To establish that Diagram~\eqref{eq:Grpd-diagram} commutes, the standard
  technique (see in particular \cite{Carlier:1903.07964}) is to fill the
  diagram \eqref{eq:Q} with commutative squares, such that furthermore the
  lower left-hand and upper right-hand squares are pullbacks, as indicated.
  Then the Beck--Chevalley isomorphisms \ref{lem:BC} will deliver the
  required natural isomorphism in \eqref{eq:Grpd-diagram}.

  We need to exhibit the middle groupoid $Q$ and the dotted maps.
  As is typical for an argument of this kind, these have formal 
  categorical definitions and at the same time clean combinatorial 
  interpretations.
  The groupoid $Q$ is the pullback
  \[\begin{tikzcd}
  Z_2 \ar[d, "d_0^Z"'] & \ar[l] Q \dlpullback \ar[d]  \\
  A & \ar[l, "d_1^Y"] Y_2 .
  \end{tikzcd}\]
  In other words, the (connected) objects of $Q$ consist of an active map
  $K \actto T$ (element in $Z_2$), and an active map $W \actto K$ with $W$
  a $2$-level tree (element in $Y_2)$, with the same $K$ (that's the
  condition of being fibred over $A$). This can be taken as a strict
  pullback, because $d_1^Y$ (in the explicit description given) is easily
  seen to be a fibration. In conclusion, $Q$ is the groupoid of (monomials
  of) active maps
  $$
  W \actto K \actto T,
  $$
  where $T$ and $K$ are $\PPP$-trees, and $W$ is a $2$-level $\PPP$-tree.
  Such configurations in turn can be interpreted as (monomials of) blobbed
  $\PPP$-trees with a compatible cut:
  \begin{center}
  \begin{tikzpicture}
	
	\begin{scope}[shift={(0.0,0.0)}]
	  \draw[densely dotted] (0.0, 0.0) circle (0.33);
	  \coordinate (b1) at (0.0, -0.17);
	  \coordinate (b21) at (-0.15, 0.11);
	  \coordinate (b22) at (0.15, 0.11);
	  \draw (b1)--(b21) pic {onedot};
	  \draw (b1)--(b22) pic {onedot};
	\end{scope}

	\begin{scope}[shift={(-0.65,0.45)}]
	  \draw[densely dotted] (0.0, 0.0) circle (0.3);
	  \coordinate (c1) at (0.1, -0.1);
	  \coordinate (c2) at (-0.1, 0.1);
	\end{scope}

	\begin{scope}[shift={(0.45,0.4)}]
	  \draw[densely dotted] (0.0, 0.0) circle (0.18);
	  \coordinate (d) at (0.0, 0.0);
	\end{scope}
	
	\begin{scope}[shift={(-1.0,1.25)}]
	  \draw[densely dotted] (0.0, 0.0) circle (0.32);
	  \coordinate (e1) at (0.0, -0.17);
	  \coordinate (e21) at (-0.14, 0.11);
	  \coordinate (e22) at (0.14, 0.11);
	  \draw (e1)--(e21) pic {onedot} -- +(-0.2, 0.6);
	  \draw (e21) -- +(0.2, 1.00);
	  \draw[densely dotted] (e21)+(0.12, 0.60) circle (0.14);
	  \draw (e1)--(e22) pic {onedot};
	\end{scope}

	\begin{scope}[shift={(0.0,1.0)}]
	  \draw[densely dotted] (0.0, 0.0) circle (0.32);
	  \coordinate (f1) at (0.0, -0.17);
	  \coordinate (f21) at (-0.14, 0.11);
	  \coordinate (f22) at (0.14, 0.11);
	  \draw (f1)--(f21) pic {onedot};
	  \draw (f1)--(f22) pic {onedot};
	  \draw (f22) -- +(-0.2, 0.8);
	  \draw (f22) -- +(0.17, 0.9);
	\end{scope}

	\begin{scope}[shift={(-0.4,1.6)}]
	  \draw[densely dotted] (0.0, 0.0) circle (0.28);
	  \coordinate (u3) at (0.0, -0.13);
	  \coordinate (u31) at (0.0, 0.1);
	  \draw (u3) -- (u31) pic {onedot} -- +(-0.2, 0.8);
	  \draw (u31) -- +(0.2, 0.8);
	\end{scope}

   \begin{scope}[shift={(1.1,1.05)}]
	  \draw[densely dotted] (0.03, 0.0) circle (0.46);
	  \coordinate (r1) at (0.0, -0.3);
	  \coordinate (r2) at (-0.25, 0.0);
	  \coordinate (r3) at (0.05, 0.0);
	  \coordinate (r4) at (0.3, -0.1);
	  \coordinate (r5) at (-0.1, 0.3);
	  \coordinate (r6) at (0.16, 0.3);
	  \draw (r1) -- (r2) pic {onedot} -- +(-0.4, 1.4);
	  \draw (r1) -- (r3) pic {onedot};
	  \draw (r1) -- (r4) pic {onedot} -- +(0.2, 0.8);
	  \draw (r3) -- (r5) pic {onedot};
	  \draw (r3) -- (r6) pic {onedot} -- +(0.2, 0.8);
	\end{scope}
	   
	\begin{scope}[shift={(0.48,1.4)}]
	  \draw[densely dotted] (0.0, 0.0) circle (0.16);
	  \coordinate (v) at (0.0, 0.0);
	\end{scope}

	\begin{scope}[shift={(1.0,2.0)}]
	  \draw[densely dotted] (0.0, 0.0) circle (0.32);
	  \coordinate (w1) at (0.0, -0.17);
	  \coordinate (w21) at (-0.14, 0.11);
	  \coordinate (w22) at (0.14, 0.11);
	  \draw (w1)--(w21) pic {onedot} -- +(-0.2, 0.6);
	  \draw (w1)--(w22) pic {onedot} -- +(0.2, 0.6);
	  \draw (w21) -- +(0.2, 0.8);
	\end{scope}

	\draw (0.0,-0.6) -- (b1) pic {onedot};
	\draw (b21)--(c1) pic {onedot} --(c2) pic {onedot};
	\draw (b22)--(d) pic {onedot};
	\draw (c2)--(e1) pic {onedot};
	\draw (b21)--(f1) pic {onedot};
	\draw (c1)--(u3) pic {onedot};
	\draw (d)--(r1) pic {onedot};
	\draw (d)--(v) pic {onedot};
	\draw (r5)--(w1) pic {onedot};

	\draw[ultra thin] (-1.5, 0.6) to[out=0, in=180] (-0.6, 0.9);
	\draw[ultra thin] (-0.6, 0.9) to[out=0, in=180] (0.0, 0.5);
	\draw[ultra thin] (0.0, 0.5) to[out=0, in=180] (0.45, 0.7);
	\draw[ultra thin] (0.45, 0.7) to[out=0, in=180] (1.2, 0.2);
	
	  \draw (-3.0, 1.05) node {$Q \ = \ $};
	  \draw (-2.2, 1.05) node {$\SSS$};
	  \draw (-1.8, 1.05) node {$\bigleftbrace{35}$};
	  \draw (2.0, 1.05) node {$\bigrightbrace{35}$};

  \end{tikzpicture}
  \end{center}
  Here $T$ is the total tree, $K$ is the tree of blobs, and the $2$-level 
  tree $W$ is represented by the cut, as explained in \ref{layerings-cuts}.

  It remains to exhibit the maps, and check that the squares are
  commutative and pullbacks as indicated. These checks occupy 
  \ref{pf:ULHS}--\ref{pf:LRHS} below.
\end{proof}

\begin{blanko}{The upper left-hand square.}\label{pf:ULHS}
  The maps constituting the upper left-hand square are clear from the 
  descriptions:
  $$
  \begin{tikzcd}
  A & \ar[l, "d_1"'] Y_2 \\
  Z_2 \ar[u, "d_1"] & \ar[l] Q \ar[u]
  \end{tikzcd}
  \qquad :
  \qquad
  \begin{tikzcd}
	\{T\} & \{W {\actto\,} T\} \ar[l, "\text{forget $W$}"'] \\
	\{K {\actto\,} T \} \ar[u, "\text{forget $K$}"] 
	& \{ W{\actto\,} K {\actto\,} T \} ,
	\ar[l, "\text{forget $W$}"] \ar[u, "\text{forget $K$}"']
  \end{tikzcd}
  $$
  and it is obvious that the square commutes (but it is not a pullback).
\end{blanko}

\begin{blanko}{The lower left-hand square is a pullback.}
  The lower left-hand square is given by
  $$\begin{tikzcd}
  Z_2 \ar[d, "{(\id,d_0)}"'] & Q \ar[l] \ar[d] \\
  Z_2\times A \ar[d, "d_2\times \id"']& Z_2\times Y_2 \ar[l, "\id 
  \times d_1"'] \ar[d, "d_2\times \id"] \\
  A \times A & A \times Y_2 \ar[l, "\id\times d_1"]
  \end{tikzcd}
  $$
  which intuitively is
  $$
  \begin{tikzcd}
	\{K {\actto\,} T \} \ar[d, "\text{(forest of image-trees , 
	$K$)}"'] & \{ W{\actto\,} K {\actto\,} T \} 
	\ar[l, "\text{forget $W$}"'] \ar[d, "\text{(forest of image-trees , 
	$W{\actto\,}K$)}"] 
	\\
	\{ S_1\cdots S_k , K \} & \{ S_1\cdots S_k,  W {\actto\,} K \}
	\ar[l, "\text{forget $W$}"]
  \end{tikzcd}
  $$
  The left-hand component of the vertical maps takes $K \actto T$, 
  interprets it as a blobbed tree, and
  returns the forest of trees seen in the blobs. Formally this is given by
  active-inert factorising the maps $C_i \into K \actto T$, where 
  $C_1\cdots C_k$ are the nodes of $K$, as explained in \ref{faceZ}.
  It is clear the square commutes.  To see it is a pullback, compose vertically 
  with the projection onto the second factor:
  $$
  \begin{tikzcd}
	\{K {\actto\,} T \} \ar[d, "\text{(forest of image-trees , 
	$K$)}"'] & \{ W{\actto\,} K {\actto\,} T \} 
	\ar[l, "\text{forget $W$}"'] \ar[d, "\text{(forest of image-trees , 
	$W{\actto\,}K$)}"] 
	\\
	\{ S_1\cdots S_k , K \} \ar[d, "\text{pr}_2"'] & \{ S_1\cdots S_k,  W {\actto\,} K \}
	\ar[l, "\text{forget $W$}"']  \ar[d, "\text{pr}_2"]
	\\
	\{ K \} & \{ W {\actto\,} K \}
	\ar[l, "\text{forget $W$}"] 
  \end{tikzcd}
  $$
  Now the outer rectangle is a pullback (it is the pullback defining $Q$).
  The bottom square is also a pullback, 
  since projecting away an identity map is always a pullback.
  Therefore, by the 
  Prism Lemma~\ref{lem:prism}, also the top square is a pullback, which is 
  the square of interest. 
\end{blanko}

\begin{blanko}{The upper right-hand square is a pullback.}
  The  upper right-hand square is
  \begin{equation}\label{eq:QY2}
  \begin{tikzcd}[column sep = {6em,between origins}]
  Y_2 \ar[r, "{(d_2^Y,d_0^Y)}"] & A\times A  \\
  Q \ar[u] \ar[r] & Z_2 \times Z_2 
  \ar[u, "{d_1^Z \times d_1^Z}"']
  \end{tikzcd}
  \quad
  :
  \quad
  \begin{tikzcd}[column sep = large]
	\{W {\actto\,}T\} \ar[r, "\text{return layers}"] 
	& \{(T',T'')\} 
	\\
	\{ W{\actto\,} K {\actto\,} T \} 
	\ar[r, "\text{return layers}"'] \ar[u, "\text{forget $K$}"] 
	&
	\{ (K'{\actto\,}T', K''{\actto\,}T'')\}  .
	\ar[u, "\text{forget $K$}"'] 
  \end{tikzcd}
  \end{equation}
  
  The horizontal map $Y_2 \stackrel{(d_2^Y,d_0^Y)}\longrightarrow A \times A$
  is described as follows (cf.~\ref{faceY}).
  A (connected) element in $Y_2$ is an active map  $W\actto T$
  where $W$ is a $2$-level tree. By being a $2$-level 
  tree,
  it has a leaf-preserving inert forest inclusion $W' \into W$
  (where $W'$ is a forest of corollas),
  and a root-preserving inert tree inclusion $W'' \into W$ (where $W''$ is 
  just a corolla), as in the 
  solid part of the diagram
  \[\begin{tikzcd}[sep = {3.8em,between origins}]
  W' \ar[d, into] \ar[r, dotted, -act] & T' \ar[d, dotted, into] \dlactinert \\
  W \ar[r, -act] & T \\
  W'' \ar[u, into] \ar[r, dotted, -act] & T'' \ulactinert .
  \ar[u, dotted, into]
  \end{tikzcd}
  \]
  The map $(d_2,d_0) : Y_2 \to A \times A$
  returns the pair $(T',T'')$ consisting of the forest $T'$ and the 
  tree $T''$ appearing in the active-inert factorisation of the two
  maps to $T$.

  The other horizontal map
  $$
  Q \longrightarrow Z_2 \times Z_2
  $$
  is of the same nature.
  A (connected) element in $Q$ is of the form $W\actto K \actto T$
  where $W$ is a $2$-level tree. Again we have the
  solid part of the diagram
  \[\begin{tikzcd}[sep = {3.8em,between origins}]
  W' \ar[d, into] \ar[r, dotted, -act] & 
  K' \dlactinert \ar[d, dotted, into]\ar[r, dotted, -act] & 
  T' \dlactinert \ar[d, dotted, into]  \\
  W \ar[r, -act] & K \ar[r, -act] & T \\
  W'' \ar[u, into] \ar[r, dotted, -act] & 
  K'' \ulactinert \ar[u, dotted, into]\ar[r, dotted, -act] & 
  T'' \ulactinert \ar[u, dotted, into]  .
  \end{tikzcd}\]
  The map $Q \to Z_2 \times Z_2$
  returns the pair $(K'\actto T',K''\actto T'')$.

  The vertical maps in \eqref{eq:QY2} simply forget the trees $K$ (and $K'$
  and $K''$). More formally, let $\kat{Act}(\PPP)$ denote the groupoid
  whose objects are the active maps of $\PPP$-trees.
  The right-hand map in \eqref{eq:QY2} is a product of two copies of
  ($\SSS$ of) $\operatorname{codom}:\kat{Act}(\PPP) \to
  \OMEGA_{\operatorname{iso}}(\PPP)$, which is clearly a fibration: given an
  active map $K \actto T$ and $T \isopil S$, just compose to get an active
  map $K \actto S$.
  To see that the square \eqref{eq:QY2} is a pullback (which is the main
  part of the proof of the theorem), we use the Fibre 
  Lemma~\ref{lem:fibre},
  applied to the {\em strict} fibres of these maps.
  The strict fibre of the map $Q \to Y_2$ over an element $W \actto T$ 
  is the groupoid
  $$
  \operatorname{Fact}(W {\actto} T)
  $$
  of all ways of 
  factoring into $W \actto K \actto T$. On the other hand, the fibre 
  over the corresponding $(W'_1\cdots W'_{k'} , W'' \}) \in A \times A$
  is 
  $$
  \big(\prod_i \operatorname{Fact}(C'_i {\actto} T'_i)\big) \times
  \operatorname{Fact}(C'' {\actto} T'')  .
  $$
  Altogether, the product is over all the nodes of $W$, so now the 
  equivalence of these two groupoids follows from the basic equivalence
  $$
  \operatorname{Act}(W) \simeq \big(\prod_i \operatorname{Act}(C'_i)\big) \times 
  \operatorname{Act}(C'') 
  $$
  of Lemma~\ref{lem:Act(K)}.

  In intuitive terms, the equivalence of the fibres says that to blob a
  tree $T$ compatibly with a $2$-layering (the active map $W \actto
  T$) is the same as blobbing the root-layer tree and all the trees in
  the leaf-layer forest.
\end{blanko}

\begin{blanko}{The lower right-hand square.}\label{pf:LRHS}
  We finally look at the bottom right square:
$$\begin{tikzcd}[column sep = 10em]
  \{ W{\actto\,} K {\actto\,} T \} 
  \ar[r, "\text{return layers}"] \ar[d, "\text{forest of image-trees, 
  $W{\actto\,}K$}"'] 
  &
  \{ (K'{\actto\,}T', K''{\actto\,}T'')\}  
  \ar[d, "\text{ disjoint union of forests of image trees, $K'$, $K''$}"]
  \\
  \{ S_1\cdots S_k,  W {\actto\,} K \} 
  \ar[r, "\text{keep the $S$-forest, return layers}"'] & \{ S_1\cdots S_k,  
  K', K'' \}
\end{tikzcd}$$
where $S_1\cdots S_k = S'_1\cdots S'_{k'} \cdot S''_1 \cdots 
S''_{k''}$. It is clear that the trees appearing in these two expressions 
are the same, but one may worry that they do not come in the same order.
But in fact the monomials are not indexed by linear orders (that is only
for notational convenience) --- in reality they are indexed by the nodes 
of $W$, and as such the two monomials are literally the same.
\end{blanko}

Now for the counit compatibility.

\begin{lemma}\label{lem:counit=comodmap}
  The counit structure of $Y$ is a $Z$-comodule map.
\end{lemma}

\begin{proof}
  We must show that the following diagram commutes up to natural isomorphism:
  $$
  \begin{tikzcd}
	 \Grpd_{/A} \ar[r, "\epsilon_Y"] \ar[d, "\gamma"'] & \Grpd_{/1}
	 \ar[d, "\eta"] \\
	 \Grpd_{/A} \tensor \Grpd_{/A} \ar[r, "\id \tensor \epsilon_Y"'] & 
	 \Grpd_{/A} \tensor \Grpd_{/1}.
  \end{tikzcd}
  $$
  Spelling out the spans that define these functors, we are faced with
  the solid diagram
  $$
  \begin{tikzcd}[sep=large]
	 A & \ar[l, "s_0^Y"'] Y_0 \ar[r] & 1
	 \\
	 Z_2 \ar[u, "d_1^Z"] \ar[d, "{(d_2^Z,d_0^Z)}"']  & 
	 ? \dlpullback \urpullback 
	 \ar[l, dotted] \ar[u, dotted] \ar[r, dotted] \ar[d, dotted] &
	 1 \ar[d, "\eta"] \ar[u, "="'] 
	 \\
	 A \times A& A \times Y_0 \ar[l, "\id \times s_0^Y"] \ar[r]& 
	 A \times 1 .
  \end{tikzcd}
  $$
  This time, as middle object we are forced to take simply $Y_0$, in order 
  to
  make the upper right-hand square a pullback.
    It remains to exhibit the other maps, and check that the squares are
  commutative and pullbacks as indicated. These checks occupy 
  \ref{pf:ULHScounit}--\ref{pf:LRHScounit} below.
\end{proof}

\begin{blanko}{Upper left-hand square (of the counit compatibility check).}
  \label{pf:ULHScounit}
  This square
  \[
  \begin{tikzcd}
	 A & \ar[l, "s_0^Y"'] Y_0 
	 \\
	 Z_2 \ar[u, "d_1^Z"]   & 
	 Y_0  
	 \ar[l] \ar[u, "="] 
   \end{tikzcd}
\]  
  commutes because the map $Z_2 \leftarrow Y_0$ sends a forest $U$ of trivial 
  trees to the identity map $U{\actto}U$. So both way around the 
  square, the result is just $U$ again.
\end{blanko}
  
\begin{blanko}{Lower left-hand square (of the counit compatibility check).}
  The square is
  \[
  \begin{tikzcd}
  Z_2 \ar[d, "{(d_2^Z,d_0^Z)}"'] & Y_0 \ar[l] \ar[d] \\
  A \times A  & A \times Y_0 \ar[l, "\id \times s_0^Y"]
  \end{tikzcd}
  \qquad \qquad
  \begin{tikzcd}
  (U{\actto}U) \ar[d, mapsto] & U \ar[l, mapsto] \ar[d, mapsto]  \\
  (\emptyset, U) & (\emptyset, U) . \ar[l, mapsto]
  \end{tikzcd}
  \]
  The vertical map on the right takes a forest $U$ of trivial trees,
  interpreted as the trivially blobbed trivial forest (hence having no
  blobs) to the pair $(\emptyset, U)$ consisting of the forest of all trees
  seen in the blobs (there are none), and the forest of all the trivial trees. It is clear
  this commutes. To see that it is also a pullback, paste below with the
  square projecting away the identity map:
  \[
  \begin{tikzcd}
  Z_2 \ar[d, "{(d_2^Z,d_0^Z)}"'] & Y_0 \ar[l] \ar[d] \\
  A \times A  \ar[d, "\operatorname{pr}_2"']& 
  A \times Y_0 
  \ar[l, "\id \times s_0^Y"]
  \ar[d, "\operatorname{pr}_2"] \\
  A & Y_0 . \ar[l, "s_0^Y"]
  \end{tikzcd}
  \]
  We first show that the composite square is a pullback. The right-hand 
  composite is the identity map. The left-hand composite,
  is ($\SSS$ applied to) the map $\operatorname{dom}: \kat{Act}(\PPP)
  \to \OMEGA_{\operatorname{iso}}(\PPP)$ sending an active map of
  $\PPP$-trees to its domain, clearly a fibration. To check that the
  composite square is a pullback, we compare the fibres of the vertical maps
  over an element $U \in Y_0$ ($U$ is a trivial forest). The fibre of the
  identity map is of course singleton. The fibre of $\operatorname{dom}:
  \kat{Act}(\PPP) \to \OMEGA_{\operatorname{iso}}(\PPP)$ is contractible by
  Lemma~\ref{lem:Act(K)}. So the composite square is a pullback by the
  Fibre Lemma~\ref{lem:fibre}. But the project-away-the-identity square is
  also a pullback. Therefore, by the Prism Lemma~\ref{lem:prism}, also the
  top is square is a pullback, as required.
\end{blanko}

\begin{blanko}{Lower right-hand square (of the counit compatibility check).}
  \label{pf:LRHScounit}
  Both ways around send a forest $U$ of trivial trees to the empty 
  forest $\emptyset$ (of all the trees seen in the zero blobs).
\end{blanko}

\bigskip

This finishes the proof of Theorem~\ref{thm:main}.

\section{Locally finite version of the incidence comodule-bialgebra construction}

\label{sec:finiteness}

\subsection{Finiteness conditions and the reduced Baez--Dolan construction 
$\PPP\bdr$}

At the objective level, {\em any} poset, category, or decomposition space
defines a coalgebra~\cite{Galvez-Kock-Tonks:1512.07573}. However, in order
to take cardinality to arrive at an ordinary coalgebra in
vector spaces, it is necessary to impose the finiteness condition that the
poset, category, or decomposition space be locally
finite~\cite{Galvez-Kock-Tonks:1512.07577}. This condition says that the
two maps
\begin{equation}
  X_0 \stackrel{s_0}\to X_1 , \qquad X_1 \stackrel{d_1}{\leftarrow} X_2
\end{equation}
have finite (homotopy) fibres.\footnote{In
\cite{Galvez-Kock-Tonks:1602.05082} it is furthermore required that $X_1$
is homotopy finite, but this has turned out to be a superfluous 
requirement.} Indeed, the formulae
for the counit and comultiplication amount to summing over these fibres.

\begin{blanko}{Locally finite operads.}
  \label{fin-cond}
  An operad $\PPP$ is called {\em locally finite}
  \cite{Kock-Weber:1609.03276} if its two-sided bar construction 
  $\TSB{\SSS}{\PPP}$ is locally
  finite. This is equivalent to demanding directly on the monad that the
  structure maps $\mu: \PPP \circ \PPP \Rightarrow \PPP$ and $\eta: \Id
  \Rightarrow \PPP$ be finite. (Here, by definition, a map in $\Grpd_{/I}$
  is {\em finite} if its image in $\Grpd$ is finite.)
\end{blanko}

\begin{lemma}\label{lem:freefinite}
  For any operad $\PPP$, the free operad $\PPP\upperstar$ is locally 
  finite. 
\end{lemma}

\begin{proof}
  Put $Y = \TSB{\SSS}{\PPP\upperstar}$. The fibre of $Y_1
  \stackrel{d_1}\leftarrow Y_2$ over a $\PPP$-tree $T$ is the discrete
  groupoid of all ways to cut the tree $T$, or more formally, the active
  maps $W \actto T$, where $W$ is a $2$-level tree. It is clear that $T$
  has only finitely many possible cuts. The map $Y_0 \stackrel{s_0}\to Y_1$
  is even a monomorphism, and therefore in particular is finite.
\end{proof}

\begin{lemma}\label{lem:BDnotfinite}
  The Baez--Dolan construction $\PPP\bd$ is never locally finite.
\end{lemma}

\begin{proof}
  Put $Z= \TSB{\SSS}{\PPP\bd}$. The fibre of $Z_1 \stackrel{d_1}\leftarrow
  Z_2$ over a $\PPP$-tree is the discrete groupoid of all ways of blobbing
  the tree, as in \ref{blobbing}. Since each edge admits an arbitrary
  number of trivial blobs, this is an infinite set.
\end{proof}

\begin{blanko}{The reduced Baez--Dolan construction.}
  For the sake of finiteness, and to be able to take cardinality, we
  need to work instead with the {\em reduced} Baez--Dolan construction
  $$
  \PPP\bdr:=\overline{\PPP\bd}
  $$
  (already considered by Baez and
  Dolan~\cite{Baez-Dolan:9702}). This is simply the operad $\PPP\bd$
  with all nullary operations removed. Recall that the operations of 
  $\PPP\bd$ are the $\PPP$-trees, and that the nullary operations are 
  the trivial $\PPP$-trees; we are thus excluding trivial $\PPP$-trees.  
  
  The two-sided bar construction of the reduced Baez--Dolan construction
  $\PPP\bdr$ will be denoted
  $$
  Z := \TSB{\SSS}{\PPP\bdr} .
  $$
  For the rest of the paper, this will replace $Z = 
  \TSB{\SSS}{\PPP\bd}$,
  which will no longer be considered.
  
  Excluding nullary operations means disallowing trivial trees and 
  trivial blobs. Intuitively, for the lowest degrees of $Z$ we have:

  $\bullet$ \
  $Z_0$ is the groupoid of monomials of $\PPP$-corollas;

  $\bullet$ \
  $Z_1$ is the groupoid of monomials of nontrivial $\PPP$-trees;

  $\bullet$ \
  $Z_2$ is the groupoid of monomials of nontrivial $\PPP$-trees with only nontrivial 
  blobs.
  
  Let us describe $Z$ more formally:
\end{blanko}

\begin{blanko}{Active injections, reduced covers, spanning forests.}
  The category of trees $\OMEGA$ (and the category of $\PPP$-trees
  $\OMEGA(\PPP)$) has another factorisation system than the active-inert
  system exploited so far, namely the surjective-injective
  factorisation system. (The notions injective and surjective refer to
  the effect on edges.) All inert maps are injective, but the active
  maps come in two flavours: {\em active injections} which refine
  nodes into nontrivial trees (these are generated by the active
  coface maps), and {\em active surjections} which refine unary nodes
  into trivial trees (these are generated by the codegeneracy maps).

  We can now describe $Z_k$  more formally as the groupoid of sequences
  of active injections
  $$
  K^{(n)} \actto K^{(n-1)} \actto \cdots \actto K^{(1)} \actto K^{(0)}=T \to \PPP   ,
  $$
  with the condition that $K^{(n)}$ is a forest of corollas. This is just
  like in \ref{faceZ}, except we now require active {\em injections}
  instead of arbitrary active maps.

  The opposite of the category of active injections into a fixed
  nontrivial tree $T$ is a preorder equivalent to the poset of reduced
  covers of $T$ (cf.~\cite[2.3.2]{Kock:0807}). Both are equivalent to
  the power set of the set of inner edges in $T$. Here a {\em reduced
  cover} of $T$ is an inert map of forests $\sum_i R_i \into T$ which
  is bijective on nodes, and where the $R_i$ are nontrivial trees. It
  could also be called a spanning forest. The correspondence goes like
  this: given an active injection $K \actto T$,
  let $R_i$ be the
  subtrees of $T$ arising from active-inert factorisation of composite
  maps $C_i \into K \actto T$ as in \ref{faceZ} (where as usual the 
  $C_i$ are the nodes
  of $K$).
  
  With this correspondence, $Z_2$ can be described also as the groupoid 
  of reduced covers of trees, and $Z_3$ can be described as the 
  groupoid of reduced covers of reduced covers.  The active-injections
  interpretation is good for describing 
  $\begin{tikzcd} 
  Z_1 & Z_2
  \ar[l, shift left, "d_0"]
  \ar[l, shift right, "d_1"']
  \end{tikzcd}$ (but not $d_2$): 
  we have $d_0( K{\actto}T) = K$ and 
  $d_1(K{\actto}T) = T$. The reduced-covers interpretation is 
  convenient for describing
  $\begin{tikzcd}
  Z_1 & Z_2
  \ar[l, shift left, "d_1"]
  \ar[l, shift right, wavy, "d_2"']
  \end{tikzcd}$ (but not $d_0$):
  we have $d_1( \sum_i R_i{\into}T) = T$ and 
  $d_2( \sum_i R_i{\into}T) = \sum_i R_i$.
  
  In the following it will be practical to favour the active-injections 
  interpretation, but we will have to convert to the reduced-covers 
  viewpoint each time we describe a top face map.
\end{blanko}

\begin{lemma}\label{lem:BDRfin}
  For any operad $\PPP$ the reduced Baez--Dolan construction $\PPP\bdr$ 
  is locally finite.
\end{lemma}

\begin{proof}
  The fibre of $Z_1 \stackrel{d_1}\leftarrow Z_2$ over a
  $\PPP$-tree $T$ is the discrete groupoid of all active injections 
  $K \actto T$. This is a finite set since there are only finitely many 
  nodes in $T$.
  Put in other terms, the fibre is the discrete groupoid of all  
  ways of blobbing the
  tree in such a way that each blob contains at least one node. 
  The degeneracy map $Z_0 \stackrel{s_0}\to Z_1$ assigns
  to a corolla the same corolla with a single blob around it. This map
  is even a mono, so in particular finite.
\end{proof}

\begin{blanko}{`General trees as comodule over nontrivial trees'.}
  We now have a locally finite simplicial groupoid $Z = 
  \TSB{\SSS}{\PPP\bdr}$ (which is a 
  symmetric monoidal Segal space), so its incidence bialgebra
  admits a homotopy cardinality. But now it no longer has the same 
  underlying space as the incidence bialgebra of 
  $Y=\TSB{\SSS}{\PPP\upperstar}$, and some further adjustments are 
  required. It is not possible to adjust $Y$ in the same way, because
  $Y_0$ consists entirely of trivial trees, so these cannot just be thrown away.
  Instead we need a separate
  comodule structure on $Y_1$, which
  in the objective setting should be a comodule configuration $u: M 
  \to Z$ with $M_0 = Y_1$.
  In rough terms we need to exhibit `general $\PPP$-trees as a comodule
  over nontrivial $\PPP$-trees'.
  (Note that it would not work simply to use $\TSB{\SSS}{\PPP\bd}$ as 
  $M$, because it cannot possibly be culf over $\TSB{\SSS}{\PPP\bdr}$:
  there are obviously more ways of drawing blobs on a tree than 
  drawing nontrivial blobs.)
  
  The simplicial groupoid $M$ is finally going to have

  $\bullet$ \
  $M_0$ the groupoid of (monomials of) arbitrary trees (possibly 
  trivial);

  $\bullet$ \
  $M_1$ the groupoid of (monomials of) arbitrary trees (possibly trivial)
  with nontrivial blobs.

  One can fiddle with these conditions, figure out what the higher $M_k$
  should be, assemble them into a simplicial groupoid, and prove that it is
  culf over $Z$, so as to form indeed a comodule configuration. Rather
  than doing this by hand, we shall embark on a small detour to be able to
  deduce these properties from a general construction: we shall
  define a relative two-sided bar construction
  $C := F\lowershriek\TSB{\PPP\bd}{\PPP\bdr}$ and then put $M := \SSS C$.
  General principles will then imply that $M$ culf over $Z$, and that $M$
  is locally finite as a comodule, 
  as required.
\end{blanko}

\subsection{Further bar constructions and a general comodule construction}

\label{sec:furtherbar}
  
In this subsection, we exploit further two-sided bar constructions to give
an abstract construction of comodules from a pair of operads, one cartesian
over the other.

\begin{blanko}{Set-up.}\label{set-up}
  We place ourselves in the situation of an operad map 
  $\RRR\Rightarrow\PPP$, in the form of polynomial monads
  related by monad opfunctors
$$
\RRR\Rightarrow \PPP \Rightarrow \SSS ,
$$
altogether represented by polynomial diagrams
  $$\begin{tikzcd}
    \RRR : & J \ar[d, "G"']& \ar[l]  U\ar[d] \drpullback \ar[r] & V\ar[d] \ar[r] & 
    J\ar[d, "G"] \\
    \PPP : & I \ar[d, "F"']& \ar[l]  E\ar[d] \drpullback \ar[r] & B\ar[d] \ar[r] & 
    I\ar[d, "F"] \\
    \SSS : & 1  &\ar[l] \B'\ar[r] & \B \ar[r]  &1 .
  \end{tikzcd}
  $$
  The monad opfunctor $\RRR \Rightarrow\PPP$ is given by
  the functor $G\lowershriek$ and a natural transformation
  $$
  \theta : G\lowershriek \RRR \Rightarrow \PPP G\lowershriek ;
  $$
  the monad opfunctor $\PPP\Rightarrow\SSS$ is given by the functor
  $F\lowershriek$ and a natural transformation
  $$
  \psi : F\lowershriek \PPP \Rightarrow \SSS F\lowershriek   ,
  $$
  both satisfying the axioms of \ref{opf-ax}. The composite exhibits also 
  $\RRR$ as an operad, with monad opfunctor $\RRR\Rightarrow\SSS$ given by the functor
  $F\lowershriek G\lowershriek$ and the natural transformation
  $$
  \phi : F\lowershriek G \lowershriek \RRR 
  \Rightarrow \SSS F\lowershriek G\lowershriek   ,
  $$
  which is simply $\psi G\lowershriek \circ F\lowershriek \theta$.
\end{blanko}

\begin{prop}\label{prop:RP}
  From operads $\RRR \Rightarrow \PPP \Rightarrow \SSS$ as in
  \ref{set-up}, there is induced a simplicial map $F\lowershriek
  \TSB{\PPP}{\RRR} \to \TSB{\SSS}{\RRR}$, and it is culf. In other
  words, $F\lowershriek \TSB{\PPP}{\RRR}$ is a comodule configuration
  over $\TSB{\SSS}{\RRR}$.
\end{prop}

\begin{proof}
  The two-sided bar constructions
  $\TSB{\SSS}{\RRR}$ and $F\lowershriek\TSB{\PPP}{\RRR}$ are the
  top and bottom rows of the diagram
  \[\begin{tikzcd}[column sep = 5em]
  \ar[r, phantom, "\TSB{\SSS}{\RRR} :" description] &
  \SSS  F\lowershriek G\lowershriek \; 1 
  \ar[r, pos=0.65, "s_0" on top] 
  &
  \ar[l, shift left=6pt, pos=0.65, "d_0" on top]
  \ar[l, wavy, shift right=6pt, pos=0.65, "d_1" on top]
  \SSS  F\lowershriek G\lowershriek \; \RRR  1 
  \ar[r, shift right=6pt, pos=0.65, "s_0" on top]
  \ar[r, shift left=6pt, pos=0.65, "s_1" on top]  
  &
  \ar[l, shift left=12pt, pos=0.65, "d_0" on top]
  \ar[l, wavy, shift right=12pt, pos=0.65, "d_2" on top]
  \ar[l, pos=0.65, "d_1" on top]
  \SSS  F\lowershriek G\lowershriek \; \RRR\RRR 1
    \ar[r, phantom, "\dots" on top]
  & {}
  \\ 
  \\
  \ar[r, phantom, "F\lowershriek\TSB{\PPP}{\RRR} :" description] &
  F\lowershriek \PPP   G\lowershriek \; 1
  \ar[uu, wavy, "{\psi_{G\lowershriek 1}}"]
  \ar[r, pos=0.65, "s_0" on top] 
  &
  \ar[l, shift left=6pt, pos=0.65, "d_0" on top]
  \ar[l, shift right=6pt, pos=0.65, "d_1" on top]
  F\lowershriek \PPP    G\lowershriek  \; \RRR 1
  \ar[uu, wavy, "{\psi_{G\lowershriek \RRR 1}}"]
  \ar[r, shift right=6pt, pos=0.65, "s_0" on top]  
  \ar[r, shift left=6pt, pos=0.65, "s_1" on top]  
  &
  \ar[l, shift left=12pt, pos=0.65, "d_0" on top]
  \ar[l, shift right=12pt, pos=0.65, "d_2" on top]
  \ar[l, pos=0.65, "d_1" on top]
  F\lowershriek \PPP    G\lowershriek  \; \RRR\RRR 1 
  \ar[uu, wavy, "{\psi_{G\lowershriek \RRR\RRR 1}}"]
  \ar[r, phantom, "\dots" on top]
  & {}
  \end{tikzcd}\]
  The vertical comparison maps are components of the natural transformation
  $\psi$. We first check that this is a simplicial map. By naturality of
  $\psi$, it is clear that the vertical maps commute with all degeneracy
  and face maps except perhaps the top face maps. The top face maps are
  special since they involve the monad multiplication of $\SSS$, and
  require a separate check: for $k\geq 0$, put $A:= \RRR^k 1$. The
  compatibility with the top face between degree $k+1$ and $k$ is
  commutativity of the outline of the diagram
	$$\begin{tikzcd}[row sep=tiny]
	& \SSS\SSS F\lowershriek G\lowershriek A 
	\ar[ld, "\mu^\SSS"']
	&
	\\
	\SSS F\lowershriek G\lowershriek A & & 
	\SSS F\lowershriek G\lowershriek \RRR A 
	\ar[lu, "\SSS \phi_A"']
	\ar[ld, "\SSS F\lowershriek \theta_A"]
	\\
	& \SSS F \lowershriek \PPP G \lowershriek A 
	\ar[uu, "\SSS \psi_{G\lowershriek A}"]
	&
	\\
	F\lowershriek \PPP G \lowershriek A 
	\ar[uu, "\psi_{G\lowershriek A}"]
	&& F\lowershriek \PPP G \lowershriek \RRR A
	\ar[uu, "\psi_{G\lowershriek \RRR A}"']
	\ar[ld, "F\lowershriek \PPP \theta_A"]
	\\
	& F\lowershriek \PPP\PPP G\lowershriek A
	\ar[uu, "\psi_{\PPP G\lowershriek A}"]
	\ar[lu, "F\lowershriek \mu^\PPP"]
	&
	\end{tikzcd}$$
  Commutativity of the pentagon on the left is a monad-opfunctor 
  axiom~\eqref{eq:colax} for
  $\psi$. The triangle is ($\SSS$ applied to) the definition of $\phi$.
  Commutativity of the square is naturality of $\psi$. Furthermore, the
  simplicial map is cartesian on the non-wavy part because $\psi$ is a
  cartesian natural transformation. In particular the simplicial map is
  culf, and since two-sided bar constructions are Segal 
  this makes the bottom row $F\lowershriek \TSB{\PPP}{\RRR}$
  a comodule over the top row $\TSB\SSS\RRR$.
\end{proof}

\begin{blanko}{Finite operad maps.}\label{finiteopmaps}
  An operad map, in the form of a monad opfunctor $ (G,\theta) : \RRR\Rightarrow\PPP$ represented by 
  $$\begin{tikzcd}
    \RRR: & J \ar[d, "G"']& \ar[l]  U\ar[d, "K"'] \drpullback \ar[r] & V\ar[d, 
	"H"] \ar[r] & 
    J\ar[d, "G"] \\
    \PPP: & I & \ar[l]  E \ar[r] & B \ar[r] & 
    I  ,
  \end{tikzcd}
  $$
  is called {\em finite} when the maps $G$ and $H$ are finite (then 
  the map $K$ is finite too, by pullback).
\end{blanko}

\begin{lemma}\label{lem:theta=finite}
  If $(G, \theta): \RRR\Rightarrow \PPP$ is finite, then all components of $\theta$ are 
  finite.
\end{lemma}
\begin{proof}
  It is enough to check that $\theta_1$ is finite, because $\theta$ is a
  cartesian natural transformation, and therefore all the other components
  of $\theta$ are pullbacks of $\theta_1$. The map $\theta_1$ is related to
  $G$ and $H$ by the diagram
  $$
  \begin{tikzcd}[column sep={1.6em,between origins}]
	V \ar[dd, "H"'] & = \phantom{l}& \RRR(J) \ar[d, "\theta_1"] \\
	&& \PPP(J) \ar[d, "\PPP(G)"] \\
	B &= \phantom{l}& \PPP(I) .
  \end{tikzcd}$$
  But $G$ and $H$ are finite by assumption, and $\PPP(G)$ is
  too (because endofunctors underlying operads preserve finite maps).
  Now it follows from the next lemma that 
  also $\theta_1$ is finite.
\end{proof}

\begin{lemma}\label{lem:fibseqfinite}
  Given maps of groupoids $A \stackrel{f}\to B \stackrel{g}\to C$,
  if two out of $f$, $g$, and $g\circ f$ are finite then so is the third.
\end{lemma}
\begin{proof}
  With $b\in B$ and $c = g(b)$, consider the pullback diagram 
  $$\begin{tikzcd}
  A_b \drpullback \ar[d] \ar[r] & A_c \drpullback \ar[d] \ar[r] & A \ar[d, "f"] \\
  1 \ar[r, "\name{b}"'] & B_c \drpullback \ar[d] \ar[r] & B \ar[d, "g"] \\
  & 1 \ar[r, "\name{c}"'] & C  .
  \end{tikzcd}
  $$
  Here $A_b$, $B_c$, and $A_c$ are fibres of the maps $f$, $g$, and $g\circ
  f$, respectively. Now the result follows from the 2-out-of-3 property for
  finiteness \cite{Galvez-Kock-Tonks:1602.05082} in the fibre sequence $A_b
  \to A_c \to B_c$.
\end{proof}

\begin{lemma}\label{lem:finfinfin}
  If $\RRR\Rightarrow \PPP$ is finite and $\RRR$ is locally finite, then
  the simplicial groupoid $C = F\lowershriek\TSB{\PPP}{\RRR}$ is locally 
  finite (as a comodule configuration). This means that
  the face map $C_0 \stackrel{d_1}\leftarrow C_1$ is finite.
\end{lemma}

\begin{proof}
  The solid diagram
  $$
  \begin{tikzcd}[column sep={4em,between origins}]
	\PPP \PPP G\lowershriek 1 \ar[rd, "\mu^\PPP_{G\lowershriek 1}"'] && 
	\PPP G\lowershriek \RRR 1 \ar[ll, "\PPP(\theta_{1})"'] \ar[ld, dotted] && G\lowershriek 
	\RRR\RRR 1 \ar[ll, "\theta_{\RRR 1}"'] \ar[ld, "G\lowershriek \mu^\RRR_1"] 
	\\
	& \PPP G\lowershriek 1&& G\lowershriek \RRR 1    \ar[ll, "\theta_1"] & 
  \end{tikzcd}
  $$
  expresses one of the axioms~\eqref{eq:colax} for the monad opfunctor $\RRR\Rightarrow
  \PPP$. The maps $\theta_1$ and $\theta_{\RRR 1}$ are finite by 
  Lemma~\ref{lem:theta=finite} since
  the operad map is finite.  Furthermore, $G\lowershriek 
  \mu^\RRR_1$ is finite since $\RRR$ is assumed locally finite.
  It now follows from  Lemma~\ref{lem:fibseqfinite} that the dotted
  arrow is finite.  But the map we are concerned with, $C_0 
  \stackrel{d_1}\leftarrow C_1$, is $F\lowershriek$ applied to this dotted 
  map (and lowershrieks preserve finiteness).
\end{proof}

\begin{blanko}{Free $\SSS$-algebra on a comodule.}
  For the desired application of this construction, we will need to pass to
  a comodule of {\em monomials} of $\RRR$-operations. This is achieved in a
  canonical way since $\TSB{\SSS}{\RRR}$ is a symmetric monoidal
  decomposition space. The $\TSB{\SSS}{\RRR}$-comodule structure on $\SSS
  F\lowershriek\TSB{\PPP}{\RRR}$ is given by
  $$
  \SSS F\lowershriek\TSB{\PPP}{\RRR} \to \SSS \TSB{\SSS}{\RRR} \to 
  \TSB{\SSS}{\RRR} ,
  $$
  where the last map is the symmetric monoidal structure. This composite is
  again culf, because $\SSS$ preserves culfness, and the structure map
  itself is culf (\ref{S}). With shorthand notation $Z:=
  \TSB{\SSS}{\RRR}$ and $C:=F\lowershriek\TSB{\PPP}{\RRR}$, the relevant
  span from this comodule configuration is
  $$
  \SSS C_0 \stackrel{\SSS(d_1)}\longleftarrow 
  \SSS C_1 \stackrel{\SSS(u,d_0)}\longrightarrow 
  \SSS(Z_1\times C_0) \isopil 
  \SSS Z_1\times_{\SSS 1} \SSS C_0 \longrightarrow 
  \SSS Z_1\times \SSS C_0 \longrightarrow 
  Z_1 \times \SSS C_0 
  $$
  (the two middle maps expressing together that $\SSS$ is colax  monoidal with 
  respect to the cartesian product).
\end{blanko}

\subsection{Main theorem, locally finite version}

\begin{blanko}{Set-up.}
  Let $\PPP$ be any operad, represented by $I \leftarrow E \to B \to 
  I$, and let $F$ denote either of the maps $I \to 1$ and $B \to 1$. (In 
  any case we use $F\lowershriek$ only to move from slices to $\Grpd$.)
  
  We now instantiate the constructions of \ref{sec:furtherbar} to the
  operad map
  $$
  \PPP\bdr \Rightarrow \PPP\bd .
  $$
  This map is finite: in the notation of \ref{finiteopmaps},
  $G$ is the identity, and $V \to B$ is
  the monomorphism given by inclusion of the nontrivial part, and in 
  particular is finite too.
  As before, we put
  $$
  Z := \TSB{\SSS}{\PPP\bdr}, \qquad
  C := F\lowershriek\TSB{\PPP\bd}{\PPP\bdr}, \qquad
  M := \SSS C.
  $$
\end{blanko}

\begin{lemma}\label{lem:comod-config-locfin}
  The simplicial map 
  $$
  u: F\lowershriek\TSB{\PPP\bd}{\PPP\bdr} \to \TSB{\SSS}{\PPP\bdr} ,
  $$
  is culf, and hence constitutes a comodule configuration. This comodule 
  configuration is furthermore locally finite.
\end{lemma}
\begin{proof}
  That $u$ is a comodule configuration is an immediate consequence of
  Proposition~\ref{prop:RP}. Local finiteness follows from 
  \ref{lem:finfinfin}, since clearly $\PPP\bdr\Rightarrow\PPP\bd$ is finite.
\end{proof}

The comodule configuration expands to
  \[\begin{tikzcd}[column sep = 5em]
  Z \quad :&
  \SSS  F\lowershriek \; 1 
  \ar[r, pos=0.65, "s_0" on top] 
  &
  \ar[l, shift left=6pt, pos=0.65, "d_0" on top]
  \ar[l, wavy, shift right=6pt, pos=0.65, "d_1" on top]
  \SSS  F\lowershriek  \; \PPP\bdr 1 
  \ar[r, shift right=6pt, pos=0.65, "s_0" on top]
  \ar[r, shift left=6pt, pos=0.65, "s_1" on top]  
  &
  \ar[l, shift left=12pt, pos=0.65, "d_0" on top]
  \ar[l, wavy, shift right=12pt, pos=0.65, "d_2" on top]
  \ar[l, pos=0.65, "d_1" on top]
  \SSS  F\lowershriek \; \PPP\bdr\PPP\bdr 1
  \ar[r, phantom, "\dots" on top]
  & {}
  \\ 
  \\
  C \quad : &
  F\lowershriek \PPP\bd  \; 1
  \ar[uu, wavy, "{\psi_{1}}"]
  \ar[r, pos=0.65, "s_0" on top] 
  &
  \ar[l, shift left=6pt, pos=0.65, "d_0" on top]
  \ar[l, shift right=6pt, pos=0.65, "d_1" on top]
  F\lowershriek \PPP\bd    \; \PPP\bdr 1
  \ar[uu, wavy, "{\psi_{\PPP\bdr 1}}"]
  \ar[r, shift right=6pt, pos=0.65, "s_0" on top]  
  \ar[r, shift left=6pt, pos=0.65, "s_1" on top]  
  &
  \ar[l, shift left=12pt, pos=0.65, "d_0" on top]
  \ar[l, shift right=12pt, pos=0.65, "d_2" on top]
  \ar[l, pos=0.65, "d_1" on top]
  F\lowershriek \PPP\bd  \; \PPP\bdr\PPP\bdr 1 
  \ar[uu, wavy, "{\psi_{\PPP\bdr\PPP\bdr 1}}"]
  \ar[r, phantom, "\dots" on top]
  & {}
  \end{tikzcd}\]
In the first row, we have arbitrary corollas, then nontrivial trees, then
nontrivial trees with only nontrivial blobs. In the second row, we have
arbitrary trees, then arbitrary trees with nontrivial blobs, then arbitrary
trees with nested nontrivial blobbings, and so on. In pictures:
  
  \begin{equation}\label{eq:bdr}
  \begin{tikzpicture}[line width=0.25mm]

	\begin{scope}[shift={(0.77, 0.07)}] 
	  \draw (0.0, 0.525) -- (0.0, 1.05);
	  \fill (0.0, 1.05) circle[radius=0.065];
	  \draw (0.0, 1.05) -- (-0.63, 1.575); 
	  \draw (0.0, 1.05) -- (-0.21, 1.855);
	  \draw (0.0, 1.05) -- (0.21, 1.855);
	  \draw (0.0, 1.05) -- (0.63, 1.575);
	  \draw (-1.33, 1.225) node {$\SSS$};
	  \draw (-0.98, 1.225) node {$\bigleftbrace{20}$};
	  \draw (0.98, 1.225) node {$\bigrightbrace{20}$};
	\end{scope}

	\begin{scope}[shift={(1.785, 0.0)}]
	  \draw[->, wavy] (2.1, 1.645) -- +(-1.4, 0.0);
	  \draw[->] (2.1, 0.945) -- +(-1.4, 0.0);
	  \draw (1.4, 1.86) node {\scriptsize nodes};
	  \draw (1.4, 1.16) node {\scriptsize residue};
	\end{scope}

	\begin{scope}[shift={(6.23, 0.0)}] 
	  \draw (0.0, 0.175) -- (0.0, 0.595);
	  \fill (0.0, 0.595) circle[radius=0.065];
	  \draw (0.0, 0.595) -- (-0.35, 1.05);
	  \fill (-0.35, 1.05) circle[radius=0.065];
	  \draw (-0.35, 1.05) -- (-1.05, 1.225);
	  \draw (-0.35, 1.05) -- (-0.7, 1.75);
	  \fill (-0.7, 1.75) circle[radius=0.065];
	  \draw (-0.35, 1.05) -- (-0.35, 2.275);
	  \draw (-0.35, 1.05) -- (0.0, 2.275);
	  \draw (0.0, 0.595) -- (0.56, 2.275);
	  \fill (0.28, 1.435) circle[radius=0.065];
	  \fill (0.42, 1.855) circle[radius=0.065];
	  \draw (-0.7,0.5) node {\tiny nontrivial};

	  \draw (-1.75, 1.33) node {$\SSS$};
	  \draw (-1.4, 1.33) node {$\bigleftbrace{30}$};
	  \draw (1.155, 1.33) node {$\bigrightbrace{30}$};
	\end{scope}

	\begin{scope}[shift={(7.595, 0.0)}]
	  \draw[->, wavy] (2.1, 1.995) -- +(-1.4, 0.0);
	  \draw[->] (2.1, 1.295) -- +(-1.4, 0.0);
	  \draw[->] (2.1, 0.595) -- +(-1.4, 0.0);
	  \draw (1.4, 2.20) node {\scriptsize blobs};
	  \draw (1.4, 1.47) node {\scriptsize forget blobs};
	  \draw (1.4, 0.81) node {\scriptsize contract blobs};
	\end{scope}

	\begin{scope}[shift={(12.18, 0.0)}] 
	  \draw (0.0, 0.175) -- (0.0, 0.595);
	  \fill (0.0, 0.595) circle[radius=0.065];
	  \draw (0.0, 0.595) -- (-0.35, 1.05);
	  \fill (-0.35, 1.05) circle[radius=0.065];
	  \draw (-0.35, 1.05) -- (-1.05, 1.225);
	  \draw (-0.35, 1.05) -- (-0.7, 1.75);
	  \fill (-0.7, 1.75) circle[radius=0.065];
	  \draw (-0.35, 1.05) -- (-0.175, 1.575);
	  \fill (-0.175, 1.575) circle[radius=0.065];
	  \draw (-0.175, 1.575) -- (-0.35, 2.275);
	  \draw (-0.175, 1.575) -- (0.0, 2.275);
	  \draw (0.0, 0.595) -- (0.56, 2.275);
	  \draw (0.0, 0.595) -- (0.42, 0.875);
	  \fill (0.42, 0.875) circle[radius=0.065];

	  	  \draw (-0.75,0.5) node {\tiny nontrivial};

	  \fill (0.28,1.435) circle[radius=0.065];
	  \fill (0.42,1.855) circle[radius=0.065];
	  \draw[line width=0.01, rotate around={70:(-0.28,1.33)}] (-0.28,1.33)  ellipse (0.56 and 0.245);
	  \draw[line width=0.01, rotate around={34:(0.21,0.77)}] (0.21,0.77)  ellipse (0.49 and 0.245);
	  \draw[line width=0.01, rotate around={70:(-0.7,1.75)}] (-0.7,1.75)  ellipse (0.17 and 0.17);
	  \draw[line width=0.01, rotate around={70:(0.28,1.435)}] (0.28,1.435)  ellipse (0.17 and 0.17);
	  \draw[line width=0.01, rotate around={70:(0.42,1.855)}] (0.42,1.855)  ellipse (0.17 and 0.17);
	  \draw (-1.75, 1.33) node {$\SSS$};
	  \draw (-1.4, 1.33) node {$\bigleftbrace{30}$};
	  \draw (1.155, 1.33) node {$\bigrightbrace{30}$};
	\end{scope}


	\begin{scope}[shift={(0.77, -4.07)}] 
	  \draw (0.0, 0.175) -- (0.0, 0.595);
	  \fill (0.0, 0.595) circle[radius=0.065];
	  \draw (0.0, 0.595) -- (-0.35, 1.05);
	  \fill (-0.35, 1.05) circle[radius=0.065];
	  \draw (-0.35, 1.05) -- (-1.05, 1.225);
	  \draw (-0.35, 1.05) -- (-0.7, 1.75);
	  \fill (-0.7, 1.75) circle[radius=0.065];
	  \draw (-0.35, 1.05) -- (-0.35, 2.275);
	  \draw (-0.35, 1.05) -- (0.0, 2.275);
	  \draw (0.0, 0.595) -- (0.56, 2.275);
	  \fill (0.28, 1.435) circle[radius=0.065];
	  \fill (0.42, 1.855) circle[radius=0.065];
	  \draw (-0.7,0.55) node {\tiny general};
	  \draw (-0.7,0.3) node {\tiny trees};
	  \draw (-1.4, 1.33) node {$\bigleftbrace{30}$};
	  \draw (1.155, 1.33) node {$\bigrightbrace{30}$};
	\end{scope}

	\begin{scope}[shift={(1.785, -4.0)}]
	  \draw[->] (2.1, 1.645) -- +(-1.4, 0.0);
	  \draw[->] (2.1, 0.945) -- +(-1.4, 0.0);
	  \draw (1.4, 1.86) node {\scriptsize forget blobs};
	  \draw (1.4, 1.16) node {\scriptsize contract blobs};
	\end{scope}

	\begin{scope}[shift={(6.23, -4.0)}] 
	  \draw (0.0, 0.175) -- (0.0, 0.595);
	  \fill (0.0, 0.595) circle[radius=0.065];
	  \draw (0.0, 0.595) -- (-0.35, 1.05);
	  \fill (-0.35, 1.05) circle[radius=0.065];
	  \draw (-0.35, 1.05) -- (-1.05, 1.225);
	  \draw (-0.35, 1.05) -- (-0.7, 1.75);
	  \fill (-0.7, 1.75) circle[radius=0.065];
	  \draw (-0.35, 1.05) -- (-0.175, 1.575);
	  \fill (-0.175, 1.575) circle[radius=0.065];
	  \draw (-0.175, 1.575) -- (-0.35, 2.275);
	  \draw (-0.175, 1.575) -- (0.0, 2.275);
	  \draw (0.0, 0.595) -- (0.56, 2.275);
	  \draw (0.0, 0.595) -- (0.42, 0.875);
	  \fill (0.42, 0.875) circle[radius=0.065];
	  \draw (-0.75,0.55) node {\tiny general};
	  \draw (-0.75,0.3) node {\tiny trees,};
	  \draw (-0.75,0.05) node {\tiny nontrivial};
	  \draw (-0.75,-0.2) node {\tiny blobs};

	  \fill (0.28,1.435) circle[radius=0.065];
	  \fill (0.42,1.855) circle[radius=0.065];
	  \draw[line width=0.01, rotate around={70:(-0.28,1.33)}] (-0.28,1.33)  ellipse (0.56 and 0.245);
	  \draw[line width=0.01, rotate around={34:(0.21,0.77)}] (0.21,0.77)  ellipse (0.49 and 0.245);
	  \draw[line width=0.01, rotate around={70:(-0.7,1.75)}] (-0.7,1.75)  ellipse (0.17 and 0.17);
	  \draw[line width=0.01, rotate around={70:(0.28,1.435)}] (0.28,1.435)  ellipse (0.17 and 0.17);
	  \draw[line width=0.01, rotate around={70:(0.42,1.855)}] (0.42,1.855)  ellipse (0.17 and 0.17);
	  \draw (-1.4, 1.33) node {$\bigleftbrace{30}$};
	  \draw (1.155, 1.33) node {$\bigrightbrace{30}$};
	\end{scope}
	
	\begin{scope}[shift={(7.595, -4.0)}]
	  \draw[->] (2.1, 1.995) -- +(-1.4, 0.0);
	  \draw[->] (2.1, 1.295) -- +(-1.4, 0.0);
	  \draw[->] (2.1, 0.595) -- +(-1.4, 0.0);
	  \draw (1.4, 2.20) node {\scriptsize forget outer blobs};
	  \draw (1.4, 1.47) node {\scriptsize forget inner blobs};
	  \draw (1.4, 0.81) node {\scriptsize contract inner blobs};
	\end{scope}

	\begin{scope}[shift={(12.18, -4.0)}] 
	  \draw (0.0, 0.175) -- (0.0, 0.595);
	  \fill (0.0, 0.595) circle[radius=0.065];
	  \draw (0.0, 0.595) -- (-0.35, 1.05);
	  \fill (-0.35, 1.05) circle[radius=0.065];
	  \draw (-0.35, 1.05) -- (-1.2, 1.27);
	  \draw (-0.35, 1.05) -- (-0.7, 1.75);
	  \fill (-0.7, 1.75) circle[radius=0.065];
	  \draw (-0.35, 1.05) -- (-0.175, 1.575);
	  \fill (-0.175, 1.575) circle[radius=0.065];
	  \draw (-0.175, 1.575) -- (-0.35, 2.3);
	  \draw (-0.175, 1.575) -- (0.0, 2.275);
	  \draw (0.0, 0.595) -- (0.56, 2.275);
	  \draw (0.0, 0.595) -- (0.42, 0.875);
	  \fill (0.42, 0.875) circle[radius=0.065];

	  \draw (-0.75,0.55) node {\tiny general};
	  \draw (-0.75,0.3) node {\tiny trees,};
	  \draw (-0.75,0.05) node {\tiny nontrivial};
	  \draw (-0.75,-0.2) node {\tiny blobs};

	  \fill (0.28,1.435) circle[radius=0.065];
	  \fill (0.42,1.855) circle[radius=0.065];
	  \draw[line width=0.01, rotate around={70:(-0.28,1.33)}] (-0.28,1.33)  ellipse (0.54 and 0.18);
	  \draw[line width=0.01, rotate around={34:(0.21,0.77)}] (0.21,0.77) ellipse (0.42 and 0.24);
	  \draw[line width=0.01, rotate around={70:(-0.7,1.75)}] (-0.7,1.75) ellipse (0.16 and 0.16);
	  \draw[line width=0.01, rotate around={70:(0.28,1.435)}] (0.28,1.435)  ellipse (0.16 and 0.16);
	  \draw[line width=0.01, rotate around={70:(0.42,1.855)}] (0.42,1.855)  ellipse (0.16 and 0.16);
	  \draw (-1.4, 1.33) node {$\bigleftbrace{30}$};
	  \draw (1.155, 1.33) node {$\bigrightbrace{30}$};
	  
	  \draw[line width=0.01, rotate around={77:(-0.5,1.44)}] 
	  (-0.5,1.44) ellipse (0.75 and 0.53);
	  
	  \draw[line width=0.01, rotate around={70:(0.35,1.645)}] 
	  (0.35,1.645) ellipse (0.5 and 0.25);
	  
	  \draw[line width=0.01, rotate around={34:(0.21,0.77)}] 
	  (0.21,0.77)  ellipse (0.5 and 0.3);
	\end{scope}		  
		  
  \begin{scope}[shift={(0.6, -1.2)}]
	\draw[->, wavy] (0.0, 0.0) -- +(0.0, 0.945);
	\draw (0.49, 0.49) node {\scriptsize nodes};
  \end{scope}
  \begin{scope}[shift={(6.0, -1.2)}]
	\draw[->, wavy] (0.0, 0.0) -- +(0.0, 0.945);
	\draw (0.45, 0.49) node {\scriptsize blobs};
  \end{scope}
  \begin{scope}[shift={(12.0, -1.2)}]
	\draw[->, wavy] (0.0, 0.0) -- +(0.0, 0.945);
	\draw (0.87, 0.49) node {\scriptsize outer blobs};
  \end{scope}

	\end{tikzpicture}
\end{equation}

\begin{prop}\label{prop:=Actop}
  For any operad $\PPP$, we have
  $$
  C := F\lowershriek \TSB{\PPP\bd}{\PPP\bdr} \simeq \fatnerve\, 
  \OMEGA_{\operatorname{act.inj}}(\PPP)\op  .
  $$
\end{prop}
\noindent Recall that $\fatnerve$ denotes the fat nerve.

\begin{proof}
  The objects in $C_0$ are arbitrary $\PPP$-trees (including the
  trivial trees); these are also the objects of
  $\OMEGA_{\operatorname{act.inj}}(\PPP)$. The objects in $C_1$ are
  active injections $K \actto T$. The face maps go as follows: $C_0
  \stackrel{d_0}\leftarrow C_1$ returns $K$ (`contract the blobs'),
  while $C_0 \stackrel{d_1}\leftarrow C_1$ returns $T$ (`underlying
  tree'). So this is precisely opposite to how the face maps work in
  $\fatnerve\OMEGA_{\operatorname{act.inj}}(\PPP)$. In higher degrees
  the checks are the same.
\end{proof}

The equivalence in Proposition~\ref{prop:=Actop} is just the 
restriction of the following (and the proof the same):

\begin{prop}
  For any operad $\PPP$, we have
  $$
  F\lowershriek \TSB{\PPP\bd}{\PPP\bd} \simeq \fatnerve 
  \OMEGA_{\operatorname{active}}(\PPP)\op .
  $$
\end{prop}

This result is interesting also in view of its analogy with the following:

\begin{prop}
  For any operad $\PPP$, we have
  $$
  F\lowershriek \TSB{\PPP\upperstar}{\PPP\upperstar} \simeq \fatnerve 
  \OMEGA_{\operatorname{inert, root-pres}}(\PPP)\op .
  $$
\end{prop}
 
Here $\OMEGA_{\operatorname{inert, root-pres}}(\PPP)$ is the category 
  of $\PPP$-trees and only root-preserving inert maps.
The two propositions add a new layer to the analogy (observed in 
\cite{Kock:0807})
between $\OMEGA_{\operatorname{inert}}$ with the (root-preserving, 
leaves-preserving) factorisation system and $\OMEGA$ with the (active, 
inert) factorisation system.

\bigskip

We are now ready for the locally finite version of the Main 
Theorem~\ref{thm:main}:

\begin{theorem}\label{thm:main-finite}
  For any operad $\PPP$, the two-sided bar constructions
  $Y=\TSB{\SSS}{\PPP\upperstar}$ and $Z=\TSB{\SSS}{\PPP\bdr}$ together
  endow the slice $\Grpd_{/\SSS(\tr(\PPP))}$ with the structure of a {\em
  locally finite} comodule bialgebra. Precisely, the incidence bialgebra of
  $Y$ is a locally finite left comodule bialgebra over the incidence bialgebra of $Z$.
\end{theorem}

\begin{proof}
  The comodule structure on $\Grpd_{/\SSS(\tr(\PPP))}$ is given by
  the comodule configuration
  $$
  u: \SSS F\lowershriek\TSB{\PPP\bd}{\PPP\bdr} \to \TSB{\SSS}{\PPP\bdr}
  $$
  of Lemma~\ref{lem:comod-config-locfin}, where also its finiteness as
  comodule is established. (Finiteness of the involved bialgebras was
  established in Lemmas~\ref{lem:freefinite} and \ref{lem:BDRfin}.)
  Establishing that the bialgebra structure maps of $Y$ are
  $Z$-comodule maps is a repetition of all the arguments in the proof
  of Theorem~\ref{thm:main}. The difference is that the comodule
  configuration $M$ is no longer just $Z$ itself, so the face
  maps of $Z=\TSB{\SSS}{\PPP\bdr}$ must be replaced by the ones from
  $M=\SSS F\lowershriek\TSB{\PPP\bd}{\PPP\bdr}$. Writing $A := Y_1 =
  M_0$, the main diagram now takes the shape
    \begin{equation}
  \begin{tikzcd}[sep=large]
	 A & \ar[l, "d_1^Y"'] Y_2 \ar[r, "{(d_2^Y, d_0^Y)}"] & A \times A
	 \\
	 M_1 \ar[u, "d_1^M"] \ar[dd, "{(u,d_0^M)}"']  & 
	 Q \dlpullback \urpullback 
	 \ar[l, dotted] \ar[u, dotted] \ar[r, dotted] \ar[dd, dotted] &
	 M_1 \times M_1 
	 \ar[d, "{(u,d_0^M)\times(u,d_0^M)}"] \ar[u, "d_1^M\times d_1^M"'] 
	 \\
	  {}&& Z_1 \times A \times Z_1 \times A \ar[d, "\mu_{13}"] 
	 \\
	 Z_1 \times A& Z_1 \times Y_2 \ar[l, "\id \times d_1^Y"] \ar[r, "{\id 
	 \times (d_2^Y,d_0^Y)}"']& Z_1 \times A \times A  ,
  \end{tikzcd}
  \end{equation}
  and the middle object has to be $Q:= M_1 \times_A Y_2$.  
  The groupoids involved are full subgroupoids of 
  those in \ref{thm:main}, so commutativity of the diagrams and the fibre 
  calculations to establish the pullback conditions are the same again.
\end{proof}

\section{Examples}

\subsection{Baez--Dolan construction on categories}
\label{sub:unary}

Before coming to more explicit examples, we deal with the case where $\PPP$
is a small category, or more precisely a Segal groupoid (such as for
example the fat nerve of an ordinary category).
We view $\PPP$ as an
operad with only unary operations. As a polynomial monad 
this means that it is cartesian over the identity monad
$$
\PPP \Rightarrow \Id \Rightarrow \SSS ,
$$
which leads to some special features.
The first of these is straightforward:
\begin{lemma}
  For any small category $\RRR$, considered as an operad with only unary
  operations, $\RRR\Rightarrow\Id$, we have
  $$
  \TSB{\SSS}{\RRR} \simeq \SSS ( \TSB{\Id}{\RRR} ) \simeq \SSS( 
  \fatnerve\RRR ).
  $$
\end{lemma}
\noindent
(Taking fat nerve here is only to stress that it is regarded as a 
simplicial groupoid, not as a polynomial monad.)

If $\PPP$ is a small category, then so is $\PPP\upperstar$.
We thus have
$$
Y := \TSB{\SSS}{\PPP\upperstar} = \SSS ( \TSB{\Id}{\PPP\upperstar} ) .
$$
It follows that the comultiplication of $Y$ is actually $\SSS$ of
a comultiplication, namely the standard
incidence comultiplication of the Segal groupoid $\PPP\upperstar$. 
Recall that in Theorem~\ref{thm:main-finite},
the comodule structure too comes from a simplicial 
groupoid which is $\SSS$ of something: $M = \SSS C$.
In fact, for $\PPP$ a category, the whole comodule {\em bi}algebra structure
is just $\SSS$ of a comodule {\em co}algebra structure.
The following result formalises this, providing a cheaper unary 
version of the main theorems. The proof is very similar.

\begin{theorem}
  For any small category $\PPP$,
  the incidence coalgebra of
  $\TSB{\Id}{\PPP\upperstar} \simeq \fatnerve \PPP\upperstar$
  is a left comodule coalgebra over the incidence bialgebra of
  $\TSB{\PPP\bd}{\PPP\bdr}$.
\end{theorem}

\begin{blanko}{Remark.}
  The comodule bialgebras of Carlier~\cite{Carlier:1903.07964}
  defined from hereditary species are all of the form `$\SSS$ of a 
  comodule coalgebra', although they do not come from neither 
  categories nor operads. 
  
  Adding an algebra structure freely on top of a coalgebra structure is
  important to get access to the machinery of antipodes \cite{Schmitt:hacs}
  rather than just M\"obius inversion (see~\cite{Carlier-Kock:1807.11858}
  for this perspective at the objective level). One important special case
  is that of the {\em cartesian envelope} of a poset, which has been
  exploited to great effect by Aguiar and Ferrer~\cite{Aguiar-Ferrer}.
\end{blanko}

\subsection{Fa\`a di Bruno}

\label{sub:FdB}

\begin{blanko}{Baez--Dolan construction on the identity monad.}
  Let $\PPP$ be the terminal category. Considered as a polynomial
  monad it is the identity monad $\PPP:=\Id$, represented by the
  polynomial $1\leftarrow 1 \to 1 \to 1$ (as in Example~\ref{ex:Id}).
  Now $\PPP$-trees are linear trees, and the category of $\Id$-trees
  is $\OMEGA(\Id) \simeq \simplexcategory$.
  
  The free monad $\PPP\upperstar = \Id\upperstar$ is the one-object
  category whose arrows are the linear trees (with leaf as domain and 
  root as codomain), 
  composed by grafting. In
  other words, the category is just the monoid $(\N, +,0)$.
  
  The Baez--Dolan construction is the operad $\PPP\bd= \Id\bd$ whose
  only colour is \inlineonetree, and whose operations are the linear
  trees. The input slots are the nodes, and substitution works like 
  this:

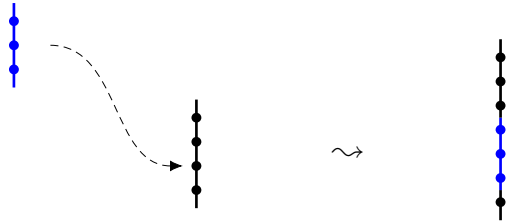
\begin{center}
  \begin{tikzpicture}[line width=1.0pt,scale=0.80]
	
	\begin{scope}[shift={(0.0, 2.5)}, blue]
	   \draw[blue] (0,0.1) -- (0,1.5);
	   \draw[blue] (0, 0.4) pic {onedot};
	   \draw[blue] (0, 0.8) pic {onedot};
	   \draw[blue] (0, 1.2) pic {onedot};
	  \coordinate (start) at (0.6, 0.8);
	\end{scope}

	\begin{scope}[shift={(3.0, 0.5)}]	
	  \coordinate (insertionpoint) at (-0.25, 0.8);
	  \draw (0,0.1) -- (0,1.9);
	  \draw (0, 0.4) pic {onedot};
	  \draw (0, 0.8) pic {onedot};
	  \draw (0, 1.2) pic {onedot};
	  \draw (0, 1.6) pic {onedot};
	\end{scope}

	\draw[-{Latex[width=3.5pt,length=4pt]}, line width=0.02, densely dashed] 
	  (start) to[out=0, in=180] (insertionpoint);

	\node at (5.5, 1.5) {$\leadsto$};

	\begin{scope}[shift={(8.0, 0.3)}]	
	  \draw (0,0.1) -- (0,0.6);
	  \draw (0, 0.4) pic {onedot};
	  \begin{scope}[shift={(0.0,0.4)}]
		\draw[blue] (0,0.2) -- (0,1.4);
		\draw[blue] (0, 0.4) pic {onedot};
		\draw[blue] (0, 0.8) pic {onedot};
		\draw[blue] (0, 1.2) pic {onedot};
	  \end{scope}
	  \begin{scope}[shift={(0.0,1.6)}]
		\draw (0,0.2) -- (0,1.5);
		\draw (0, 0.4) pic {onedot};
		\draw (0, 0.8) pic {onedot};
		\draw (0, 1.2) pic {onedot};
	  \end{scope}
	\end{scope}

  \end{tikzpicture}
\end{center}
  
  This operad is the free-monoid operad: $\Id\bd\simeq\MMM$. From the
  $\Id\bd$ viewpoint, the operations are linear trees, and the
  operations of $\Id\bd\circ\Id\bd$ are active maps of linear trees.
  From the viewpoint of the equivalent operad $\MMM$, the operations
  are planar corollas, and the operations of $\MMM\circ\MMM$ are
  $2$-level planar trees.
  For the reduced Baez--Dolan construction $\Id\bdr\simeq 
  \overline\MMM$ we just exclude the trivial
  tree (or in the corolla interpretation, the nullary corolla).
  
  The two-sided bar construction giving the comodule configuration is
  described by Proposition~\ref{prop:=Actop}:
  $$
  \TSB{\Id\bd}{\Id\bdr} 
  \simeq
  \fatnerve \simplexcategory_{\operatorname{act.inj}}\op  ,
  $$
  and combined with the well-known equivalence 
  (see~\cite{Galvez-Kock-Tonks:1512.07573})
  $$
  \simplexcategory_{\operatorname{act.inj}}\op
  \simeq \simplexcategory^+_{\operatorname{surj}}
  $$
  we get the interpretation from the $\MMM$-viewpoint:
  $\TSB{\MMM}{\overline\MMM}$ is the fat nerve of the category
  $\simplexcategory^+_{\operatorname{surj}}$ of finite ordinals 
  (including the empty ordinal) and monotone surjections.
  
  (For the bar construction $Z=\TSB{\SSS}{\Id\bdr}$ giving the
  bialgebra, there is also a nerve interpretation: it is the category
  whose objects are finite sets and whose maps are surjections
  equipped with a linear order on each fibre.)
\end{blanko}

\begin{blanko}{Incidence comodule bialgebra.}
  Denote by $a_n$ the linear tree with $n$ nodes.  The 
  comultiplication coming from $\TSB{\SSS}{\Id\upperstar} = 
  \TSB{\SSS}{\N}$ is given (on connected elements) by
  $$
  \Delta(a_n) = \sum_{i+j = n} a_i \tensor a_j \qquad n\geq 0 .
  $$
  Here the sum is over nonnegative numbers summing to $n$.
  
  The comultiplication coming from $ \TSB{\SSS}{\Id\bdr}$
  is given (on connected elements) by
  $$
  \Delta(a_n) = \sum_{n_1+\dots+n_k =n} a_{n_1} \cdots a_{n_k} 
  \tensor a_k   \qquad n\geq 1.
  $$
  Here the sum is over all compositions of the natural number $n$.
  This composition comes about as an active map $a_k \actto a_n$.
  Inclusion of the $k$ one-node trees into $a_k$ defines the numbers 
  $n_i$ by
  active-inert factorisation:
  \[
  \begin{tikzcd}[sep = {4em,between origins}]
  {[1]} \ar[d, into] \ar[r, dotted, -act] & {[n_i]} \ar[d, dotted, into] \dlactinert \\
  {[k]} \ar[r, ->|] & {[n]} .
  \end{tikzcd}
  \]
\end{blanko}

\begin{prop}\label{prop:FdB}
  The incidence comodule bialgebra of the Baez--Dolan construction on 
  $\Id$ is the Fa\`a di Bruno comodule bialgebra.
\end{prop}

\begin{proof}
  This is clear once we describe the Fa\`a di Bruno comodule
  bialgebra. Consider the ring of formal power series $\Q[[z]]$
  under multiplication and under substitution (assuming zero constant 
  term).
  
  For $k\geq 0$, consider the linear functional
  \begin{eqnarray*}
   a_k : \Q[[z]] & \longrightarrow & \Q  \\
    \textstyle{\sum_n} f_n z^n & \longmapsto & f_k  .
  \end{eqnarray*}
  The {\em Fa\`a di Bruno bialgebra} is the polynomial ring $\FF =
  \Q[a_1,a_2,\ldots]$, with comultiplication $\Delta: 
  \FF\to\FF\tensor \FF$ given as dual to
  substitution of power series (we 
  disallow constant terms):
  $$
  \Delta(a_n)(f\tensor g) := a_n(g\circ f).
  $$
  It follows that
  $$
  \Delta(a_n) = \sum_{n_1+\dots+n_k =n} a_{n_1} \cdots a_{n_k} 
  \tensor a_k ,   \qquad n\geq 1  ,
  $$
  just as for the $Z$-comultiplication of linear trees.
  
  On the other hand, $\MM =
  \mathbb{C}[a_0,a_1,a_2,\ldots]$ is a bialgebra too, with
  comultiplication $\Delta:\MM\to\MM\tensor\MM$ dual to multiplication of power series:
  $$
  \Delta(a_n)(f\tensor g) := a_n(f \cdot g).
  $$
  This expands to
  $$
  \Delta(a_n) = \sum_{i+j = n} a_i \tensor a_j , \qquad n\geq 0  ,
  $$
  precisely as for the $Y$-comultiplication of linear trees.
  
  Together,
  $\FF$ and $\MM$ form the {\em Fa\`a di Bruno comodule bialgebra}: $\MM$ is
  a comodule bialgebra over $\FF$. The comodule-bialgebra axioms 
  follow
  from the fact that substitution distributes over multiplication on 
  the left.
\end{proof}

\begin{blanko}{Remarks.}
  The present form of the Fa\`a di Bruno bialgebra is how it appears in
  algebraic topology, where it is called the (dual) Landweber--Novikov
  bialgebra (see for example~\cite{Morava:Oaxtepec}).
  The usual form of the Fa\`a di Bruno bialgebra (see for 
  example~\cite{Figueroa-GraciaBondia:0408145})
  uses rather the basis
  \begin{eqnarray*}
    A_k : \Q[[z]] & \longrightarrow & \Q  \\
    \textstyle{\sum_n} f_n \frac{z^n}{n!} & \longmapsto & f_k  ,
  \end{eqnarray*}
  in which case the natural simplicial realisation is as the fat nerve
  of the category of finite sets and surjections
  (cf.~\cite{Joyal:1981}, \cite{Galvez-Kock-Tonks:1207.6404}). This
  comes about as the two-sided bar construction
  $\TSB{\SSS}{\overline\SSS}$ (see \cite{Kock-Weber:1609.03276}), but
  cannot arise as a Baez--Dolan construction.
  
  The factorial-free form of the Fa\`a di Bruno bialgebra given here
  admits a noncommutative
  variant~\cite{Brouder-Frabetti-Krattenthaler:0406117} which is also
  called the Dynkin--Fa\`a di Bruno bialgebra in the theory of
  numerical integration on manifolds~\cite{MuntheKaas:BIT95}; see also
  \cite{EbrahimiFard-Lundervold-Manchon:1402.4761}. Objectively, the
  noncommutative bialgebra is the incidence bialgebra of
  $\TSB{\Id\bd}{\Id\bdr}$, but this is not so interesting in the
  present context, since noncommutative bialgebras do not in general
  admit comodule bialgebras.
\end{blanko}

\subsection{Mould calculus}

\begin{blanko}{Mould calculus.}
  The mould calculus was introduced by \'Ecalle~\cite{Ecalle:I} as a
  combinatorial toolbox for his theory of resurgence in the theory of
  local dynamical systems (see Cresson~\cite{Cresson}).
  For $\Omega$ a monoid, let
  $(\Omega\upperstar, \,\cdot\,)$ denote the free monoid on $\Omega$.
  A {\em mould} is a function $M^\bullet: \Omega\upperstar \to k$ (for
  $k$ a ring), taking a word $w\in \Omega\upperstar$ to $M^w$. There
  are two basic operations: the {\em product} is defined as
  $$
  (M\times N)^w = \sum_{\omega=w'\cdot w''} M^{w'} N^{w''} ,   \qquad 
  \omega\in\Omega\upperstar
  $$
  (where $\cdot$ is concatenation of words),
  and the {\em composition} is defined as
  $$
  (M \circ N)^w = 
  \sum_{k>0} \sum_{w=w_1\cdots w_k} 
  N^{w_1}\cdots N^{w_k} \ M^{||w_1|| \cdots ||w_k||},  \qquad 
  \omega\in\Omega\upperstar.
  $$
  Here $||w|| \in \Omega$ denotes the multiplication of the word $w\in
  \Omega\upperstar$ in the monoid $\Omega$. Composition distributes
  over the product, but only from the left~\cite{Ecalle:I}.
\end{blanko}

\begin{blanko}{Moulds via Baez--Dolan construction.}
  Consider the monoid $\Omega$ as an operad with only one colour and
  only unary operations, and let $\Omega\bdr$ be the Baez--Dolan
  construction. Its colours are the elements of
  $\Omega$. There is a $k$-ary operation of profile $(a_1,\ldots,a_k ;
  b)$ if and only if $a_1\cdots a_k = b$, for $k>0$. There is also the operad
  $\Omega\upperstar$, the free operad on $\Omega$. The incidence
  algebra of the operad $\Omega\upperstar$ is the algebra of moulds
  under $\times$, and the incidence algebra (that is, the convolution 
  algebra of the incidence coalgebra) of $\Omega\bdr$ is the
  algebra of moulds under $\circ$. More precisely, the incidence
  bialgebra of the whole structure is a comodule bialgebra, dual to the 
  near-semiring of $\Omega$-moulds.
\end{blanko}

\subsection{B-series, and the Calaque--Ebrahimi-Fard--Manchon comodule bialgebra}

\label{sec:CEFM}

Throughout we have considered operadic trees --- trees with open-ended
edges. In this subsection we are concerned with trees {\em without}
open-ended edges: they are defined as connected and simply-connected graphs
with a distinguished vertex called the root. In the (huge) literature
employing them, they are simply called {\em rooted trees}. Below we
adhere to this convention as long as no operadic trees are involved,
but call them {\em combinatorial trees} when contrast with operadic
trees is required.

\begin{blanko}{B-series.}
  {\em B-series} were introduced by Butcher~\cite{Butcher:1972} in his
  study of order conditions for Runge--Kutta methods, and named after him
  by E.~Hairer (see~\cite{Hairer-Lubich-Wanner}). They are formal series
  indexed by (combinatorial) rooted trees, of the form

  \begin{align*}
B(a,hf,y) &=
\sum_{\tau\in \mathcal{T}}   \frac{h^{\left|\tau\right|}}{\tau!}   a(\tau) f^\tau(y) 
\\
&=
a(\emptyset)y+ h\,a(\!\raisebox{1pt}{\ctreeone}\!) f(y) + h^2 \,
a(\!\raisebox{-2pt}{\ctreetwo}\!) (f'f)(y) 
\\ & \phantom{xxxxxxxxxxxxxxxxx} 
+ 
h^3 \left( a(\!\raisebox{-5pt}{\ctreethreeL}\!) (f'f'f)(y)
+ \textstyle{\frac12} a(\!\raisebox{-2pt}{\ctreethreeV}\!) (f''(f,f))(y) \right) 
+  \dots
\end{align*}
  where $\mathcal{T}$ is the set of rooted trees, $\left|\tau\right|$
  denotes the number of nodes in $\tau$, and $f^\tau$ is the elementary
  differential associated to $\tau$ (for $f$ a vector field), and $h$ 
  is a
  step-size parameter. $a(\tau)$ are the coefficients, encoded as a
  complex-valued function on $\mathcal{T}$. For an initial-value problem
  $\dot y = f(y)$, $y(0)= y_0$, the exact solution can be expanded as a
  B-series, but more importantly, many numerical methods, including all 
  Runge--Kutta methods,\footnote{An intrinsic characterisation of B-series 
  methods was given only recently in terms of affine 
  equivariance~\cite{McLachlan-Modin-MuntheKaas-Verdier}.}
  can be regarded as a
  B-series, the coefficients $a(\tau)$ being the weights assigned to the
  elementary differentials $f^\tau$ of $f$.
\end{blanko}

\begin{blanko}{Composition and substitution of B-series.}
  There are two fundamental operations one can perform on B-series:
  composition and substitution. {\em Composition} (due to
  Butcher~\cite{Butcher:1972}): for $b(\emptyset)=1$,
  $$
  B(a, hf, B(b,hf, y)) = B(b\cdot a, hf, y)  .
  $$
%
  This characterises the product $\cdot$ defining the {\em Butcher
  group} of B-series, which was later rediscovered as the group of
  characters of the Connes--Kreimer Hopf algebra of rooted trees in
  perturbative renormalisation~\cite{Dur:1986},
  \cite{Connes-Kreimer:9808042} (see Example~\ref{ex:CK}).

  {\em Substitution} (introduced by Chartier, E.~Hairer, and
  Vilmart~\cite{Chartier-Hairer-Vilmart:INRIA},~\cite{Chartier-Hairer-Vilmart:FCM2010}):
%
  if $b(\emptyset)=0$ then
  $B(b,hf, -)$ is a vector field, so it makes sense to substitute it into
  another $B$-series in the $hf$ slot:
  $$
  B(a, B(b,hf, - ),y) = B( b \star a, hf,y) .
  $$
  This characterises a new product $\star$, which can be 
  described combinatorially in terms of contracting 
  subtrees.  Chartier--Hairer--Vilmart showed that $\star$ acts on $\cdot$
  by group homomorphisms.  (The substitution product is
  important in backward error analysis and in the more
  general theory of modified (preprocessed)
  integrators~\cite{Chartier-Hairer-Vilmart:MathComp}.)
\end{blanko}

\begin{blanko}{The Calaque--Ebrahimi-Fard--Manchon comodule bialgebra.}
  Calaque, Ebrahimi-Fard, and Manchon
  \cite{Calaque-EbrahimiFard-Manchon:0806.2238} gave a Hopf-algebra
  theoretic interpretation of composition and substitution inspired by
  quantum field theory, relating the substitution product with a tree
  version of the Connes--Kreimer Hopf algebra of {\em Feynman
  graphs}~\cite{Kreimer:9707029}, \cite{Connes-Kreimer:9912092}. The
  bialgebras are both the free commutative algebra on the set of rooted
  trees. The comultiplications are such that their respective characters
  form the group structures on B-series.

  The comultiplication corresponding to composition of B-series is simply
  the Butcher--Connes--Kreimer Hopf algebra of rooted trees of
  Example~\ref{ex:CK}. The comultiplication corresponding to substitution
  of B-series is defined by summing over all ways of partitioning the set
  of nodes into subtrees. The left-hand tensor factor is then constituted
  by the forest of all these subtrees, whereas the right-hand tensor factor
  is obtained by contracting each subtree to a single node.

  Since the trees involved are only combinatorial trees (as opposed to
  operadic trees), it is not possible to realise this bialgebra as the
  incidence bialgebra of an operad --- there is not enough typing
  information available to make sense of substituting a tree into the node
  of another tree. But for {\em operadic trees} this works, and we shall
  see that the operadic analogue of the Calaque--Ebrahimi-Fard--Manchon
  comodule bialgebra is the incidence comodule bialgebra of a Baez--Dolan
  construction (of the terminal operad). The precise relationship is given
  by taking core of an operadic tree, as we proceed to explain.
\end{blanko}

\begin{blanko}{The core of a $\PPP$-tree \cite{Kock:1109.5785}.}\label{core}
  The {\em core} of a $\PPP$-tree is the combinatorial tree constituted by
  its inner edges, obtained by forgetting all decorations and shaving off
  leaf edges and root edge. Taking core constitutes a bialgebra
  homomorphism from the bialgebra of $\PPP$-trees (the incidence bialgebra
  of $\TSB{\SSS}{\PPP\upperstar}$) to the Butcher--Connes--Kreimer Hopf
  algebra. This bialgebra homomorphism is induced by a symmetric monoidal
  culf map from $\TSB{\SSS}{\PPP\upperstar}$ to the decomposition space
  of Example~\ref{ex:CK}.
  
  The core map compares operadic approaches with the standard
  combinatorial viewpoint in renormalisation. One advantage of the operadic
  viewpoint is that the operadic trees have a leaf grading, which cannot be
  seen in the core. This has been exploited in connection with BPHZ
  renormalisation~\cite{Kock:1411.3098} and in connection with
  combinatorial Dyson--Schwinger equations~\cite{Kock:1512.03027}.
  Existence of the leaf grading is closely related to the fact that
  operadic trees form a Segal object, whereas taking core destroys the
  Segal-ness. Combinatorial trees form instead only a decomposition
  space~\cite{Galvez-Kock-Tonks:1612.09225}, as mentioned in 
  Example~\ref{ex:CK}.
\end{blanko}

We now specialise to the case where $\PPP$ is the terminal operad (Comm),
so that $\PPP$-trees are just naked (operadic) trees.

\begin{blanko}{Core of $\TSB{\SSS}{\PPP\bdr}$ and $F\lowershriek\TSB{\PPP\bd}{\PPP\bdr}$.}
  Each of the simplicial groupoids $\TSB{\SSS}{\PPP\bdr}$ and
  $F\lowershriek\TSB{\PPP\bd}{\PPP\bdr}$ admits a core version (which
  however cannot possibly arise as the two-sided bar construction of an
  operad). We describe the first one by hand, by mimicking the explicit
  description of $\TSB{\SSS}{\PPP\bdr}$.
  
  Let $\dstrees\bdr$ denote the simplicial groupoid with $\dstrees\bdr_1$
  the groupoid of combinatorial forests, and $\dstrees\bdr_2$
  the groupoid of blobbed combinatorial forests, such that each blob 
  contains at least one node (and as usual, each node is contained in 
  precisely one blob). $\dstrees\bdr_0$ is the groupoid of forests 
  consisting only of one-node trees. The face and degeneracy maps have the 
  same descriptions as in Figure~\eqref{eq:bdr}.
  
  For the comodule $F\lowershriek\TSB{\PPP\bd}{\PPP\bdr}$,
  the core is in fact the upper 
  decalage of $H\bdr$. Just as $\TSB{\PPP\bd}{\PPP\bdr}$ is the fat nerve 
  of the opposite of the category of nontrivial $\PPP$-trees and active 
  injections (cf.~Proposition~\ref{prop:=Actop}), also the core is a fat nerve, namely
  the fat nerve of the category of combinatorial trees and edge contractions.
\end{blanko}

\begin{lemma}
  The simplicial groupoid $H\bdr$ just described is a symmetric monoidal
  decomposition space. Taking core defines a culf map
  $\TSB{\SSS}{\PPP\bdr} \to H\bdr$ (and hence also a culf map from
  $F\lowershriek\TSB{\PPP\bd}{\PPP\bdr}$).
\end{lemma}

\begin{proof}
  The proofs are standard and will not be given in detail. The symmetric
  monoidal structure is just disjoint union. To check the
  decomposition-space axiom, the first observation is that the inner face
  maps are discrete fibrations (the fibres being discrete sets of possible
  blobbings). The appropriate pullback squares are now verified by
  computing fibres of the inner face maps involved. Culfness of the
  taking-core map is clear: the possible blobbings of a $\PPP$-tree depends
  only on its core, not on leaves, root, or $\PPP$-structure.
\end{proof}

\begin{prop}
  The core of the incidence comodule bialgebra of the Baez--Dolan
  construction of the terminal operad is the
  Calaque--Ebrahimi-Fard--Manchon comodule bialgebra.
\end{prop}

\begin{proof}
  This is just a verification. For the Butcher--Connes--Kreimer
  comultiplication the result was already in \cite{Kock:1109.5785}.
  For the Chartier--Hairer--Vilmart--Calaque--Ebrahimi-Fard--Manchon comultiplication, the point is
  simply that the summation is over blobbings of a given tree. (See
  also \cite[Ex.~7.5]{Kock-Weber:1609.03276} for this result.)
\end{proof}

\subsection{Other examples}

\label{sub:freeP}

\begin{blanko}{Linear orders and comodule bialgebras of monotone words.}
  \label{ex:linearorder}
  Let $\PPP$ be a (countable) linear order, considered as a linearly ordered alphabet.
  For definiteness we shall take the linear order $\PPP=\N$. Since $\PPP$
  is a poset and hence a category, it is a coloured operad with only unary
  operations. The $\PPP$-trees are the non-empty monotone words.
  
  The operad $\PPP\upperstar$ is the category whose objects are the
  letters, and whose morphisms from $i$ to $j$ are the monotone words that
  start in $i$ and end in $j$ (allowing the one-letter word for $i=j$). Two
  words are composed by $1$-letter-overlap concatenation, as exemplified by
  \begin{center}
	\begin{tikzpicture}[line width=1.0pt,scale=0.60]
	  \begin{scope}[shift={(0.0, 2.0)}, blue]
		\node at (1.5,0.0) {$266\underline{7}$};
		\coordinate (start) at (1.55, -0.7);
	  \end{scope}

	  \begin{scope}[shift={(3.0, 0.0)}]
		\node at (1.5,0.1) {$\underline{7}8899$};
		\coordinate (insertionpoint) at (0.45, 0.15);
	  \end{scope}

	  \draw[-{Latex[width=3.5pt,length=4pt]}, line width=0.02, densely dashed] 
		(start) to[out=-75, in=180] (insertionpoint);
	  \node at (8.0, 1.0) {$\leadsto$};

	  \begin{scope}[shift={(10.5, 0.8)}]	
		\node at (1.3,0.2) {$\textcolor{blue}{2667}8899$};
	  \end{scope}
	\end{tikzpicture}
  \end{center}
  
  The operad $\PPP\bdr$ has colours monotone two-letter words. The
  operations are monotone words of length at least $2$. The output colour
  of a word is the pair consisting of the first and the last letter. The
  input slots of such a word are the gaps between letters, and the colour
  of an input slot is the pair of adjacent letters. A whole word can thus
  be substituted into a gap provided its first and last letters agree with
  the adjacent letters of the gap, and in the result of the substitution,
  these two letters are replaced by the whole word, as pictured here:
  
  \begin{center}
	\begin{tikzpicture}[line width=1.0pt,scale=0.60]
	  \begin{scope}[shift={(0.0, 3.0)}, blue]
		\node at (1.5,0.0) {$\underline{2}245\underline{6}$};
		\coordinate (start) at (1.85, -0.5);
	  \end{scope}

	  \begin{scope}[shift={(3.0, 0.0)}]
		\node at (1.5,0.1) {$11\overline{2}\,\overline{6}6899$};
		\coordinate (insertionpoint) at (1.16, 0.45);
	  \end{scope}

	  \draw[-{Latex[width=3.5pt,length=4pt]}, line width=0.02, densely dashed] 
		(start) to[out=-60, in=90] (insertionpoint);
	  \node at (8.2, 1.5) {$\leadsto$};

	  \begin{scope}[shift={(10.5, 1.36)}]	
		\node at (1.3,0.2) {$11\raisebox{1.8pt}{\textcolor{blue}{22456}}6899$};
	  \end{scope}
	\end{tikzpicture}
  \end{center}
	
  (Had we taken $\PPP\bd$ instead of $\PPP\bdr$, we would have also the 
  one-letter words, and an allowed substitution would be, for instance, to replace a 
  subword $44$ by $4$.)
  
  The comodule bialgebra is the polynomial algebra on these monotone words.
  For the comultiplication corresponding to $\PPP\upperstar$, the $1$-letter
  words are group-like, and the degree of a word is its length minus $1$.

  Example:
  $$
  \Delta(2335) = 2 \tensor 2335 + 23 \tensor 335 + 233 \tensor 35 + 2335 
  \tensor 5 .
  $$
  The comultiplication corresponding to $\PPP\bdr$
  is given by summing over subwords that include the first and last letter,
  and then putting this subword in the right tensor factor, and putting
  on the right the monomial consisting of the words read within the 
  original word from one letter in the subword to the next. 

Example:
\begin{align*}
\Delta(35688) =& \
35 \cdot 56 \cdot 68 \cdot 88 \tensor 35688
\\
&+ 35 \cdot 56 \cdot 688 \tensor 3568
\ + \ 35 \cdot 568 \cdot 88 \tensor 3588
\ + \ 356 \cdot 68 \cdot 88 \tensor 3688
\\
&+ 35 \cdot 5688  \tensor 358
\ + \ 356 \cdot 688 \tensor 368
\ + \ 3568 \cdot 88 \tensor 388
\\
&+ 35688 \tensor 38  .
\end{align*}

\end{blanko}

\begin{blanko}{Variation: contractible groupoids.}
  As a variation of the previous example, instead of a linear order 
  take $\PPP$ to be a contractible groupoid (codiscrete groupoid on an 
  alphabet). Now $\PPP$-trees are non-empty words in the 
  alphabet (without any monotonicity constraints). Concatenation and 
  substitution work exactly as before.
\end{blanko}
  
\begin{blanko}{Quivers and comodule bialgebras of paths.}
  Linear orders are free categories on linear quivers. 
  Example~\ref{ex:linearorder}
  immediately generalises to general quivers (directed graphs). Let $Q$ be 
  a quiver, and let $\PPP$ be the free category on $Q$. Then a $\PPP$-tree
  is a
  {\em marked path} in the quiver, meaning a path with `stations' 
  marked along the path, including the start and the finish. Formally a 
  marked path is a configuration
  $$
  \Delta^k \actto \Delta^t \to Q .
  $$
  (Warning: the two maps live in different categories and cannot be
  composed: since $Q$ is only a quiver, not a category, the two maps do not
  make $\Delta^k$ into a path in $Q$!)
  
  Now $\PPP\upperstar$ is the category whose objects are the vertices of $Q$
  and whose morphisms are the marked paths. Composition is just 
  concatenation of marked paths at their endpoints.
  
  
  In the operad $\PPP\bdr$, the colours are the nontrivial paths. The operations are
  the nontrivial marked paths without trivial stages. The output colour 
  of a marked path is the path itself;
  the input colours are the stages of the marked path, meaning the paths
  from one station to the next. 
  
  Substitution replaces a stage with a 
  further marking.
  The formalisation of this (unpacking the general 
  constructions) involves active-inert pushouts in $\simplexcategory$ 
  (\ref{active-inert-Delta}).
  In detail, a marked path with a chosen input slot is the 
  configuration
  $$
  \begin{tikzcd}
  \Delta^1 \ar[d, into] \ar[r, dotted, -act]& 
  \Delta^s \dlactinert \ar[d, dotted, into]
  \ar[rd, 
  dashed, bend left]&  \\
  \Delta^k \ar[r, -act] & \Delta^t \ar[r] & Q
  \end{tikzcd}
  $$
  with input colour $\Delta^s \to Q$ obtained by active-inert 
  factorisation in $\simplexcategory$, as indicated. Note that the curved map
  does make sense as a composite, because 
  inert maps in $\simplexcategory$ are quiver maps, so $\Delta^s \to Q$ is
  a path in $Q$.
  To give another operation with 
  matching output is to give $\Delta^h \actto \Delta^s \to Q$, and since 
  $\Delta^1$ is initial in the category of active maps, this must factor
  the map $\Delta^1 \actto \Delta^s$:
  $$
  \begin{tikzcd}
  \Delta^1 \ar[d, into] \ar[r, dotted, -act] \ar[rr, bend left, -act]& 
  \Delta^h \ar[r, -act] & \Delta^s 
  \ar[d, into] &  \\
  \Delta^k \ar[rr, -act] && \Delta^t \ar[r] & Q .
  \end{tikzcd}
  $$
  The result of the substitution is obtained by taking the active-inert 
  pushout as indicated:
  $$
  \begin{tikzcd}
  \Delta^1 \ar[d, into] \ar[r, -act]& \Delta^h \ar[r, -act] \ar[d, 
  dotted, into] & \Delta^s 
  \ar[d, into] &  \\
  \Delta^k \ar[rr, bend right, -act] \ar[r, dotted, -act] & \Delta^r 
  \ulpullback
  \ar[r, dotted, -act]& \Delta^t \ar[r] & Q ,
  \end{tikzcd}
  $$
  which finally induces $\Delta^r \actto \Delta^t$ 
  by the universal property of the pushout.
  The result of the substitution is thus
  $$
  \Delta^r \actto \Delta^t \to Q .
  $$
\end{blanko}

\begin{blanko}{Polynomial endofunctors and subdivided trees.}
  The preceding examples have multi versions. The multi analogue of a
  linear order is a tree, and the multi analogue of a quiver is a
  polynomial endofunctor. Recall that a tree is a special case of a
  polynomial endofunctor (\ref{polytree-def}). For brevity we treat the
  latter case.
  
  Let $Q$ be a polynomial endofunctor, and consider $\PPP = Q\upperstar$,
  the free monad on $Q$. Then $\PPP$-trees are subdivided $Q$-trees,
  or more formally: $Q$-trees $T \to Q$ equipped with an active map $K 
  \actto T$. We write
  $$
  K \actto T \to Q 
  $$
  (with the warning again that the two maps live in 
  different categories and cannot be composed, and $K$ is not a $Q$-tree).
  
  The colours of $\PPP\upperstar$ are the original colours of $Q$.
  The operad structure on $\PPP\upperstar$ is given by grafting,
  which is gluing of subdivided trees. 
  
  The Baez--Dolan construction $\PPP\bdr$ has colours the
  nontrivial $Q$-trees. For the operations, we now have nontrivial
  subdivided trees, meaning
  $$
  K \actto T \to Q
  $$
  where $K \actto T$ is an active injection.
  The output colour is the whole tree $T \to Q$.  The input slots
  are the nodes of $K$; the colour of an input slot is the $Q$-tree $S \to 
  Q$
  obtained as active-inert factorisation (\ref{active})
  $$
  \begin{tikzcd}
  C \ar[d, into] \ar[r, dotted, -act]& S \dlactinert \ar[d, dotted, into] 
  \ar[rd, dashed, bend left]&  \\
  K \ar[r, -act] & T \ar[r] & Q .
  \end{tikzcd}
  $$
  Note that the curved map does make sense as a composite, because 
  inert maps are maps of polynomial endofunctors (\ref{inert}), so $S$ becomes indeed a 
  $Q$-tree.
  
  Substitution amounts to further refinement of subtrees of $T$. Formally
  it involves again active-inert pushouts, now in $\OMEGA$. In the diagram
  $$
  
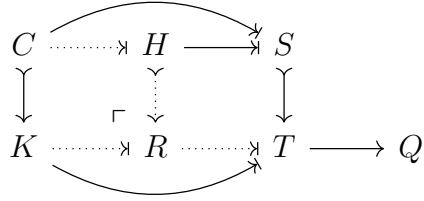
\begin{tikzcd}
  C \ar[d, into] \ar[r, dotted, -act] \ar[rr, bend left, -act]& 
  H \ar[r, -act] \ar[d, dotted, into] & S 
  \ar[d, into] &  \\
  K \ar[rr, bend right, -act] \ar[r, dotted, -act] & R
  \ulpullback
  \ar[r, dotted, -act]& T \ar[r] & Q
  \end{tikzcd}
  $$  
  the solid part represents the data required for a substitution:
  an operation $H{\actto}S{\to} Q$ with output colour $S{\to}Q$, to be 
  substituted into the input slot of $K{\actto}T{\to}Q$ corresponding to 
  $C{\into} K$. The result is the operation $R{\actto}T{\to} Q$ obtained 
  by first factoring $C{\actto}S$ through $H$, then forming the active-inert
  pushout, and finally using its universal property.
\end{blanko}

\subsection{Non-examples and outlook}

\label{sub:outlook}

The examples in the previous subsection concerned the case where $\PPP$
is itself free. In the general case, $\PPP$ is not required to be free, but
freeness comes in since of course $\PPP\upperstar$ is free.
The following discussion looks beyond the free case, towards more general
Baez--Dolan constructions.

\begin{blanko}{Moment-cumulant relations in free probability.}
  {\em Free probability}, introduced by
  Voiculescu~\cite{Voiculescu:lectures} in the 1980s, is a noncommutative
  analogue of classical probability, originally motivated by operator
  algebras. 
  Speicher~\cite{Speicher:multiplicative} discovered that the combinatorics
  underlying free probability is that of noncrossing partitions,
  contrasting the ordinary partitions in classical probability, and
  established a beautiful cumulant-moment formula for free cumulants in
  terms of M\"obius inversion in the incidence algebra of the noncrossing
  partitions lattice~\cite{Nica-Speicher:Lectures}.
  Ebrahimi-Fard and Patras \cite{EbrahimiFard-Patras:1409.5664},
  \cite{EbrahimiFard-Patras:1502.02748}
  gave a very different approach to the moment-cumulant
  relations, in terms of a time-ordered exponential coming from a
  half-shuffle in the tensor algebra.
\end{blanko}

\begin{blanko}{The comodule bialgebra of noncrossing partitions.}
  The link between the two constructions was found recently by
  Ebrahimi-Fard, Foissy, Kock, and
  Patras~\cite{EbrahimiFard-Foissy-Kock-Patras:1907.01190}, in terms of a
  comodule bialgebra structure on noncrossing partitions. This in turn is
  induced by two different operad structures on noncrossing partitions: the
  {\em gap-insertion operad} structure on noncrossing partitions works like
  this:
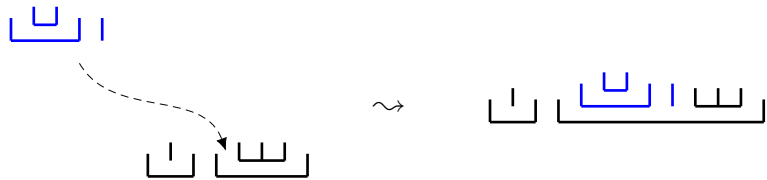
\begin{center}
  \def\h{0.5}
  \begin{tikzpicture}[line width=1.0pt,scale=0.60]
	\begin{scope}[shift={(0.0, 3.0)}, blue]
	  \draw  (0.0,0.0) -- +(1.5,0);
		\draw  (0.0,0.0) -- +(0,\h);
		\draw  (1.5,0.0) -- +(0,\h);
		\draw  (2.0,0.0) -- +(0,\h);
	  \draw  (0.5,0.35) -- +(0.5,0);
		\draw  (0.5,0.35) -- +(0,0.8*\h);
		\draw  (1.0,0.35) -- +(0,0.8*\h);
	  \coordinate (start) at (1.5, -0.5);
	\end{scope}

	\begin{scope}[shift={(3.0, 0.0)}]	
	  \coordinate (insertionpoint) at (1.7, 0.6);
	  \draw  (0.0,0.0) -- +(1.0,0);
		\draw  (0.0,0.0) -- +(0,\h);
		\draw  (1.0,0.0) -- +(0,\h);
		\draw  (0.5,0.35) -- +(0,0.8*\h);

	  \begin{scope}[shift={(1.5, 0.0)}]
		\draw  (0.0,0.0) -- +(2.0,0);
		  \draw  (0.0,0.0) -- +(0,\h);
		  \draw  (2.0,0.0) -- +(0,\h);
		\draw  (0.5,0.35) -- +(1.0,0);
		  \draw  (0.5,0.35) -- +(0,0.8*\h);
		  \draw  (1.0,0.35) -- +(0,0.8*\h);
		  \draw  (1.5,0.35) -- +(0,0.8*\h);
	  \end{scope}
	\end{scope}

	\draw[-{Latex[width=3.5pt,length=4pt]}, line width=0.02, densely dashed] 
	  (start) to[out=-60, in=110] (insertionpoint);
	\node at (8.3, 1.5) {$\leadsto$};

	\begin{scope}[shift={(10.5, 1.2)}]	
	  \draw  (0.0,0.0) -- +(1.0,0);
		\draw  (0.0,0.0) -- +(0,\h);
		\draw  (1.0,0.0) -- +(0,\h);
		\draw  (0.5,0.35) -- +(0,0.8*\h);

	  \begin{scope}[shift={(1.5, 0.0)}]
		\draw  (0.0,0.0) -- +(4.5,0);
		  \draw  (0.0,0.0) -- +(0,\h);
		  \draw  (4.5,0.0) -- +(0,\h);
		\draw  (3.0,0.35) -- +(1.0,0);
		  \draw  (3.0,0.35) -- +(0,0.8*\h);
		  \draw  (3.5,0.35) -- +(0,0.8*\h);
		  \draw  (4.0,0.35) -- +(0,0.8*\h);
	  \end{scope}
	  
	  \begin{scope}[shift={(2.0, 0.35)}, blue]
		\draw  (0.0,0.0) -- +(1.5,0);
		  \draw  (0.0,0.0) -- +(0,\h);
		  \draw  (1.5,0.0) -- +(0,\h);
		\draw  (2.0,0.0) -- +(0,\h);
		\draw  (0.5,0.35) -- +(0.5,0);
		  \draw  (0.5,0.35) -- +(0,0.8*\h);
		  \draw  (1.0,0.35) -- +(0,0.8*\h);
	  \end{scope}
	\end{scope}

  \end{tikzpicture}
\end{center}
The {\em block-substitution operad} structure on noncrossing partitions
works like this:
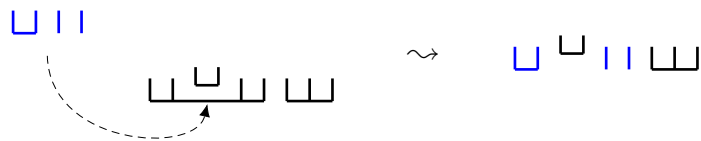
\begin{center}
  \def\h{0.5}
  \begin{tikzpicture}[line width=1.0pt,scale=0.60]
	\begin{scope}[shift={(0.0, 2.0)}, blue]
	  \draw (0.0,0.0) -- +(0.5,0);
	  \draw (0.0,0.0) -- +(0,\h);
	  \draw (0.5,0.0) -- +(0,\h);
	  \draw (1.0,0.0) -- +(0,\h);
	  \draw (1.5,0.0) -- +(0,\h);
	  \coordinate (start) at (0.75, -0.5);
	\end{scope}

	\begin{scope}[shift={(3.0, 0.5)}]	
	  \coordinate (insertionpoint) at (1.25, -0.1);
	  \draw  (0.0,0.0) -- +(2.5,0);
		\draw  (0.0,0.0) -- +(0,\h);
		\draw  (0.5,0.0) -- +(0,\h);
		\draw  (2.0,0.0) -- +(0,\h);
		\draw  (2.5,0.0) -- +(0,\h);
	  \draw  (3.0,0.0) -- +(1.0,0);
		\draw  (3.0,0.0) -- +(0,\h);
		\draw  (3.5,0.0) -- +(0,\h);
		\draw  (4.0,0.0) -- +(0,\h);
	  \draw  (1.0,0.35) -- +(0.5,0);
		\draw  (1.0,0.35) -- +(0,0.8*\h);
		\draw  (1.5,0.35) -- +(0,0.8*\h);
	\end{scope}

	\draw[-{Latex[width=3.5pt,length=4pt]}, line width=0.02, densely dashed] 
	  (start) .. controls (0.8, -0.5) and (4.0, -0.7) .. (insertionpoint);
	\node at (9.0, 1.5) {$\leadsto$};

	\begin{scope}[shift={(11.0, 1.2)}]	
	  \draw[blue]  (0.0,0.0) -- +(0.5,0);
		\draw[blue]  (0.0,0.0) -- +(0,\h);
		\draw[blue]  (0.5,0.0) -- +(0,\h);
		\draw[blue]  (2.0,0.0) -- +(0,\h);
		\draw[blue]  (2.5,0.0) -- +(0,\h);
	  \draw  (3.0,0.0) -- +(1.0,0);
		\draw  (3.0,0.0) -- +(0,\h);
		\draw  (3.5,0.0) -- +(0,\h);
		\draw  (4.0,0.0) -- +(0,\h);
	  \draw  (1.0,0.35) -- +(0.5,0);
		\draw  (1.0,0.35) -- +(0,0.8*\h);
		\draw  (1.5,0.35) -- +(0,0.8*\h);
	\end{scope}

  \end{tikzpicture}
\end{center}
The incidence bialgebras of the two operads together form a comodule 
bialgebra~\cite{EbrahimiFard-Foissy-Kock-Patras:1907.01190}.
\end{blanko}

\begin{blanko}{Balanced Baez--Dolan construction (tentative).}
  The comodule bialgebra of noncrossing partitions cannot result directly
  from a Baez--Dolan construction, since the gap-insertion operad is not
  free. However, it is `not too far' from being
  free~\cite[Prop.~3.1.4]{EbrahimiFard-Foissy-Kock-Patras:1907.01190}; precisely it 
  satisfies the equation of the following definition.

  Define a {\em balanced operad} to be a nonsymmetric operad
  satisfying the following equation: substituting operation $a$ into
  the first slot of operation $b$ equals substituting $b$ into the
  last slot of $a$:
  
  \begin{center}
\begin{tikzpicture}
  \footnotesize
  
  \begin{scope}[shift={(0.0,0.0)}]
    \coordinate (L) at (-0.33, 0.6);
	\node at (-0.5, 0.5) {$a$};
    \coordinate (R) at (0.0, 0.0);
	\node at (0.25, 0.0) {$b$};
	\draw (0.0, -0.4) -- (R) pic {onedot} -- (L) pic {onedot};
	\draw (L) -- + (-0.6, 0.6);
	\draw (L) -- + (-0.2, 0.6);
	\draw (L) -- + (0.6, 0.6);
	\draw (L) -- + (0.2, 0.6);
	\draw (R) -- + (0.0, 0.6);
	\draw (R) -- + (0.33, 0.6);
  \end{scope}

  \node at (1.5, 0.3) {$=$};
  
  \begin{scope}[shift={(3.0,0.0)}]
    \coordinate (L) at (0.0, 0.0);
	\node at (-0.2, -0.1) {$a$};
    \coordinate (R) at (0.6, 0.6);
	\node at (0.85, 0.6) {$b$};
	\draw (0.0, -0.4) -- (L) pic {onedot} -- (R) pic {onedot};
	\draw (L) -- + (-0.6, 0.6);
	\draw (L) -- + (-0.2, 0.6);
	\draw (L) -- + (0.2, 0.6);
	\draw (R) -- + (-0.33, 0.6);
	\draw (R) -- + (0.0, 0.6);
	\draw (R) -- + (0.33, 0.6);
  \end{scope}

  \end{tikzpicture}
  \end{center}
  (For ternary operations, this is the equation for the algebraic theory of
  generalised pseudo-heaps (non-Mal'cev heaps), studied by Wagner in the
  1950s in differential geometry.)

  We now claim that the forgetful functor from balanced operads to endofunctors
  cartesian over the free-monoid monad $\MMM$ admits a left adjoint, the
  free-balanced-operad functor.

  {\em Example:} the free balanced operad on the terminal (reduced) nonsymmetric
  operad $\overline\MMM$ should be the gap-insertion operad of
  noncrossing partitions.

  Now one should essentially just modify the Baez--Dolan construction to
  refer to this putative free-balanced-operad monad.

  {\em Example:} the balanced Baez--Dolan construction on the terminal (reduced)
  nonsymmetric operad $\overline\MMM$ should be the operad of noncrossing
  partitions with block substitution.
    
  Assuming these claims and constructions work out, it will follow that the
  comodule bialgebra of noncrossing partitions of 
  Ebrahimi-Fard--Foissy--Kock--Patras~\cite{EbrahimiFard-Foissy-Kock-Patras:1907.01190}
  is the incidence comodule bialgebra of the balanced Baez--Dolan 
  construction on $\overline\MMM$. 
  
  Verifying all the details has been postponed for future work, as indeed 
  it would seem worthwhile developing a theory for more general 
  Baez--Dolan constructions.
\end{blanko}

\begin{blanko}{Regularity structures.}
  As mentioned briefly in the introduction, Bruned, Hairer and
  Zambotti~\cite{Bruned-Hairer-Zambotti:1610.08468} have given an algebraic
  approach to renormalisation of regularity structures, where a comodule
  bialgebra plays a key role. The overall shape of the comultiplications
  involved resembles the Calaque--Ebrahimi-Fard--Manchon situation, but the
  tree structures are considerably more complicated (the
  paper~\cite{Bruned-Hairer-Zambotti:1610.08468} contains 40 pages of tree
  combinatorics!), because of intricate decorations required to encode the
  associated analytic objects. Briefly, vertices represent integration
  variables, edges represent integration kernels; there are additional
  numerical decorations, of nodes to represent Taylor remainders, and of
  edges to represent derivatives of the kernels. All these decorations are
  not just dead weight with respect to the combinatorics of the
  comultiplications, but transform in a non-trivial way with the
  contractions and extractions, as required in order to express the
  behaviour of the analytic objects.

One quickly sees that this comodule bialgebra escapes the range of examples
covered by the Baez--Dolan construction in its pure form. For examples, it
is easy to check that the simplicial groupoids defining the 
comultiplications are not Segal spaces, and therefore cannot arise as 
two-sided bar constructions. Another issue is that the sums involved are
not finite, as would be the case in the situation of a free operad.

In spite of these discouragements, it is not unlikely that there are still
relationships to be uncovered, involving passage to the core, and perhaps a
Baez--Dolan construction relative to something fancier than the free-monad
monad. The challenge is to obtain an operadic interpretation of the
decorations.
  
Even if the Baez--Dolan construction turns out not to be useful in this
context, there may still be opportunities for the techniques of the present
paper, and in particular for an objective approach. Firstly, the infinite
sums appearing in the comultiplication formulae, which in the paper are
handled through clever gradings and conditions ensuring that certain
infinite matrices are upper-triangular, suggest that slice categories
could be a natural framework.
Secondly, the formulae
for the comultiplications involve factorial denominators which
transform according to some generalised Chu--Vandermonde identity,
suggesting that these slices should be groupoid slices.

At the moment these considerations are speculative, and at the moment we
list the Bruned--Hairer--Zambotti comodule bialgebra as a non-example, 
calling for further investigation.
\end{blanko}

\bigskip

\footnotesize

\noindent {\bf Acknowledgments.} This work was presented
at the {\em Workshop on comodule bialgebras (GDR Renormalisation)} in
Clermont-Ferrand, November 2018. I wish to thank Dominique Manchon for a
wonderful conference, and for the perfect opportunity for me to expose
this material. I have benefitted much from related
collaborations with Imma G\'alvez, Andy Tonks, Mark Weber, Louis
Carlier, Kurusch Ebrahimi-Fard, Lo\"ic Foissy, and Fr\'ed\'eric Patras, all
of whom I thank for their influence on various parts of this work.
Thanks are due also to Marcelo Fiore, Ander Murua, Pierre-Louis Curien, Paul-Andr\'e Melli\`es,
Andr\'e Joyal,
Gabriella B\"ohm, and Birgit Richter, for input and feedback, and to the 
anonymous referees for catching a couple of small mistakes and for other 
pertinent remarks.
Support from grants MTM2016-80439-P (AEI/FEDER, UE) of Spain and
2017-SGR-1725 of Catalonia is gratefully acknowledged.


\hyphenation{mathe-matisk}

\end{document}